\documentclass[a4paper, 11pt]{article}

\usepackage{amsmath}
\usepackage{amssymb}
\usepackage{cite}
\usepackage{makeidx}
\usepackage{graphicx}
\usepackage{multirow}
\usepackage{booktabs}
\usepackage{appendix}
\usepackage{framed}
\usepackage[normalem]{ulem}
\usepackage{amsthm}
\usepackage{subfig}
\usepackage{diagbox}
\usepackage{algorithm}
\usepackage{color}
\usepackage{bigstrut}
\usepackage{algpseudocode}
\usepackage{amsfonts}
\usepackage[colorlinks,linkcolor=blue,anchorcolor=blue,citecolor=red]{hyperref}
\usepackage{enumerate}
\numberwithin{equation}{section}
\newcommand{\UP[1]}{^{(#1)}}
\newcommand\pd[2]{\frac{\partial {#1}}{\partial {#2}}}
\newcommand\pa{\partial}
\newcommand\rd{\mathrm{d}}
\newcommand\dd{\mathrm{d}}
\newcommand\bu{\boldsymbol{u}}
\newcommand\bg{\boldsymbol{g}}
\newcommand\bh{\boldsymbol{h}}
\newcommand\bm{\boldsymbol{m}}
\newcommand\bn{\boldsymbol{n}}
\newcommand\bi{\boldsymbol{i}}
\newcommand\bj{\boldsymbol{j}}

\newcommand\bq{\boldsymbol{q}}
\newcommand\bQ{\boldsymbol{Q}}
\newcommand\bv{\boldsymbol{v}}
\newcommand\bx{\boldsymbol{x}}
\newcommand\bw{\boldsymbol{w}}
\renewcommand\bf{\boldsymbol{f}}
\newcommand\bF{\boldsymbol{F}}
\newcommand\bsigma{\boldsymbol{\sigma}}
\newcommand\bT{\overline{\theta}}
\newcommand\oT{\overline{\theta}}
\newcommand\ou{\overline{\bu}}
\newcommand\ot{\overline{\te}}
\newcommand{\bz}{\boldsymbol{0}}
\newcommand\mF{\mathbb{F}}
\newcommand\bbH{\mathbb{H}}
\newcommand\mO{\mathcal{O}}
\newcommand\mQ{\mathcal{Q}}
\newcommand\mR{\mathbb{R}}
\newcommand\mT{\mathcal{T}}
\newcommand\bbR{\mathbb{R}}
\newcommand\mM{\mathcal{M}}
\newcommand\mH{\mathcal{H}}
\newcommand\mN{\mathbb{N}}
\newcommand\bbN{\mathbb{N}}
\newcommand\mZ{\mathbb{Z}}

\newcommand\sr{\sqrt{r}}
\newcommand\al{\alpha}
\newcommand\be{\beta}
\newcommand{\eps}{\varepsilon}

\newcommand\te{\theta}
\newcommand\ze{\zeta}

\newcommand\aut{_{\al}^{\ou,\oT}}

\newcommand\but{_{\be}^{\ou,\oT}}

\newcommand\Ha{H_{\al}}
\newcommand\Hai{H_{\al}\UP[i]}
\newcommand\Hl{H_{\la}}
\newcommand\Hli{H_{\la}\UP[i]}
\newcommand{\Hj}{H_{\bj}}
\newcommand\Hk{H_{\ka}}
\newcommand\Hki{H_{\ka}\UP[i]}
\newcommand\Hkp{H_{\ka'}}
\newcommand\Haj{H_{\al}\UP[j]}
\newcommand\mHai{\mH_{\al}\UP[i]}
\newcommand\mHli{\mH_{\lbd}\UP[i]}
\newcommand\mHkj{\mH_{\kp}\UP[j]}
\newcommand\mMut{\mM_{\ou,\oT}}
\newcommand\ga{\gamma}
\newcommand{\om}{\omega}
\newcommand\et{\eta}
\newcommand\de{\delta}
\newcommand\la{\lambda}
\newcommand\lbd{\lambda}
\newcommand\ka{\kappa}
\newcommand\kp{\kappa}
\newcommand{\Kn}{{\rm Kn}}

\newcommand\alk{_{\alpha,\lambda,\kappa}}
\newcommand\Aalk{A_{\alpha, \lambda, \kappa}^{(ij)}}

\newcommand\pq{\preceq}
\newcommand\ca[1]{a_{l'_{#1}k'_{#1}}^{l_{#1}k_{#1}}}

\newtheorem*{definition}{Definition}
\newtheorem{property}{Property}
\newtheorem{theorem}{Theorem}
\newtheorem{lemma}{Lemma}
\newtheorem{corollary}{Corollary}
\newtheorem{remark}{Remark}

\graphicspath{{images/}}

\title{Hermite spectral method for multi-species Boltzmann equation}

\author{Ruo Li\thanks{CAPT, LMAM \& School of Mathematical Sciences,
    Peking University, Beijing, China, email: {\tt
      rli@math.pku.edu.cn}.},~~Yixiao Lu\thanks{HEDPS, Center for Applied Physics and Technology \& School of Mathematical
    Sciences, Peking University, Beijing, China, 100871, email: {\tt
      luyixiao@pku.edu.cn}.},~~Yanli Wang\thanks{Beijing Computational
    Science Research Center, email: {\tt ylwang@csrc.ac.cn}.},~~Haoxuan Xu\thanks{School of Mathematical
    Sciences, Peking University, Beijing, China, 100871, email: {\tt
      xhxxhx@pku.edu.cn}.}}

\begin{document}
\maketitle
\begin{abstract}
    We introduce a numerical scheme for the full multi-species Boltzmann equation based on Hermite spectral method. With the proper choice of expansion centers for different species, a practical algorithm is derived to evaluate the complicated multi-species binary collision operator. New collision models are built by combining the quadratic collision model and the simple BGK collision model under the framework of the Hermite spectral method, which enables us to balance the computational cost and accuracy.
    Several numerical experiments are implemented to validate the dramatic efficiency of this new Hermite spectral method. Moreover, we can handle the problems with as many as 100 species, which is far beyond the capability of the state-of-art algorithms.
    \vspace*{4mm}
  \end{abstract}

\section{Introduction}
\label{sec:intro}
For rarefied gas flows where the molecular mean free path is comparable to the characteristic length, the continuum-fluid models such as Euler equations and Navier-Stokes equations are no longer accurate. The Boltzmann equation is an important model to study rarefied gas dynamics, which uses the distribution function to govern the dilute gas behavior at the mesoscopic scale. Gas mixtures with species diffusion, which play an essential role in the evaporation/condensation process, could also be described by the multi-species Boltzmann equation. For each species of the gases, it is represented by a distribution function $f^{(i)}(t, \bx, \bv)$, where $t$ is time, $\bx$ is the physical space, $\bv$ is the particle velocity space, and $(i)$ represents the kind of species. Thus, the multi-species Boltzmann equation, which describes a many-particle system comprised of a mixture of $s$ species, has the form below \cite{Cercignani1988, Harris1971}  
\begin{equation}
    \label{eq:Boltz}
    \pd{f\UP[i]}{t}+\bv\cdot\nabla_{\bx} f\UP[i]=\sum_{j=1}^s\mQ\UP[ij]
		[f\UP[i],f\UP[j]](\bv),\quad i=1,2,\cdots,s,
\end{equation}
where $\mQ\UP[ij]$ is the collision operator that models binary collisions
between $i$-th and $j$-th species:
\begin{equation}
    \label{eq:Collision}
    \begin{split}
        &\mQ\UP[ij][f\UP[i],f\UP[j]](\bv)=
		\int_{\mR^3}\int_{S^2}B_{ij}(|\bv-\bv_{\ast}|,\varsigma)
		\left[f\UP[i](\bv')f\UP[j](\bv'_{\ast})-f\UP[i](\bv)
		f\UP[j](\bv_{\ast})\right]\rd\varsigma \rd \bv_{\ast},
    \end{split}
\end{equation}
During the collisions, the momentum and energy are conserved. $(\bv, \bv_{\ast})$ and $(\bv', \bv'_{\ast})$ denote the pre-collision and post-collision velocity pairs satisfying 
\begin{equation}
    \label{eq:post-velo}
    \left\{
	\begin{array}{l}
	\bv'=\frac{m\UP[i]\bv+m\UP[j]\bv_{\ast}}{m\UP[i]+m\UP[j]}
	+\frac{m\UP[j]}{m\UP[i]+m\UP[j]}|\bv-\bv_{\ast}|\varsigma, \\
	\bv'_{\ast}=\frac{m\UP[i]\bv+m\UP[j]\bv_{\ast}}{m\UP[i]+m\UP[j]}
	-\frac{m\UP[i]}{m\UP[i]+m\UP[j]}|\bv-\bv_{\ast}|\varsigma,
	\end{array}
	\right.
\end{equation}
where $\varsigma$ is a unit vector indicating the direction of the post-collision relative velocity, and $m\UP[i],m\UP[j]$ are the mass of particles of $i$-th and $j$-th species respectively. 

The collision kernel $B_{ij}$ in \eqref{eq:Boltz} characterizes the interaction between the particles, which could be determined by the classical scattering theory based on the interaction potential as
\begin{equation}
    \label{eq:B}
    B_{ij}(|\bv -\bv_{\ast}|, \varsigma) =  |\bv - \bv_{\ast}| \frac{b_{ij}}{\sin \chi} \left|\frac{\dd b_{ij}}{\dd \chi}  \right|, \qquad \chi = \arccos\left(\frac{\varsigma \cdot (\bv - \bv_{\ast})}{|\bv - \bv_{\ast}|}\right), 
\end{equation}
with $\chi$ the deviation angle between $\bv- \bv_{\ast}$ and $\bv' - \bv_{\ast}'$, and $b_{ij}$ the impact factor \cite{Shashank2019}. 
For the multi-species Boltzmann equation, the most difficult part lies in the binary collision term due to its complex quadratic form and high-dimensional integral. Several collision models are introduced, such as the variable soft sphere model (VSS) \cite{Koura} by assuming 
\begin{equation}
    \label{eq:vss}
    \chi = 2 \cos^{-1}\left[ (b_{ij} / d_{ij})^{1/\alpha_{ij}}\right]
\end{equation}
where $d_{ij}$ is the diameter and $\alpha_{ij}$ is the scattering parameter. For the diameter $d_{ij}$, we choose the same as in \cite{Shashank2019}, which is borrowed from \cite[Eq. (4.79)]{Bird} as
\begin{equation}
    \label{eq:diame}
    d_{ij} = d_{{\rm ref}, ij} \left[  \left(\frac{2 k_B T_{{\rm ref}, ij}}{\mu_{ij} |\bv - \bv_{\ast}|^2} \right)^{\omega_{ij} - 0.5}\frac{1}{\Gamma(2.5 - \omega_{ij})} \right]^{1/2}
\end{equation}
with $\mu_{ij} = \frac{m\UP[i] m\UP[j]}{m\UP[i] + m\UP[j]}$, $\Gamma$ the Gamma function, $\omega_{ij}$ the viscosity index, and $d_{{\rm ref}, ij}, T_{{\rm ref}, ij}$ are the reference diameter and temperature. Particularly, we could obtain the variable hard sphere model (VHS) when $\alpha_{ij} = 1$ and $0.5 \leqslant \omega_{ij} \leqslant 1$ (especially, the Maxwell molecules and hard sphere model (HS) are corresponding to $\omega = 1$ and $\omega =0.5 $ respectively); and the VSS model when $1 < \alpha_{ij} \leqslant 2$ and $0.5 \leqslant \omega_{ij} \leqslant 1$. It is full of difficulty in numerically dealing with the binary collision term, especially for the VSS model, due to its extremely complicated form of integral.

For now, several numerical methods have been proposed to solve the multi-species Boltzmann equation. By stochastically modelling the binary interactions, the Direct Simulation Monte Carlo (DSMC) method \cite{Bird} is widely used in the simulation of the Boltzmann equation. DSMC method is efficient in the highly-rarefied gas flows but does not work well in the unsteady-state problems due to its stochastical noise. In recent years, research on the deterministic methods has made great progress. Discrete velocity method (DVM) \cite{goldstein1989, Liu2020} is a classical deterministic method. However, low efficiency of DVM has greatly limited its usage \cite{Panferov2002}. Spectral methods such as Fourier spectral method \cite{Pareschi1996, Mouhot, Hu2016} have made great success in the simulation of the Boltzmann equation, and then, it has been extended to the multi-species Boltzmann equations \cite{Shashank2019, Wulei2015}. Another important spectral method is the Hermite spectral method \cite{Gamba2018, Approximation2019}, where the Hermite polynomials are utilized as the basis functions to approximate the distribution function. An asymptotic-preserving scheme for the two-species binary collisional kinetic system was proposed in \cite{Jin2019}.  A comprehensive review of these methods can be referred to \cite{Dimarco2014}. Another important method is the moment method, which was firstly proposed by Grad \cite{Grad} in 1949. However, the non-hyperbolicity (even near the Maxwellian) of the moment equation led by Grad's method, has blocked the its development. Recently, several regularization methods \cite{Fan2013, Koellermeier2017} have been proposed to derive the globally hyperbolic moment systems.

In this paper, we are focusing on the Hermite spectral method. Since the Maxwellian is the steady-state solution of the Boltzmann equation, it is natural to consider expanding the distribution function with Hermite polynomials, whose weight function is just the Maxwellian. Another advantage of the Hermite spectral method is that several macroscopic variables such as the density, velocity, temperature, shear stress, and heat flux could be expressed by the expansion coefficients of the first few orders under the framework of the Hermite expansions. Recently, an algorithm to reduce the computational cost of approximating the binary collision term is proposed in \cite{Approximation2019}, which is achieved by exactly calculating the expansion coefficients of the quadratic collision model and building new collision models through combining the quadratic collision term and the simplfied collision models such as the BGK model. This method is verified to make great success in the simulation of rarefied gas \cite{ZhichengHu2019}, and is soon extended to the field of the collisional plasma \cite{FPL2018, FPL2020}.

Following similar lines, we develop the Hermite spectral method for the multi-species Boltzmann equation. Distribution functions of each species are expanded with the Hermite basis functions, with the expansion center in the basis function chosen differently for different species. We first simplify the complicated quadratic collision models, where the expansion centers are decided by the macroscopic velocity, the average temperature, and the mass of each species. Following the techniques in \cite{Approximation2019, FPL2018}, we will derive the algorithm to calculate the expansion coefficients of the binary collision models under the framework of the Hermite expansions. The different expansion centers for different species will make this calculation much easier. Moreover, the choice of the expansion centers makes the expansion coefficients of the collision models only depend on the mass ratio of the two particles, which also could be pre-computed and only need to be computed once. 
To reduce the computational cost of this complicated binary collision model, new collision models are built by combining the quadratic collision models and the simple multi-species BGK model \cite{Klingenberg2017}, which enables us to obtain a high-order approximation with low computational cost and this advantage is also verified in the numerical experiments. The time-splitting scheme is utilized to deal with the convection term and the collision term separately, where the convection term is solved similarly as in \cite{ZhichengHu2019, framework} with the standard finite volume method. With the new collision model in the collision step, the total computational cost to approximate one collision term  is greatly reduced, which is about $\mathcal{O}(M_0^9+M^3)$, where $M_0$ and $M$ are the length of the binary collision part and the order of Hermite expansion, respectively. 

In the numerical experiments, the test for the exact Krook-Wu solution \cite{krook1977exact} validates the correctness of this new numerical method. The simulations are implemented for one-dimensional benchmark problems, including Couette flow and Fourier heat transfer, and the two-dimensional lid-driven cavity flow. To further verify the accuracy and efficiency of this Hermite spectral method, we finally implement a 100-species Krook-Wu solution with tolerable costs of time and memory space, where the number of species can be hardly handled by most of the methods nowadays.

The rest of this paper is organized as follows. In Sec. \ref{sec:Prelimaries}, the multi-species Boltzmann equation and the Hermite spectral method are introduced. The numerical method to calculate the expansion coefficients of the quadratic collision term in the framework of the Hermite spectral method is shown in Sec. \ref{sec:collision}, and the algorithm to build the new collision model is also presented in this section. The discussion of the moment equations and the description of the whole numerical scheme is given in Sec. \ref{sec:numerical}. The numerical experiments are presented in Sec. \ref{sec:experiment}, with some concluding remarks in Sec. \ref{sec:conclusion}. Some supplementary contents are presented in App. \ref{sec:app}.

\section{Preliminaries}

\label{sec:Prelimaries}
For easier computation, we follow the similar nodimensionalization as in \cite{Shashank2019} to scale the variables  as 
\begin{equation}
    \label{eq:nondim}
    \hat{\bx} = \frac{\bx}{x_0},  \qquad \hat{\bv} = \frac{\bv}{u_0}, \qquad     \hat{t}=\frac{t}{x_0/u_0}, \qquad \hat{m}\UP[i] = \frac{m\UP[i]}{m_0}, \qquad \hat{f}\UP[i] = \frac{f\UP[i]}{n_0 / u_0^3},
\end{equation}
where $x_0, n_0, T_0$ and $m_0$ are the characteristic length, number density, temperature and mass, with $u_0$ the character velocity defined as $u_0 = \sqrt{k_B T_0 / m_0}$. 
\subsection{Multi-species Boltzmann equation}
Substituting the nondimensionalization into the Boltzmann equation \eqref{eq:Boltz}, we derive the dimensionless Boltzmann equation as 
(dropping $\hat{}$ for simplicity)
\begin{equation}
    \label{eq:nondim_Bol}
        \pd{f\UP[i]}{t}+\bv\cdot\nabla_{\bx} f\UP[i]=\sum_{j=1}^s
        \frac{1}{\Kn_{ij}}
        \mQ\UP[ij][f\UP[i],f\UP[j]],
  \end{equation}
where the collision term $\mQ\UP[ij][f\UP[i],f\UP[j]]$ is simplified as 
  \begin{equation}
      \label{eq:collision}
      \mQ\UP[ij]=\int_{\mR^3}\int_{S^2}B_{ij}(|\bv-\bv_{\ast}|,\varsigma)
        \left[f\UP[i](\bv')f\UP[j](\bv'_{\ast})-f\UP[i](\bv)
        f\UP[j](\bv_{\ast})\right]\rd\varsigma \rd \bv_{\ast}.
  \end{equation}
 and $\Kn_{ij}$ is the Knudsen number, which is defined as the ratio of the mean free path and characteristic length. 
 For the VSS model, the dimensionless collision kernel $B_{ij}$ has the form 
  \begin{equation}
    \label{eq:colKer}
    B_{ij}=\frac{2^{\om_{ij}-1/2}\alpha_{ij}}
        {\sqrt{1+\frac{m\UP[i]}{m\UP[j]}}
        \mu_{ij}^{\om_{ij}-1/2}
        \Gamma(\frac52-\omega_{ij})2^{\al_{ij}+1}\pi}|\bv-\bv_{\ast}|^
        {2(1-\om_{ij})}(1+\cos\chi)^{\al_{ij}-1}.
  \end{equation}
 From the nondimensionalization, the Knudsen number can be computed as
 \begin{equation}
     \label{eq:Kn}
     \Kn_{ij} = \frac{1}{\sqrt{1 + m\UP[i] / m\UP[j]} \pi n_0 d^2_{{\rm ref}, ij} (T_{{\rm ref}, ij}/ T_0)^{\omega_{ij} - 0.5} x_0}.
 \end{equation}
 The macroscopic variables such as the number density $n\UP[i]$, density $\rho\UP[i]$, macroscopic velocity $\bu\UP[i]$ and temperature $T\UP[i]$ for each species could be derived from the distribution function as 
 \begin{equation}
     \label{eq:macro}
     \begin{aligned}
     &n\UP[i](t,\bx)=\int_{\mR^3}f\UP[i](t,\bx,\bv)\rd\bv, \qquad 
    \rho\UP[i](t,\bx) = m\UP[i]\int_{\mR^3}f\UP[i](t,\bx,\bv)\rd\bv, \\
    &\bu\UP[i](t,\bx)=\frac{1}{n\UP[i]}\int_{\mR^3}\bv f\UP[i](t,\bx,\bv)\rd\bv, \qquad 
    T\UP[i](t,\bx)=\frac{m\UP[i]}{3n\UP[i]}\int_{\mR^3}|\bv-\bu\UP[i]|^2f\UP[i](t,\bx,\bv)
    \rd\bv. 
     \end{aligned}
 \end{equation}
 The stress tensor $\bsigma\UP[i]$, heat flux $\bq\UP[i]$ and total energy $E\UP[i]$ could also be derived from the distribution function as
 \begin{equation}
     \label{eq:sigma_q}
     \begin{aligned}
   &\bsigma\UP[i](t,\bx)=m\UP[i]\int_{\mR^3}\left[(\bv-\bu\UP[i])\otimes(\bv-\bu\UP[i])-\frac13|\bv-\bu\UP[i]|^2I\right] f\UP[i](t,\bx,\bv)\rd\bv,  \\
&\bq\UP[i](t,\bx) = \frac12 m\UP[i]\int_{\mR^3}(\bv-\bu\UP[i])|\bv-\bu\UP[i]|^2 f\UP[i](t,\bx,\bv)\rd\bv, \\
&E\UP[i]=\frac12 m\UP[i]\int_{\mR^3} f\UP[i]\bv^2\rd\bv=\frac12\rho\UP[i]|\bu\UP[i]|^2+\frac32n\UP[i]T\UP[i].
 \end{aligned}
\end{equation}
The macroscopic variables for the whole system are defined as 
\begin{equation}
    \label{eq:macro_tot}
    \begin{split}
        &n = \sum_{i = 1}^s n\UP[i], \qquad \rho = \sum_{i=1}^s \rho\UP[i], \qquad  \rho\bu = \sum_{i = 1}^s \rho\UP[i] \bu\UP[i],
        \qquad n T=\sum_{i = 1}^s  n\UP[i]T\UP[i], \\
        &\bsigma = \sum_{i = 1}^s \bsigma\UP[i], \qquad \bq = \sum_{i = 1}^s\bq \UP[i], \qquad E=\sum_{i=1}^s E\UP[i].
    \end{split}
\end{equation}
Define $\theta\UP[i] = T\UP[i] / m\UP[i]$, then the steady-state solution to the multi-species Boltzmann equation \eqref{eq:nondim_Bol} has the form as \cite{Shashank2019}
\begin{equation}
    \label{eq:local_max}
        \bu\UP[i]=\bu,\qquad T\UP[i]=T, \qquad f\UP[i]=n\UP[i]\mM_{\bu\UP[i],\theta\UP[i]}(\bv), \qquad i=1,2,\cdots,s,
\end{equation}
with 
\begin{equation}
    \label{eq:maxwellian}
   \mM_{\bu, T}(\bv) =  \frac{1}{(2\pi T)^{\frac32}}
    \exp\left(-\frac{|\bv-\bu|^2}{2 T}\right),
\end{equation}
which is named as the local Maxwellian. Here, we define the total average temperature $\mT$ as 
\begin{equation}
    \label{eq:barT}
    \frac{3}{2} n \mT + \frac{1}{2}\rho |\bu|^2 = E =  \sum_{i=1}^s \frac12 m\UP[i] \int_{\bbR^3} |\bv|^2 f\UP[i] \dd \bv,
\end{equation}
which could express the total average temperature of the system. We also want to emphasize that due to the nonlinearity of the total energy respected to the macroscopic velocity $\bu$, $\mT$ is different from $T$ in \eqref{eq:macro_tot}. Based on the conservation of mass, momentum, and energy during the collision, the properties hold for the collision term as 
\begin{equation}
    \label{eq:conserv_Q}
    \begin{split}
        &\int_{\mR^3}\mQ\UP[ij](\bv)\rd\bv=0, \\
        &\int_{\mR^3}\mQ\UP[ij](\bv)m\UP[i]\bv\rd\bv+
       \int_{\mR^3}\mQ\UP[ji](\bv_{\ast})m\UP[j]\bv_{\ast}\rd\bv_{\ast} =0, \\
        &\int_{\mR^3}\mQ\UP[ij](\bv)m\UP[i]|\bv|^2\rd\bv+
       \int_{\mR^3}\mQ\UP[ji](\bv_{\ast})m\UP[j]|\bv_{\ast}|^2\rd\bv_{\ast} =0.
    \end{split}
\end{equation}

Due to the complex form of the quadratic collision term, several simplified collision models are introduced, such as the BGK collision model \cite{Klingenberg2017}
\begin{equation}
    \label{eq:BGK}
     \mQ\UP[ij][f\UP[i],f\UP[j]](\bv)=\nu\UP[ij]n\UP[j]\big[f\UP[i]-
     n\UP[i]\mM_{\bu\UP[ij],\frac{T\UP[ij]}{m\UP[i]}}(\bv)\big],
\end{equation}
where $\nu\UP[{ij}]$ is the collision frequency, $\bu\UP[{ij}]$ and $T\UP[{ij}]$ are decided by the macroscopic variables of species $i$ and $j$, 
which should satisfy the conversation relationship \eqref{eq:conserv_Q}. Here, we want to mention if it holds that 
\begin{equation}
    \label{eq:BGK_u}
    \bu\UP[i] = \bu\UP[j] = \bu, \qquad T\UP[i] = T\UP[j]= T,
\end{equation}
then $\bu\UP[ij] = \bu\UP[ji] = \bu$, and $T\UP[ij] = T\UP[ji] = T$. 
Once $\bu\UP[ij]$ is decided, $\bu\UP[ji]$ could be calculated with the conservation law \eqref{eq:conserv_Q}, so as $T\UP[ij]$ and $T\UP[ji]$. 
There are several types of BGK models to decide $\bu\UP[ij], \bu\UP[ji]$ and $T\UP[ij], T\UP[ji]$, such as the work in  \cite{Klingenberg2017, Haack2021}, which we will not discuss in detail.


For now, we have introduced the properties of the multi-species Boltzmann equation. Due to its high dimensionality and the complex collision terms, several numerical methods have been proposed, such as the stochastic DSMC method \cite{Bird}, the discrete velocity method \cite{goldstein1989, Panferov2002}, and the Fourier spectral method \cite{Shashank2019,Wulei2015}. So far, the Hermite spectral method has been successfully utilized to solve Boltzmann equation with single species and the plasma with collisions \cite{FPL2020, ZhichengHu2019}. In this work, we will extend the methodology therein to solve the multi-species Boltzmann equation, which we will discuss in detail in the next sections.

\subsection{Hermite spectral method}
In the framework of the Hermite spectral method, the distribution function $f\UP[i]$ is approximated by a series of expansions. Following the method in \cite{ZhichengHu2019}, we choose the weighted Hermite polynomials as the basis function as 

\begin{definition}[Hermite Polynomials]
\label{def:Her}
For $\al=(\al_1,\al_2,\al_3)$, the three-dimensional 
Hermite polynomial $H_{\al}^{\ou,\oT}(\bv)$ is defined as
\begin{equation}
    \label{eq:Hermite}
    H_{\al}^{\ou,\oT}(\bv)=\frac{(-1)^{|\al|}\oT^{\frac{|\al|}{2}}}
    {\mM_{\ou,\oT}(\bv)}
    \dfrac{\partial^{|\al|}}{\partial \bv^{\alpha}}
    \mM_{\ou,\oT}(\bv), 
\end{equation}
with $|\al|=\al_1+\al_2+\al_3$ and $\partial \bv^{\alpha} = \pa v_1^{\al_1}\pa v_2^{\al_2}\pa v_3^{\al_3}$. Here $\ou \in \bbR^3$ and $\oT \in \bbR^+$ are two parameters which are chosen problem-dependently. $\mM_{\ou, \oT}(\bv)$ is the Maxwellian defined in  \eqref{eq:maxwellian}. The Hermite polynomials have several useful properties when approximating the complex collision term, which are listed in Appendix \ref{app:Her}.
\end{definition}

With the Hermite polynomials, the distribution function $f\UP[i]$ is expanded as 
\begin{equation}
    \label{eq:Her-expan}
    f\UP[i](t,\bx,\bv)=\sum_{\al\in\mN^3}f_{\al}\UP[i](t,\bx)
		\mH\aut(\bv),
\end{equation}
where $\mH\aut(\bv)=H\aut(\bv)\mMut(\bv)$ are the basis functions, $[\ou, \oT] \in \bbR^3\times\bbR^+$ is the problem-dependent expansion center. 
One of the advantages of this expansion is that the problem-dependent expansion center could also be interpreted as the "moments" of the distribution function. With a properly chosen expansion center, the computational cost may be greatly reduced. More details will be revealed in the successive sections. 
By the orthogonality of the basis function \eqref{eq:orth}, it holds that \begin{equation}
    \label{eq:falpha}
    f_{\al}\UP[i](t,\bx)=\frac{1}{\al!}\int_{\mR^3}f\UP[i](t,\bx,\bv) H\aut(\bv)
    \rd\bv.
\end{equation}
The macroscopic variables \eqref{eq:macro} could also be expressed using $f_{\al}\UP[i]$ as  
\begin{equation}
    \label{eq:macro_f}
    \begin{split}
    &n\UP[i] = f_{0}\UP[i],\quad u\UP[i]_k= \overline{u}_k+\frac{\sqrt{\oT}}{n\UP[i]}f\UP[i]_{e_k}, 
    \quad T\UP[i]=\frac{2m\UP[i]\oT}{3n\UP[i]}\sum_{k=1}^3f\UP[i]_{2e_k}+m\UP[i]\oT, \\
    &\sigma\UP[i]_{kl}=(1+\delta_{kl})m\UP[i]\oT f\UP[i]_{e_i+e_j}+\delta_{kl}\rho\UP[i]\left(\oT - \frac{T\UP[i]}{m\UP[i]}\right)-
    \rho\UP[i]\left(\overline{u}_k-u\UP[i]_k\right)\left(\overline{u}_l-u\UP[i]_l\right), \\
    &q\UP[i]_k=2m\UP[i]\oT^{\frac32}f\UP[i]_{3e_k}+(\overline{u}_k-u\UP[i]_k)\oT f\UP[i]_{2e_k}+|\ou-\bu\UP[i]|^2\sqrt{\oT} f\UP[i]_{e_k} \\
    &\qquad \qquad +\sum_{l=1}^3\left[\oT^{\frac32}f\UP[i]_{2e_l+e_k}+\left(\overline{u}_l-u\UP[i]_l\right)\oT f\UP[i]_{e_l+e_k}+
    \left(\overline{u}_k-u\UP[i]_k\right)\oT f\UP[i]_{2e_l}\right].
    \end{split}
\end{equation}

Sometimes, the expansion center is chosen according to the numerical scheme, such as dealing with the force term when simulating the plasma with collision \cite{FPL2020}, or according to the average temperature in the whole space for the Fourier flow problem \cite{ZhichengHu2019}. In this work, this expansion center is also chosen differently for the convection part and the collision term.

\section{Modelling the quadratic collision term with Hermite spectral method}
\label{sec:collision}
In this section, we will model the quadratic collision term using the Hermite spectral method. The influence of the convection term will be ignored temporarily, and we only discuss the algorithm to approximate the collision term $\mQ\UP[ij]$ in \eqref{eq:nondim_Bol}. 

\subsection{Hermite expansion of the collision term}
\label{sec:Her_col}
For the homogeneous Boltzmann equation, the total momentum and the kinetic energy are conserved. With a proper frame of reference, we can suppose that the total macroscopic velocity $\bu$ in \eqref{eq:macro_tot} is $\bu = \bz$. With \eqref{eq:barT}, the total average temperature $\mT$ is constant and solely decided by the total energy $E$. Therefore, we choose the macroscopic velocity and total average temperature to decide the expansion center. For the $i-$th species, the expansion center is set as $[\ou, \oT]\UP[i] = [\bu, \zeta\UP[i]]$ with $\bu = \bz$ and $\zeta\UP[i] = \mT / m\UP[i]$ to approximate the collision term. 

Thus, the basis function of $f\UP[i]$ is chosen as $\mHai=\mH_{\al}^{\bz,\ze\UP[i]}$, and the variable $\bx$ is also omitted in the distribution function for simplicity. The Hermite expansion \eqref{eq:Her-expan} is then written as
\begin{equation}
    \label{eq:Her-expan-coll}
    f\UP[i](t,\bv)=\sum_{\al\in\mN^3}f_{\al}\UP[i](t)\mHai(\bv).
\end{equation}
In what follows, we denote $\Hai$ and $\mM\UP[i]$ as the Hermite polynomial and Maxwellian corresponding to $\mHai$. Substituting \eqref{eq:Her-expan-coll} in the the collision term \eqref{eq:collision}, the collision $\mQ\UP[ij]$ could be expanded as 
\begin{equation}
    \label{eq:exp_col}
    \mQ\UP[ij][f\UP[i],f\UP[j]](\bv) = \sum_{\alpha \in \mN^3} Q\UP[ij]_{\alpha}(\bz, \zeta\UP[i]) \mHai(\bv). \end{equation}
Now we consider how to obtain $Q\UP[ij]_{\al}(\bz, \zeta\UP[i])$ for fixed $i$ and $j$. From the orthogonality \eqref{eq:orth}, we can obtain that 
\begin{equation}
    \label{eq:Qij}
		Q\UP[ij]_{\al}(\bz, \zeta\UP[i])=\frac{1}{\al!}\int_{\mR^3}\Hai(\bv)\mQ\UP[ij]
		[f\UP[i],f\UP[j]](\bv)\rd\bv
		=\sum_{\lbd\in\mN^3}\sum_{\kp\in\mN^3}\Aalk(\bz, \zeta\UP[i]) f_{\la}
		\UP[i]f_{\ka}\UP[j],
\end{equation}
with $\al!=\al_1!\al_2!\al_3!$. Here, the second equality can be obtained by inserting \eqref{eq:Her-expan-coll} into \eqref{eq:collision}, and
\begin{equation}
    \label{eq:Aalk}
	\Aalk(\bz, \zeta\UP[i]) = \frac{1}{\al!}\int_{\mR^3}\int_{\mR^3}\int_{S^2}B_{ij}
	(|\bv-\bv_{\ast}|,\varsigma)
	\Bigg[\mHli(\bv')\mHkj(\bv'_{\ast}) 
	-\mHli(\bv)\mHkj(\bv_{\ast})\Bigg]\Hai(\bv)\rd\varsigma \rd\bv \rd\bv_{\ast},
\end{equation}
where $(\bv',\bv'_{\ast})$ is the post-collision velocity pair determined by \eqref{eq:post-velo}. Here, $Q\UP[ij]_{\al}(\bz, \zeta\UP[i])$ and $\Aalk(\bz, \zeta\UP[i])$ are functions related to the expansion center $[\bz, \zeta\UP[i]]$. For simplicity, from now on, the coefficients $Q\UP[ij]_{\al}(\bz, \zeta\UP[i])$ and $\Aalk(\bz, \zeta\UP[i])$ will be shortened as $Q\UP[ij]_{\al}$ and $\Aalk$ if the expansion center is set as $[\bz, \zeta\UP[i]]$; and this parameter will be explicitly written out if expansion center other than $[\bz, \zeta\UP[i]]$ is utilized. 

In the computation of the expansion coefficients $\Aalk$, it has a complicated high-dimensional integral. Moreover, for the multi-species particles, this $\Aalk$ should be calculated for each pair of $(ij)$, which has greatly increased the computational cost compared to the single species case. Here, the properties of the Hermite polynomials and the conservation of the collision model are all utilized to compute $\Aalk$.

Denoting $r\UP[ij]=\frac{m\UP[j]}{m\UP[i]}=\frac{\zeta\UP[i]}{\zeta\UP[j]}$ as the mass ratio, with the transitivity of Hermite polynomials \eqref{eq:tran} and the definition of Maxwellian \eqref{eq:maxwellian}, it holds that 
\begin{equation}
    \label{eq:tran2}
    \Haj(\bv)=\Hai(\sqrt{r}\bv),\qquad \mM\UP[j](\bv)={r^{\frac32}}
    \mM\UP[i](\sqrt{r}\bv),
\end{equation}
where we omit the superscript of $r\UP[ij]$ and short it as $r$ for simplicity. 
Besides, from the conversation of energy, it also holds that
\begin{equation}
    \label{eq:con_energy}
|\bv|^2+r|\bv_{\ast}|^2=|\bv'|^2+r|\bv'_{\ast}|^2.
\end{equation}
Based on \eqref{eq:maxwellian}, \eqref{eq:tran2} and \eqref{eq:con_energy}, we can derive that 
\begin{equation}
    \label{eq:Max}
    \mM\UP[i](\bv')\mM\UP[j](\bv'_{\ast})
    =\mM\UP[i](\bv)\mM\UP[j](\bv_{\ast}).
\end{equation}
Then, \eqref{eq:Aalk} can be simplified as
\begin{equation}
    \label{eq:Aalk2}
    \begin{split}
        \Aalk=&\frac{r^{\frac32}}{\al!}\int_{\mR^3}\int_{\mR^3}\int_{S^2}
		B_{ij}(|\bv-\bv_{\ast}|,\varsigma)
		\Bigg[\Hli(\bv')\Hki(\sr \bv'_{\ast}) \\
		&\qquad -\Hli(\bv)\Hki(\sr \bv_{\ast})\Bigg]\Hai(\bv)\mM\UP[i](\bv)
		\mM\UP[i](\sr \bv_{\ast})\rd\varsigma \rd\bv \rd\bv_{\ast}.
    \end{split}
\end{equation}
With the properties of the basis functions, the calculation of $\Aalk$ could be greatly simplified, and the final result is listed in the theorem below. 
\begin{theorem} \label{theorem:step1}
The coefficients $\Aalk$ could be simplified as the form below:
\begin{equation}
    \label{eq:theorem_A}
    \Aalk=\frac{r^{\frac32}}{(1+r)^{\frac{|\al|+3}{2}}}
    \sum_{\la'\pq\al,\la'\pq\ka+\la}\frac{r^{\frac{|\bj|}{2}}}{\bj!}
    \ca1(r)\ca2(r)\ca3(r)\ga_{\ka'}^{\bj}(r,\zeta\UP[i]),
\end{equation}
where $\ka'=\ka+\la-\la'$ and $\bj=\al-\la'$, with $\la=(l_1,l_2,l_3),\ka=(k_1,k_2,k_3)$. Besides, the symbol `$\pq$' means
$$(i_1,i_2,i_3) \pq (j_1,j_2,j_3) \Leftrightarrow 
i_s\leqslant j_s \quad\text{for}\quad s=1,2,3.$$
The coefficients $\ca{d}(r)$ are given by
\begin{equation}
    \label{eq:ca}
    \ca{d}(r)=(1+r)^{-\frac{l'_d+k'_d}{2}}
    \sum_{s\in \mZ}C_{l_d}^sC_{k_d}^{l'_d-s}
    (-1)^{k_d-l'_d+s}r^{\frac{l_d+l'_d}{2}-s},
\end{equation}
with $C_n^k$ the generalized combination number which is regarded as zero when $k>n$ or $k<0$. The coefficients $\ga_{\ka'}^{\bj}(r, \bT)$ are defined as
{\small 
\begin{equation}
    \label{def:gamma}
    \begin{split}
     & \ga_{\ka'}^{\bj}(r,\oT)= \\ & \int_{\mR^3}\int_{S^2}
    \left[\Hj^{\bz,\oT}\left({\sqrt{\frac{r}{1+r}}\bg'}\right)-\Hj^{\bz,\oT}
    \left({\sqrt{\frac{r}{1+r}}\bg}\right)\right] 
    \Hkp^{\bz,\oT}\left({\sqrt{\frac{r}{1+r}}\bg}\right)B_{ij}(|\bg|,\varsigma)\mM_{\bz,\oT}
    \left({\sqrt{\frac{r}{1+r}}\bg}\right) \rd\varsigma \rd\bg,
    \end{split}
\end{equation}
}
where $\bg= \bv-\bv_{\ast}$ is the relative velocity and $\bg'=|\bg|\varsigma$ is the post-collision relative velocity.
\end{theorem}

The proof of Theorem \ref{theorem:step1} could be finished with the following steps. First, we introduce the lemma below
\begin{lemma} 
\label{lemma:tran-Hermite}
For any $r>0$, define $\bv=\bh+\frac{r}{1+r}\bg$, $\bw=\sr\bh-\frac{\sr}{1+r}\bg$, then $\forall\; \bT > 0$
\begin{equation}
    \label{eq:tran-Hermite}
    \Hl^{\bz,\oT}(\bv)\Hk^{\bz,\oT}(\bw)=\sum_{\ka'+\la'=\ka+\la}\ca1(r)\ca2(r)\ca3(r)
    H^{\bz,\oT}_{\la'}(\sqrt{1+r}\bh)H^{\bz,\oT}_{\ka'}\left(\sqrt{\frac{r}{1+r}}\bg\right),
\end{equation}
where $\ka=(k_1,k_2,k_3), \la=(l_1,l_2,l_3)$,
$\ka'=(k'_1,k'_2,k'_3), \la'=(l'_1,l'_2,l'_3)$ and 
the coefficients $\ca{d}(r)$ are defined in \eqref{eq:ca}.
\end{lemma}
The proof of Lemma \ref{lemma:tran-Hermite} could be found in App. \ref{app:proof}. Based on Lemma \ref{lemma:tran-Hermite}, we could derive the corollary below 
\begin{corollary} 
\label{corollary:tran-Hermite1}
For any $r > 0$, define $\bv=\bh+\frac{r}{1+r}\bg$, then it holds that 
\begin{equation}
    \label{tran-Hermite1}
    \Hl^{\bz,\oT}(\bv)=(1+r)^{-\frac{|\la|}{2}}\sum_{\ka'+\la'=\la}
    \frac{\la!}{\ka'!\la'!}r^{\frac{|\la'|}{2}}
    H_{\ka'}^{\bz,\oT}(\sqrt{1+r}\bh)H_{\la'}^{\bz,\oT}\left(\sqrt{\frac{r}{1+r}}\bg\right).
\end{equation}
\end{corollary}
The proof is straightforward by letting $\ka=\bz$ in \eqref{eq:tran-Hermite}. Based on Lemma \ref{lemma:tran-Hermite} and Corollary \ref{corollary:tran-Hermite1}, the proof of Theorem \ref{theorem:step1} is as below. 
\begin{proof}[Proof of Theorem \ref{theorem:step1}]
First, we denote $\bg=\bv-\bv_{\ast}$, $\bh=\frac{\bv+r\bv_{\ast}}{1+r}$ in \eqref{eq:Aalk2}, then it holds that
\begin{equation}
    \label{eq: dd_con}
    \rd \bg \rd \bh=\rd \bv \rd \bv_{\ast}, \qquad |\bv|^2+r|\bv_{\ast}|^2=\frac{r}{1+r}|\bg|^2+(1+r)|\bh|^2.
\end{equation}
Similarly, with the conservation of the collision term, if $\bw=\bv', \bw_{\ast}=\bv'_{\ast}$, $\varsigma'=\frac{\bg}{|\bg|}$, then it holds that
\begin{equation}
    \label{eq:dd_con1}
    \bw'=\bv, \bw'_{\ast}=\bv_{\ast}, \qquad \rd \bw\rd \bw_{\ast}=\rd \bv\rd \bv_{\ast}, \qquad |\bw|^2+r|\bw_{\ast}|^2 = |\bv|^2+r|\bv_{\ast}|^2.
\end{equation}
Thus, following the transformation method in \cite{Approximation2019}, we can derive that 
\begin{equation}
    \label{Aalk3}
    \begin{split}
    \Aalk=&\frac{r^{\frac32}}{\al!}\int_{\mR^3}\int_{\mR^3}\int_{S^2}
    B_{ij}(|\bv-\bv_{\ast}|,\varsigma)
    \Hl\UP[i](\bv)\Hk\UP[i](\sr \bv_{\ast})\\
    &\qquad \qquad \left[\Ha\UP[i](\bv')-\Ha\UP[i](\bv)\right] 
    \mM\UP[i](\bv)\mM\UP[i](\sr \bv_{\ast})\rd\varsigma \rd\bv \rd\bv_{\ast}.
    \end{split}
\end{equation}
Combining Lemma \ref{lemma:tran-Hermite} and Corollary 
\ref{corollary:tran-Hermite1}, we rewrite
the above equality as an integral of $\bg$ and $\bh$
\begin{equation}
    \label{Aalk4}
    \Aalk=\frac{r^{\frac32}}{(1+r)^{\frac{|\al|}{2}}}
    \sum_{\la'+\ka'=\la+\ka}\sum_{\bi+\bj=\al}
    \frac{r^{\frac{|\bj|}{2}}}{\bi!\bj!}\ca1(r)\ca2(r)\ca3(r)
    \ga_{\ka'}^{\bj}(r, \zeta\UP[i])\et_{\la'}^{\bi},
\end{equation}
where the coefficients $\ga_{\ka'}^{\bj}$ are integrals with 
respect to $\bg$ defined in
(\ref{def:gamma}), and $\et_{\la'}^{\bi}$ are integrals with 
respect to $\bh$ defined by
\begin{equation}
    \label{def:eta}
    \et_{\la'}^{\bi}=\int_{\mR^3}H\UP[i]_{\la'}(\sqrt{1+r}\bh)H\UP[i]_{\bi}
    (\sqrt{1+r}\bh)
    \mM\UP[i](\sqrt{1+r}\bh)\rd\bh=\frac{\la'!}{(1+r)^{\frac32}}\de_{\la',\bi}.
\end{equation}
The second equality can be obtained by orthogonality \eqref{eq:orth}.
 With \eqref{Aalk4} and \eqref{def:eta}, we complete the 
proof of Theorem \ref{theorem:step1}.
\end{proof}

For some special collision kernels, the calculation of $\Aalk$ could be further simplified. For example, for the VSS collision kernel \eqref{eq:vss}, where $B_{ij}(|\bg|, \varsigma)$ could be rewritten in the form 
\begin{equation}
    \label{eq:form_B}
    B_{ij}(|\bg|,\varsigma)=|\bg|^C\Sigma_{ij}(\varsigma), 
\end{equation}
where $C$ is some constant and $\Sigma_{ij}(\varsigma)$ is a function of $\varsigma$. Then for the VSS model, with the transitivity property \eqref{eq:tran}, the coefficient $ \ga_{\ka}^{\bj}(r,\oT)$ in \eqref{def:gamma} could be further simplified as 
\begin{equation}
    \label{eq:tran_gamma}
    \begin{split}
    \ga_{\ka}^{\bj}(r,\oT)=&\int_{\mR^3}\int_{S^2}
    \left[\Hj^{\bz,1}\left(\beta \bg'\right)-\Hj^{\bz,1}
    \left(\beta\bg\right)\right] \Hk^{\bz,1}\left({\beta\bg}\right)
    B_{ij}(|\bg|,\varsigma)
    \oT^{-\frac32}\mM_{\bz,1}\left(\beta \bg\right)\rd\varsigma \rd\bg \\
    =&\left(\frac{1+r}{2r}\right)^{\frac32}\left(\frac{(1+r)\oT}{2r}\right)^{\frac{C}{2}}
    \ga_{\ka}^{\bj}(1,1), \qquad {\rm with}~\beta = \sqrt{\frac{r}{(1+r)\oT}}.
    \end{split}
\end{equation}
Here, a further simplification of $\ga_{\ka}^{\bj}(1,1)$ is introduced in \cite{Approximation2019}, and we list the result in App. \ref{app:gamma}. For the VSS model, the ultimate form of $\Aalk$ only needs a two-dimensional integration, and the computational cost has been greatly reduced. 

\begin{remark}
    \label{rmk:otherexpans}
    From the calculation of $\Aalk$, we want to emphasize that the expansion center $[\bz, \zeta\UP[i]]$ is not unique. As long as the same expansion center $\ou$ and scaling factor $\mT$ is chosen for each species, (i.e. the expansion center is $[\ou, \mT/m\UP[i]]$ for species $i$), \eqref{eq:Aalk} could be computed in a similar way. 
    It is easy to verify that for the VSS model, it holds that 
    \begin{equation}
        \label{eq:VSS_A}
        \Aalk(\ou,\oT)=\oT^{\frac{C}{2}}\Aalk(\bz,1),
    \end{equation}
    and $\Aalk(\bz, 1)$ only depends on the mass ratio $r = \frac{m\UP[j]}{m\UP[i]}$. Therefore, in the real simulation, we only need to compute and store $\Aalk(\bz, 1)$ with each different mass ratio $r$ in the system. 
\end{remark}

For now, we have introduced the algorithm to calculate the expansion coefficients $\Aalk$ for the Boltzmann collision term. Especially for the VSS model, the final computation is reduced into some elementary operations and a two-dimensional integral. Moreover, the coefficients $\Aalk$ could be precomputed offline and saved for the simulation. The computational cost of the coefficients is listed in Tab. \ref{table:cost1}, where the computational cost for the related coefficients is also listed. 
\begin{table}[ht]
    \centering
    \def\arraystretch{1.5}
    {\footnotesize
    \begin{tabular}{lll}
    \hline 
    Coefficients & Formula & Computational cost \\
    \hline 
       $ D_{\bm}^{\ka}$ & \eqref{eq:def_D} & $\mO(M^6)$ \\
        $S_{\ka}^{\la}$ & \eqref{eq:recur_R} & $\mO(M^5)$ \\
        $K_{\bm\bn}^{\ka\bj}$ & \eqref{eq:def_K} & $\mO(M^4)$ \\
        $\ga_{\ka}^{\la}(r,\oT)$ & \eqref{eq:tran_gamma} and \eqref{eq:algo_gamma} & $\mO(M^{11})$ \\
        $\ca{d}$ & \eqref{eq:ca} & $\mO(M^4)$ \\
        $A\alk$ & \eqref{eq:Aalk2} & $\mO(M^{12})$ \\
    \hline
    \end{tabular}
    }
    \caption{The formula and computational cost to obtain  $A\alk$ and related coefficients \cite{Approximation2019}.}\label{table:cost1}
\end{table}

Even though the expansion coefficients $\Aalk$ could be precomputed, the memory required to store $\Aalk$ is extremely large. Thus, in the framework of the Hermite spectral method to solve the Boltzmann equation, the direct simulation of the collision term is still quite expensive. Following the thought in \cite{Approximation2019}, a new collision model will be built to approximate the quadratic collision term in the framework of Hermite expansion. 

\subsection{Approximation to the Boltzmann collision term}
Supposing that the truncation order of the expansion approximation in \eqref{eq:Her-expan-coll} is $M$, then the computational cost to compute the collision terms $\mQ\UP[ij]$ in \eqref{eq:Qij} is $\mO(s^2M^9)$, where $s$ is the number of species. Such computational cost is unacceptable for a large $M$. For the spatially non-homogeneous problems, the computational cost will be much larger. Therefore, we want to build new collision models following the method proposed in \cite{Approximation2019, ZhichengHu2019}. 

\subsubsection{The BGK collision term} 
Before introducing the new collision model, we will discuss the BGK model \eqref{eq:BGK} firstly. For the BGK model in the multi-species Boltzmann equation, other than that in the single-species case where the BGK model has only one form, there are several variants here. The essential part is how to choose $\bu\UP[ij]$, $T\UP[ij]$ and collision frequency $\nu\UP[ij]$ in \eqref{eq:BGK}, with different models discussed in related literature. It is assumed that $\bu\UP[ij]$ is a linear combination of $\bu\UP[i]$ and $\bu\UP[j]$ in \cite{Klingenberg2017}, and an estimation depending on the macroscopic variables is provided in \cite{Haack2021}. In this work, for simplicity, we assume that $\bu\UP[ij] = \bu\UP[ji]$ and $T\UP[ij] = T\UP[ji]$. Then,  as is discussed in \cite{Haack2017}, it holds for $\bu\UP[ij]$ and $T\UP[ij]$ 
that
\begin{equation}
     \label{eq:mean_uT}
    \begin{gathered}
        \bu\UP[ij]=\frac{\rho\UP[i]\nu\UP[ij]\bu\UP[i]+\rho\UP[j]\nu\UP[ji]\bu\UP[j]}
        {\rho\UP[i]\nu\UP[ij]+\rho\UP[j]\nu\UP[ji]}, \\ 
        T\UP[ij]=\frac{n\UP[i]\nu\UP[ij]T\UP[i]+n\UP[j]\nu\UP[ji]T\UP[j]}
        {n\UP[i]\nu\UP[ij]+n\UP[j]\nu\UP[ji]}+\frac{\rho\UP[i]\nu\UP[ij](|\bu\UP[i]|^2-|\bu\UP[ij]|^2) +  \rho\UP[j]\nu\UP[ji](|\bu\UP[j]|^2-|\bu\UP[ji]|^2)}{3(n\UP[i]\nu\UP[ij]+n\UP[j]\nu\UP[ji])}.
    \end{gathered}
\end{equation}
To approximate the BGK collision term under the framework of the Hermite expansion, we choose the same expansion center $[\bz, \zeta\UP[i]]$ as in \eqref{eq:Her-expan-coll}. Thus, the BGK collision term \eqref{eq:BGK} could be expanded as  
\begin{equation}
    \label{eq:expan_BGK}
    \mQ\UP[ij]_{\rm BGK}(\bv) = \sum_{\alpha \in \mN^3} Q\UP[ij]_{\alpha, \rm BGK} \mH_{\alpha}\UP[i](\bv),
\end{equation}
where 
\begin{equation}
    \label{eq:coe_BGK}
    Q\UP[ij]_{\alpha, \rm BGK} = \nu\UP[ij] n\UP[j]     \Bigg[\bigg(\Pi_{[\bz,\zeta\UP[i]]}\left(\mM_{\bu\UP[ij],\frac{T\UP[ij]}{m\UP[i]}}\right)\bigg)_{\al}- f\UP[i]_{\al}\Bigg]. 
\end{equation}
Here $\Pi_{[\bz,\zeta\UP[i]]}(f)$ is the operator defined as below.

\begin{definition}
\label{def:pro_oper}
First, define the function space, $\mF$ and $\bbH[\ou, \bT]$ as 
 \begin{equation}
        \begin{split}
        & \mF=\left\{f(\bv)\Big|
        \int_{\mR^3}(1+|\bv|^N)f(\bv)\rd \bv < 
        \infty,\;
        \forall N\in \mN
        \right\}, \\
        & \bbH[\ou,\bT]=\text{span}\left\{\mH\aut(\bv) \Big| 
        \alpha \in \mN^3 \right\} \subset \mF.
        \end{split}
    \end{equation}
    Assume that for $t > 0, \bx\in \Omega, f(t,\bx,\cdot)\in \mF$, then the projection operator $\Pi_{[\ou, \oT]}: \mF \to \bbH[\ou, \oT]$ is defined as 
    \begin{equation}
    \label{eq:def_proj}
    \Pi[\ou,\oT](f)=\sum_{\al \in \mN^3}\left[\left(\frac{1}{\al!}
    \int_{\mR^3}H\aut(\bv)f(\bv)\rd \bv\right) \mH\aut(\bv)\right].
\end{equation}
\end{definition}

In the simulation, the expansion coefficient for the the Maxwellian $\mM_{\bu\UP[ij],\frac{T\UP[ij]}{m\UP[i]}}$ in \eqref{eq:coe_BGK} is calculated using the Theorem \ref{thm:project} instead of the definition of the projection operator \eqref{eq:def_proj}, and readers can be referred to \cite[Theorem 3.1]{ZhichengHu2019} for the related proof and the implementation of this projection algorithm. 

\begin{theorem}
    \label{thm:project}
    Suppose $f(\bv) \in \mF$ is expanded with two 
    different expansion centers $[\ou\UP[1],\oT\UP[1]]$ and
    $[\ou\UP[2],\oT\UP[2]]$. From \eqref{eq:falpha}, we can compute these two sets of expansion coefficients as
    \begin{equation}
        \label{eq:twoexpan}
        \begin{split}
            & f^{[\ou\UP[1],\oT\UP[1]]}_{\al}=\frac{1}{\al!}\int H_{\alpha}^{\ou\UP[1],\oT\UP[1]}(\bv)f(\bv)\rd\bv \\
            & f^{[\ou\UP[2],\oT\UP[2]]}_{\al}=\frac{1}{\al!}\int H_{\alpha}^{\ou\UP[2],\oT\UP[2]}(\bv)f(\bv)\rd\bv \\
        \end{split}
    \end{equation}
    Then we have
    \begin{equation}
        \label{eq:project_expan}
        f^{[\ou\UP[2],\oT\UP[2]]}_{\al}=\Big(\oT\UP[2]\Big)^{-\frac{|\al|}2}\sum_{l=0}^{|\al|}\phi_{\alpha}\UP[l], 
    \end{equation}
    where $\phi_{\alpha}\UP[l]$ is defined recursively by
    \begin{equation}
        \label{eq:def_phi}
        \phi_{\alpha}\UP[l]=\left\{
        \begin{array}{ll}
            \Big(\oT\UP[1]\Big)^{\frac{|\alpha|}2}f^{[\ou\UP[1],\oT\UP[1]]}_{\al}, & l=0,  \\
            \frac{1}{l}\sum_{d=1}^3\left[\Big(\ou\UP[2]-\ou\UP[1]\Big)\phi_{\al-e_d}\UP[l-1]+\frac12 \Big(\oT\UP[2]-\oT\UP[1]\Big)\phi_{\al-2e_d}\UP[l-1]\right], & 1\leqslant l \leqslant |\al|. 
        \end{array}
        \right.
    \end{equation}
    In \eqref{eq:def_phi}, terms with any negative index are 
    regarded as $0$.
\end{theorem}

\subsubsection{Building new collision model}
In the problem that is far from equilibrium, the BGK model could not describe the movements of the particles well, so the quadratic collision model should be utilized. However, the quadratic collision is quite expensive to simulate. Thus, we want to build a new collision model combining the quadratic and BGK collision model. 

For the single-species Boltzmann equation, the BGK collision model could not give the correct Prandtl number \cite{Struchtrup}, and the Shakhov model has fixed this problem by modifying the Maxwellian term in the BGK model. To build the new collision model, we borrow this idea from the Shakhov model. 
Instead of modifying the Maxwellian, we will revise the expansion coefficients of the collision model under the framework of the Hermite spectral method. 

In the new collision model, we first choose $M_0 > 0$ to decide the length of the expansion coefficients which are derived from the quadratic collision term, then the rest are obtained from the BGK collision model. Precisely, the new collision model could be written as 
\begin{equation}
    \label{eq:new_coll}
     Q\UP[ij]_{\rm new}  = \sum Q\UP[ij]_{\alpha, \rm new} \mH_{\alpha}\UP[i](\bv),
     \end{equation}
with 
\begin{equation}
\label{eq:coe_new_coll}
    Q\UP[ij]_{\alpha,\rm new}=\left\{ 
    \begin{array}{ll}
        \sum\limits_{|\lbd|, |\kp| \leqslant M_0}\Aalk(\bu,\zeta\UP[i]) f_{\la}
		\UP[i]f_{\ka}\UP[j],  & |\alpha| \leqslant M_0, \\[4mm]
        \nu\UP[ij]_{M_0} Q\UP[ij]_{\alpha, \rm BGK}, & |\alpha| > M_0.
    \end{array}
    \right.  
\end{equation}
Here, $\nu\UP[ij]_{M_0}$ are some positive constants deciding the damping rate, where we follow the same idea in \cite{ZhichengHu2019} to decide this parameter. Since $\Aalk$ could be regarded as a matrix respect to $\la$ and $\ka$ for each fixed $\al$, and then, we set $\nu\UP[ij]_{M_0}$ as the minimum of all eigenvalues of $\Aalk$ among $|\al|\leqslant M_0$. Building the new collision model \eqref{eq:new_coll} here could also be understood as combining the quadratic and BGK collision models as in \cite{ZhichengHu2019, Approximation2019}. In this way, the computational cost for the quadratic collision part in the new collision model will be $\mO(s^2M_0^9)$, which will be reduced greatly compared to \eqref{eq:exp_col}. 

For the new collision model, the length of the quadratic collision model $M_0$ is problem-dependent as the further the system from the equilibrium, the larger $M_0$ is needed. However, there is no standard principle to decide $M_0$. There is an adaptive method in \cite{Adaptive2021} to choose $M_0$, which is still an initial work, and we have not adopted the algorithm therein. 

\begin{remark}
In this section, we are focusing on approximating the quadratic collision term in the homogeneous problem. Therefore, the expansion center to calculate the expansion coefficients $\Aalk$ and build the new collision model \eqref{eq:new_coll} is set as $[\bz, \zeta\UP[i]]$. As is stated in Remark \ref{rmk:otherexpans}, the expansion center could be chosen arbitrarily, so as the expansion of the BGK collision term \eqref{eq:expan_BGK}. 
The transition between different expansion centers could be derived through Theorem \ref{thm:project}, and we refer the readers to \cite{ZhichengHu2019} for more details. 
\end{remark}

\section{Moment equations and numerical algorithm}
\label{sec:numerical}
In the last section, we have explained the approximation of the quadratic collision term for the homogeneous problem under the framework of the Hermite expansion. In this section, the numerical scheme to solve the whole Boltzmann equation \eqref{eq:nondim_Bol} will be introduced. We will first introduce the moment systems in the framework of the Hermite expansions. Since the Strang splitting is adopted to solve the convection and collision part separately, the moment equations will also be derived for the convection and collision step. Precisely, the Boltzmann equation \eqref{eq:nondim_Bol} is split into the convection step and the collision step as 
\begin{itemize}
    \item convection step:
    \begin{equation}
        \label{eq:con}
        \pd{f\UP[i]}{t} + \bv \cdot \nabla_{\bx} f\UP[i] = 0,
    \end{equation}
    \item collision step:
    \begin{equation}
        \label{eq:col}
        \pd{f\UP[i]}{t} = \sum_{j=1}^s
        \frac{1}{\Kn_{ij}}\mQ\UP[ij][f\UP[i],f\UP[j]].
    \end{equation}
\end{itemize}

\subsection{Moment equations} 
\label{sec:moment equations}
To obtain the finite moment system, we first conduct a truncation of the expansion \eqref{eq:Her-expan} as 
\begin{equation}
    \label{eq:Her-expan-fin}
    f\UP[i](t, \bx, \bv) \approx \sum_{|\alpha|\leqslant M} f_{\al}\UP[i](t,\bx)\mH^{\ou\UP[i],\bT\UP[i] }(\bv) \in \bbH_M[\ou^{(i)},\oT^{(i)}], \qquad M \in \mathbb{N},
\end{equation}
where the expansion center is chosen differently for the convection and collision step, and $\bbH_M[\ou^{(i)},\oT^{(i)}]$ is the functional space
\begin{equation}
    \label{eq:space_HM}
    \bbH_M[\ou,\bT]=\text{span}\left\{\mH\aut(\bv) \Big| 
        |\alpha| \leqslant M \right\}.
\end{equation}
For the convection step, we follow the framework in \cite{ZhichengHu2019}, and the identical expansion center all over the spatial space is utilized for each species. Precisely, the choice of $[\ou\UP[i],\bT\UP[i]]$ is based on a priori estimation of the problem such that the expansion center is not too far away from the average macroscopic variables. Then, substituting \eqref{eq:Her-expan-fin} into \eqref{eq:con}, we could derive the moment equations for the convection part as 
\begin{equation}
    \label{eq:mom_con}
    \pd{\bf\UP[i]}{t}+\sum_{d=1}^3\boldsymbol{A}_d\pd{\bf\UP[i]}{x_d}=0.
\end{equation}
where $\bf\UP[i]$ is a column vector of all $f_{\al}\UP[i]$ for $|\al|\leqslant M$ and $\boldsymbol{A}_d$ is a constant matrix made up by the convection term, the detailed form of which can be referred to \cite{ZhichengHu2019} and will be omitted here. Moreover, without loss of generality, we will record $\bf\UP[i]$ as $\bf\UP[i] \in \bbH_M[\ou^{(i)},\oT^{(i)}]$ if the corresponding distribution function $f\UP[i] \in\bbH_M[\ou^{(i)},\oT^{(i)}]$.
For the collision term, the expansion center could be the same as the convection term, the same expansion center as in Sec. \ref{sec:collision}, or 
some other choices could also be utilized. For any expansion center, substituting \eqref{eq:Her-expan-fin} into \eqref{eq:col} and matching the basis functions on both sides, we can derive the moment equations for the collision step as 
\begin{equation}
    \label{eq:mom_col}
    \pd{f\UP[i]_{\alpha}}{t}  = \sum_{j = 1}^s \frac{1}{\Kn_{ij}}Q\UP[ij]_{\alpha}, \qquad |\alpha| \leqslant M. 
\end{equation}
To solve the moment equations of the convection step \eqref{eq:mom_con} and the collision step \eqref{eq:mom_col}, the standard finite volume method will be adopted, which we will introduce in detail in the following sections. 

Here, we also want to emphasize that if different expansion centers are utilized for the convection and the collision steps, the transition of the distribution functions between different expansion centers can be achieved with Theorem \ref{thm:project}.

\subsection{Numerical scheme}
\label{sec:conv}
Supposing the spatial domain $\Omega \in \bbR^3$ is discretized by a uniform grid with the cell size $\Delta \bx$ and the cell centers $\bx_k = (x_{k_1}, x_{k_2}, x_{k_3})$. Using $\bf\UP[i]_{n,k} \in \bbH_M[\ou^{(i)},\oT^{(i)}]$ to approximate the average of $\bf\UP[i]$ over the $k$th grid cell at time $t^n$, then we will introduce the numerical scheme for the Boltzmann equation. It will split into two steps, which we will begin from the convection step. The system \eqref{eq:mom_con} can be solved by the forward-Euler method with time step length $\Delta t$ as follows:
\begin{equation}
    \label{eq:scheme_conserv}
    \frac{\bf^{(i),\ast}_{n, k}-\bf\UP[i]_{n,k}}{\Delta t}+\sum_{d=1}^3
    \frac{\bF\UP[i]_{n,k+\frac12 e_d}-\bF\UP[i]_{n,k-\frac12 e_d}}{\Delta x_d}=0,
\end{equation}
where the finite volume method is employed for spatial discretization and $\bf^{(i),\ast}_{n, k}$ is the result of convection step \eqref{eq:mom_con}. $\bF\UP[i]_{n, k+\frac12 e_d}$ is the HLL flux \cite{HLL} given by
\begin{equation}
    \label{eq:HLL_flux}
    \bF\UP[i]_{n,k+\frac12 e_d}=
    \left\{
    \begin{array}{ll}
    A_d\bf_{n,k+\frac12 e_d}^{(i),L},& \la^L_d\geqslant 0, \\
    \frac{\la_d^RA_d\bf_{n,k+\frac12 e_d}^{(i),L}-
    \la_d^LA_d\bf_{n,k+\frac12 e_d}^{(i),R}+\la_d^R\la_d^L
    \left(\bf_{n,k+\frac12 e_d}^{(i),R}-\bf_{n,k+\frac12 e_d}^{(i),L}\right)}
    {\la_d^R-\la_d^L},& \la^L_d<0<\la^R_d, \\[2mm]
    A_d\bf_{n,k+\frac12 e_d}^{(i),R},& \la^R_d\leqslant 0,
    \end{array}
    \right.
\end{equation}
where $\bf_{n,k+\frac12 e_d}^{(i),L}$
and $\bf_{n,k-\frac12 e_d}^{(i),R}$ can be computed with
the linear reconstruction (also used in \cite{ZhichengHu2019}) 
\begin{equation}
    \label{eq:linear}
    \begin{split}
    & \bg_d^n=\frac{\bf_{n,k+e_d}-\bf_{n,k-e_d}}{2\Delta x}, \\
    & \bf_{n,k+\frac12 e_d}^L=\bf_{n,k}+\frac12\Delta x\bg_d^n, \\
    & \bf_{n,k-\frac12 e_d}^R=\bf_{n,k}-\frac12\Delta x\bg_d^n,
    \end{split}
\end{equation}
or the 3-order WENO reconstruction \cite{Weno}
\begin{equation}
    \label{eq:WENO}
    \begin{split}
    & \bf^{L,1}=\frac32\bf_{n,k}-\frac12\bf_{n,k-e_d}, \quad
    \bf^{L,2}=\frac12\bf_{n,k}+\frac12\bf_{n,k+e_d}, \\
    & \bf^{R,1}=\frac32\bf_{n,k}-\frac12\bf_{n,k+e_d}, \quad
    \bf^{R,2}=\frac12\bf_{n,k}+\frac12\bf_{n,k-e_d}, \\
    & \om_{L,1}=\frac{\ga_1}{\Big[\eps+(\bf_{n,k}-\bf_{n,k-e_d})^2\Big]^2}, \quad
    \om_{L,2}=\frac{\ga_2}{\Big[\eps+(\bf_{n,k+e_d}-\bf_{n,k})^2\Big]^2}, \\
    & \om_{R,1}=\frac{\ga_1}{\Big[\eps+(\bf_{n,k+e_d}-\bf_{n,k})^2\Big]^2}, \quad
    \om_{R,2}=\frac{\ga_2}{\Big[\eps+(\bf_{n,k}-\bf_{n,k-e_d})^2\Big]^2}, \\
    & \bf_{n,k+\frac12 e_d}^L=\frac{\om_{L,1}\bf^{L,1}+\om_{L,2}\bf^{L,2}}
    {\om_{L,1}+\om_{L,2}}, \quad
    \bf_{n,k-\frac12 e_d}^R=\frac{\om_{R,1}\bf^{R,1}+\om_{R,2}\bf^{R,2}}
    {\om_{R,1}+\om_{R,2}}, \\
    & \eps=10^{-6},\quad \ga_1=\frac13, \quad \ga_2=\frac23,
    \end{split}
\end{equation}
where the square of $\bf$ in \eqref{eq:WENO} indicates squaring by each expansion coefficient, and the superscript $(i)$ of $\bf$ is omitted in \eqref{eq:linear} and \eqref{eq:WENO} for simplicity.

In \eqref{eq:HLL_flux}, $\la_d^L$ and $\la_d^R$ are the minimum and maximum eigenvalue of $A_d, d = 1, 2, 3$ respectively.  Due to the same form of $A_d$ here as in \cite{ZhichengHu2019}, $\la_d^L$ and $\la_d^R$ are set as $\la_d^L = \overline{u}_d - C_{M+1} \sqrt{\bT}$ and $\lambda_d^R = \overline{u}_d + C_{M+1}\sqrt{\bT}$. The time step is chosen to satisfy the CFL condition as 
\begin{equation}
    \label{eq:CFL}
    \Delta t \sum_{d = 1}^3 \frac{|\overline{u}_d| + C_{M+1}\sqrt{\bT}}{\Delta x_d} < 1.
\end{equation}

After obtaining the results $\bf^{(i),\ast}_{n, k}$ at each cell in the convection step, we will update the collision step. Since the expansion center of the convection step and the collision step may be different, we will apply the algorithm in Theorem \ref{thm:project} to obtain the new expansion coefficients in the expansion center of the collision step.

For the collision step, we choose a similar expansion center as in Sec. \ref{sec:collision}. Precisely, the expansion center for species $i$ is $[\bu_{n+1, k}, \mT_{n+1, k}/m\UP[i]]$, where $\bu_{n+1, k}$ is the the average macroscopic velocity in \eqref{eq:macro_tot} and $\mT_{n+1,k}$ is the average temperature in \eqref{eq:barT} in the $k$-th cell at time step $n+1$. Since the collision term does not change the average macroscopic velocity and the average temperature in each cell, the expansion center $[\bu_{n+1, k}, \mT_{n+1, k}/m\UP[i]]$ here is computed using the distribution function $\bf^{(i), \ast}_{n,k}$ after the convection step.  
Thus, with the projection operator defined in \eqref{eq:def_proj}, we first compute $\bf^{(i), \ast\ast}_{n, k}$ as 
\begin{equation}
    \label{eq:coll_step_1}
    \bf^{(i),\ast\ast}_{n, k}= \Pi\left[\bu_{n+1, k}, \mT_{n+1, k}/m\UP[i]\right](\bf^{(i),\ast}_{n, k}).    
\end{equation}
Then, at the $k$-th cell, we compute the collision term  $\bQ\UP[ij]_{\rm new,n,k}$ in \eqref{eq:coe_new_coll} with the distribution function $\bf^{(i), \ast\ast}_{n, k}$. The forward-Euler scheme is adopted here to update the distribution function in the collision step \eqref{eq:mom_col} as 
\begin{equation}
\label{eq:coll_step_2}
\bf^{(i),\ast\ast\ast}_{n,k} =\bf^{(i),\ast\ast}_{n,k}+\Delta t \sum_{j=1}^s\frac{1}{Kn_{ij}}\bQ\UP[ij]_{\rm new, n, k}. 
\end{equation}
The final step is to project the distribution function $\bf^{(i),\ast\ast\ast}_{n,k}$ to the expansion center $[\ou\UP[i], \oT\UP[i]]$  of the convection step as 
\begin{equation}
    \label{eq:coll_step_3}
    \bf\UP[i]_{n+1,k}=\Pi[\ou\UP[i],\oT\UP[i]](\bf^{(i),\ast\ast\ast}_{n, k}).
\end{equation}
For the forward Euler scheme in \eqref{eq:coll_step_3}, the Runge-Kutta scheme could also be utilized, and the time step $\Delta t$ may also be restricted by the collision term. A small enough and problem-dependent $\Delta t$ should be adopted, which we will not discuss in detail here. The Maxwell boundary condition \cite{Maxwell1878} is adopted here and the detailed implementation can be referred to \cite{ZhichengHu2019}. 

For now, we have finished the whole numerical scheme, and it will be summarized as Algorithm \ref{algo:inhomo}. 
\begin{algorithm}[htbp]
    \caption{Numerical algorithm for one time step}
    \label{algo:inhomo}
    \begin{algorithmic}[1]
        \item Given $\bf_{n,k}\UP[i]\in \bbH_M[\ou\UP[i],\oT\UP[i]]$ at
        each cell $k$ for time step $t_n$.
        \item Solve the convection step \eqref{eq:mom_con} 
        with the numerical scheme \eqref{eq:scheme_conserv} to gain
        $$\bf^{(i),\ast}_{n,k}\in \bbH_M[\ou^{(i)},\oT^{(i)}].$$
        \item Calculate the expansion center  $\bu_{n+1, k}, \mT_{n+1, k}$ of $\bf^{(i),\ast}_{n,k}$ at each cell $k$.
        \item In each cell, obtain $\bf^{(i),\ast\ast}_{n,k}$ with       \eqref{eq:coll_step_1}.
        \item Compute the collision term $\bQ\UP[ij]_{\rm new,n, k}$ in \eqref{eq:coe_new_coll} with $\bf^{(i),\ast\ast}_{n,k}$.
        \item Update $\bf^{(i),\ast\ast \ast}_{n,k} $ with the forward Euler scheme \eqref{eq:coll_step_2}. 
        \item Obtain the final $\bf_{n+1, k}\UP[i] \in \bbH_M[\ou\UP[i],\oT\UP[i]]$ with \eqref{eq:coll_step_3}
        and enter the next time step.
    \end{algorithmic}
\end{algorithm}

\section{Numerical experiments}
\label{sec:experiment}
In this section, several numerical examples are studied to validate the numerical algorithm proposed in the last section. Precisely, the spatial homogeneous test with Krook-Wu solution, two spatially one-dimensional and a spatially two-dimensional problem are tested. For the spatially non-homogeneous problems, the SPARTA \cite{Sparta} is utilized to obtain the reference solution with DSMC method, and the mixture of Argon and Krypton is taken as the working gas, the parameters of which are listed in Tab. \ref{table:coll_para} in App. \ref{app:col_param}. In the final example, we study the Krook-Wu solution with $100$ species to show the superiority of this numerical method.

\subsection{Spatially homogeneous case: Krook-Wu solution}
\label{subsec:2KW}
For the spatially homogeneous problem, we will consider a constant collision kernel. In this case, the spatially homogeneous two-species Boltzmann equation can be simplified as
\begin{equation}
    \label{eq:KW2}
    \pd{f^{(i)}}{t}(t,\bv)=\sum_{j=1}^s\int_{\mR^3}\int_{S^2}B_{ij}
    \left[f\UP[i](\bv')f\UP[j](\bv'_{\ast})-f\UP[i](\bv)f\UP[j](\bv_{\ast})\right]
    \rd\varsigma \rd\bv_{\ast}, 
\end{equation}
where $s = 2$ is the number of species. $B_{ij}=B_{ji}=\frac{\la_{ij}}{4\pi n^{(j)}}$ and $\la_{ij}$ are some positive constants. This problem is also studied in \cite{Shashank2019, Approximation2019}. There exists an exact solution to this problem \cite{krook1977exact} as 
\begin{equation}
    \label{eq:soluKW2}
    f^{(i)}(t,\bv)=n\UP[i]\left(\frac{m\UP[i]}{2\pi K(t)}\right)^{\frac32}
    \exp\left(-\frac{m\UP[i]|\bv|^2}{2K(t)}\right)
    \left[1-3p_iQ(t)+\frac{m\UP[i]}{K(t)}p_iQ(t)|\bv|^2\right],
\end{equation}
where 
\begin{equation}
    \label{eq:paraKW2_1}
    \begin{gathered}
     p_1=\la_{22}-\la_{21}\mu(3-2\mu),  \qquad p_2=\la_{11}-\la_{12}\mu(3-2\mu), \\
      Q(t)=\frac{A}{A\exp(A(t+t_0))-B}, \qquad    K(t)=\frac{n\UP[1]+n\UP[2]}{(n\UP[1]+n\UP[2])+2(n\UP[1]p_1+n\UP[2]p_2)
    Q(t)},
\end{gathered}
\end{equation}
with 
\begin{equation}
    \label{eq:paraKW2_2}
 \mu=\frac{4m\UP[1]m\UP[2]}{(m\UP[1]+m\UP[2])^2}, \qquad A=\frac16\left[\la_{11}+\la_{21}\mu\left(3-2\mu\frac{p_2}{p_1}\right)\right],
    \qquad B=2p_1A, 
\end{equation}
and $t_0$ is a constant. The number density $n\UP[i]$ should be chosen such that $B_{ij}=B_{ji}$, while the mass $m\UP[i]$ can be arbitrary with the following condition satisfied
\begin{equation}
(p_1-p_2)\left[2\mu^2\left(\frac{\la_{21}}{p_1}-\frac{\la_{12}}{p_2}\right)-1\right]=0.
\end{equation}
Here, we choose the same parameters as in \cite{Approximation2019} that  $n\UP[1] = n\UP[2] = 1$, $(m\UP[1], m\UP[2]) = (1, 2)$,  and the constant $t_0 = 3$. Moreover, the constants $\lambda_{ij}$ are chosen as $(\lambda_{11}, \lambda_{12}) = (1, 0.5)$, and $(\lambda_{21}, \lambda_{22}) = (0.5, 1)$. 

It can be directly computed from \eqref{eq:soluKW2} that the average temperature is 
$\mT=1$, and therefore we choose the expansion center as $[\ou\UP[i], \ot\UP[i]] = [\bz, \mT /m\UP[i]], \;i = 1, 2.$ Based on this expansion center, we can obtain the expansion coefficient $f_{\al}\UP[i](t)$ of the exact solution \eqref{eq:soluKW2} as 
\begin{equation}
    \label{eq:exactmom}
    f_{\al}\UP[i](t)=
    \left\{
    \begin{array}{ll}
    \frac{1-\frac{|\al|}{2}}{\al!}\prod\limits_{j=1}^3\Big[(\al_j-1)!!\Big](-2p_i)^
    {\frac{|\al|}2} 
    \exp\left(-\frac{|\al|}2A(t+t_0)\right),
    & \al_1,\al_2,\al_3\; \text{all even}, \\
    0, &\text{Otherwise},
    \end{array}
    \right.
\end{equation}
for all $\alpha\in\bbN^3$ and $i = 1, 2$. 

In the computation, the time step length is fixed as $\Delta t = 0.01$, and the fourth-order Runge-Kutta scheme is adopted. The total expansion order $M$ is $M=20$, and 
two cases with different length of the quadratic collision term $M_0$ are tested, where $M_0$ is set as $M_0 = 5$ and $M_0 = 10$ respectively. The numerical results of
$f_{400}(t)$ from $t=0$ to $t=5$ are given in Fig. \ref{fig:Coe400}, since the expansion coefficients for the lower orders are all constant with respect to $t$. We can find that for the two cases, the numerical results of both species all match well with the exact solution at any time. 
\begin{figure}[h!]
  \centering
  \subfloat[Species 1, $M_0=5$]
  {\includegraphics[width=0.45\textwidth, height=0.365\textwidth,
    clip]{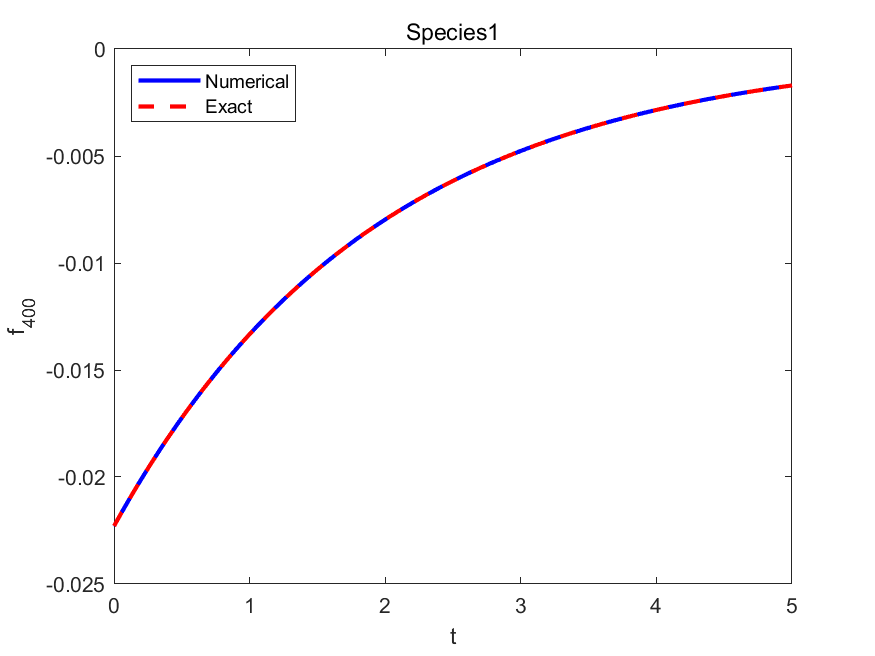}}\hfill
  \subfloat[Species 2, $M_0=5$]
  {\includegraphics[width=0.45\textwidth, height=0.365\textwidth,
    clip]{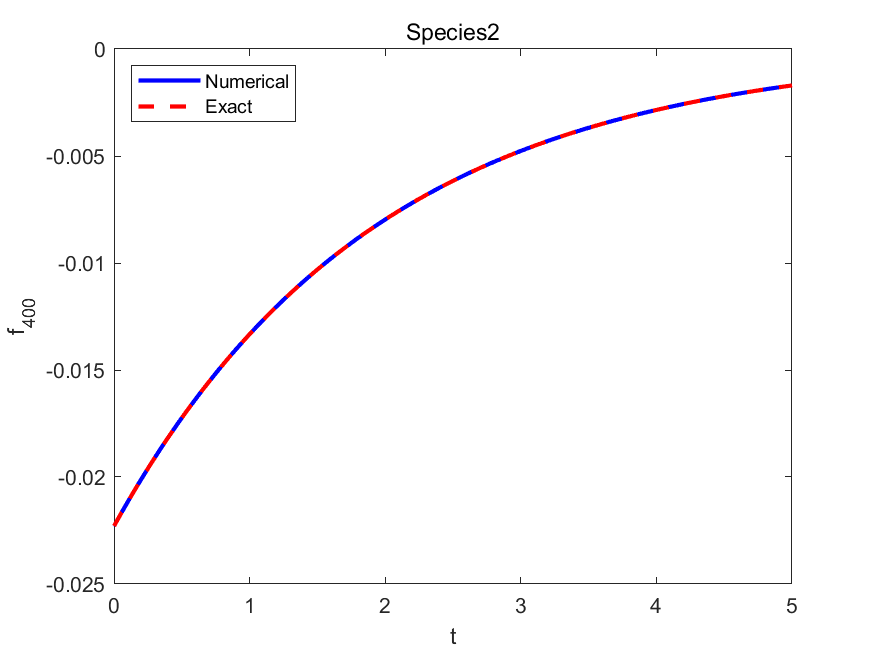}} \\
  \subfloat[Species 1, $M_0=10$]
  {\includegraphics[width=0.45\textwidth, height=0.36\textwidth,
    clip]{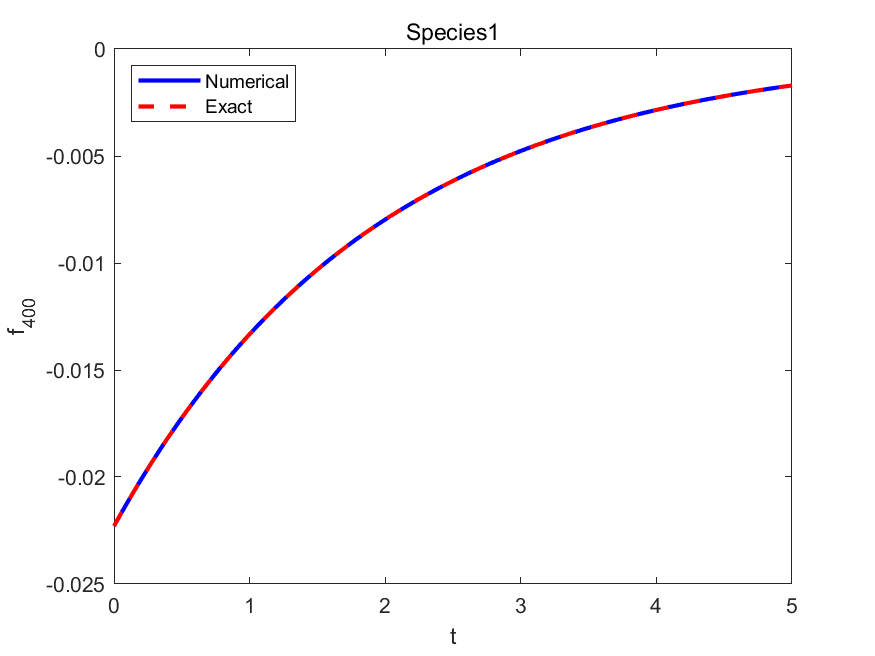}}\hfill
  \subfloat[Species 2, $M_0=10$]
  {\includegraphics[width=0.45\textwidth, height=0.365\textwidth,
    clip]{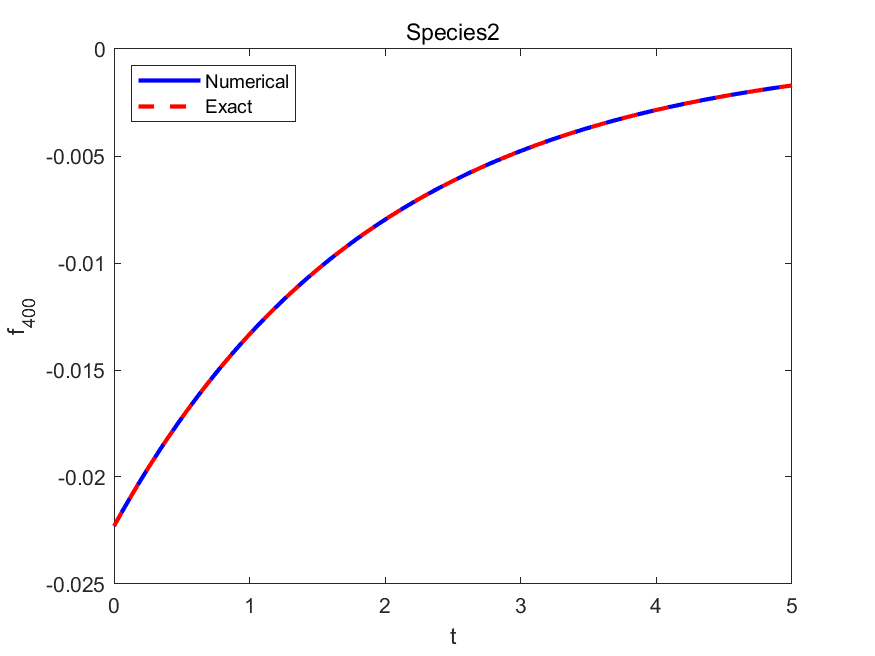}}\hfill
  \caption{Krook-Wu solution with two species in Sec. \ref{subsec:2KW}: Numerical results of $f_{400}$. Blue lines are numerical solutions, while red lines are exact solution.}
  \label{fig:Coe400}
\end{figure}


To see the numerical results of the distribution function more clearly, we randomly select several vectors  $\sigma \in S^2$ for each species and show the results of $f(\bv)$ along $\sigma$ at time $t=0.5$ and $t=1$, which are shown in Fig. \ref{fig:KW2}. In Fig. \ref{fig:KW2},  three cases are plotted with $(M, M_0) = (20, 5), (10, 10), (20, 10)$ respectively. For the case $(20, 5)$, there is an obvious distance between the numerical solution and the exact solution, and for the case $(10, 10)$, the numerical solution is still far away from the numerical solution, while the results for the case $(20, 10)$ are almost on top of the exact solution. From this, we can see that with the increase of $M_0$, the numerical solution could describe the distribution function better. The results in Fig. \ref{fig:KW2} also show that the $M$-order approximation is much better than those with $M_0$-order, which exhibits the importance of the expansion coefficients with the order higher than $M_0$ and validates the advantage of this numerical algorithm as well.

\begin{figure}[!ht]
  \centering
  \subfloat[Species 1, $t=0.5$]
  {\includegraphics[width=0.45\textwidth, height=0.365\textwidth,
    clip]{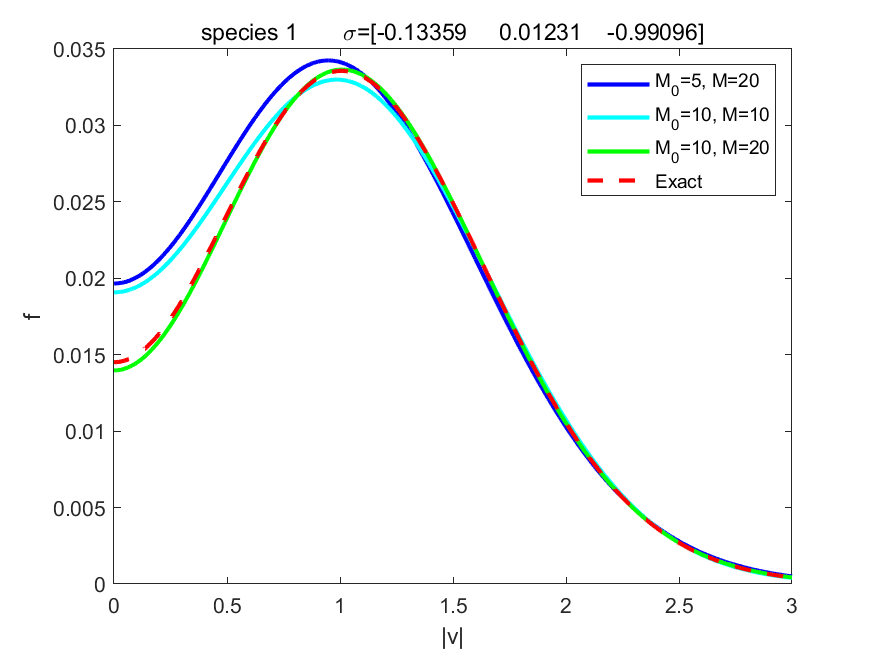}}\hfill
  \subfloat[Species 2, $t=0.5$]
  {\includegraphics[width=0.45\textwidth, height=0.365\textwidth,
    clip]{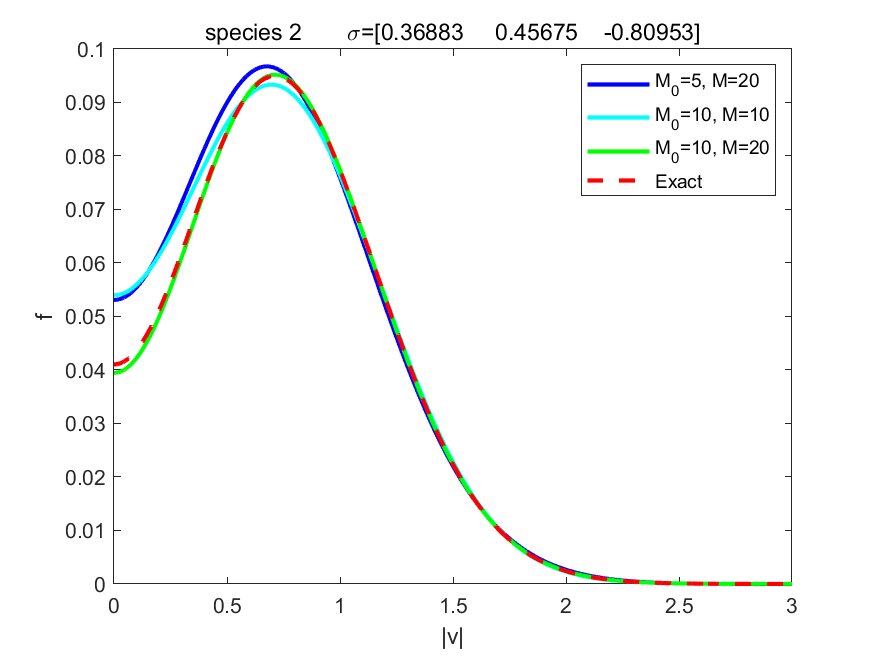}} \\
  \subfloat[Species 1, $t=1$]
  {\includegraphics[width=0.45\textwidth, height=0.36\textwidth,
    clip]{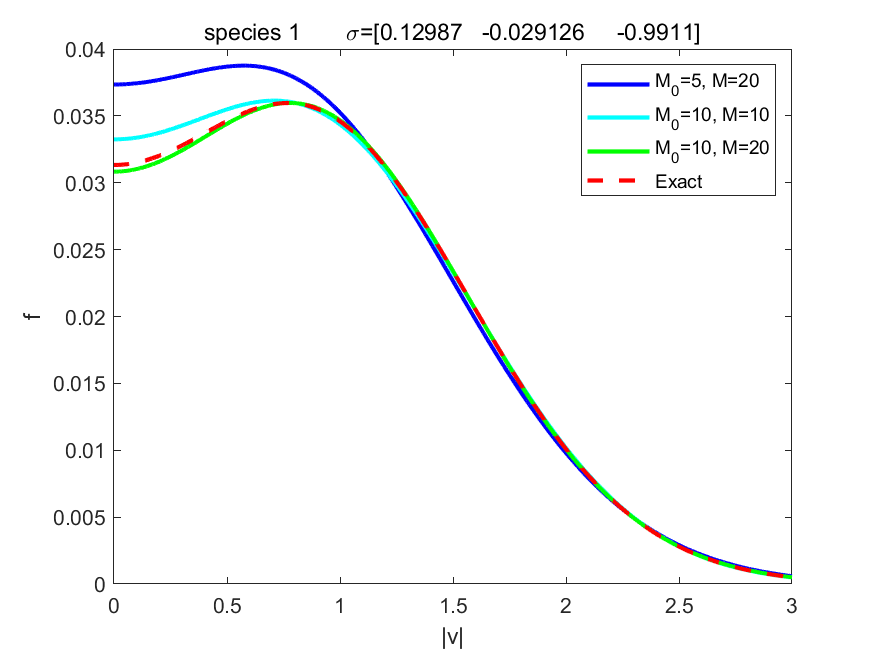}}\hfill
  \subfloat[Species 2, $t=1$]
  {\includegraphics[width=0.45\textwidth, height=0.365\textwidth,
    clip]{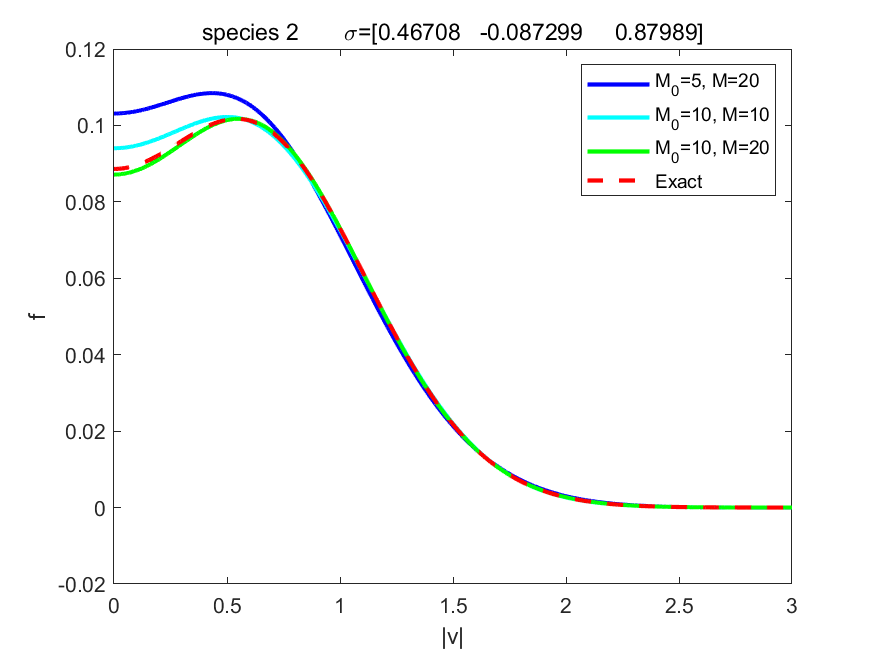}}\hfill
  \caption{Krook-Wu solution with two species in Sec. \ref{subsec:2KW}: Distribution functions with different expansion order for the two species at different time. The top row is at $t = 0.5$ while the bottom row is at $t = 1$.  Blue lines correspond to the results with $M_0=5, M=20$; Cyan lines correspond to $M_0=10, M=10$; Green lines correspond to $M_0=10, M=20$, and red lines correspond to exact solutions.}
  \label{fig:KW2}
\end{figure}

\subsection{1D case: Couette flow}
\label{sec:couette}
In this section, we will consider the Couette flow, which is also tested in \cite{Shashank2019, ZhichengHu2019}. The Ar-Kr mixture between two infinite parallel plates with a distance of $10^{-3}$m will be studied when they arrive at the steady state. Both the left and right plates have the temperature $T=273$K and move in the opposite direction along with the plate with the speed $u^w= (0, \mp50, 0)$ m/s. Besides, both walls are purely diffusive. The VSS collision model \eqref{eq:vss} is utilized here, the detailed parameters of which are shown in Tab. \ref{table:coll_para}.  In this simulation, a uniform grid with $25$ cells is utilized for the spatial discretization with the WENO reconstruction adopted. The CFL condition is set as ${\rm CFL} = 0.45$. Three tests with different number densities are carried out, which are corresponding to different Knudsen numbers. The detailed parameters for the initial conditions are listed in Tab. \ref{table:couette}, where the parameters for the nondimensionalization are also listed. With \eqref{eq:Kn}, we can obtain the corresponding Knudsen number for the three cases as in Tab. \ref{table:couette}. 

\begin{table}[!htp]
\centering
\def\arraystretch{1.5}
{\footnotesize
\begin{tabular}{llll}
\hline 
                                   & Case 1         & Case 2         & Case 3\\
\hline
Initial conditions:                 &               &               &               \\
Temperature $T$(K)                 & $273$           & 273           & 273           \\
Velocity $\bv$ (m/s)               & (0, 0, 0)       & (0, 0, 0)       & (0, 0, 0)       \\
Number density $n_{\rm Ar}$(m$^{-3}$)        & 1.68e21       & 5.6e20        & 1.68e20       \\
Number density $n_{\rm Kr}$(m$^{-3}$)        & 8.009e20      & 2.667e20      & 8.009e19      \\ \hline
Characteristic variables:           &               &               &               \\
Length $x_0$ (m)                   & 1e-3  & 1e-3  & 1e-3 \\
Mass $m_0$ (kg)                    & 6.63e-27      & 6.63e-27      & 6.63e-27      \\
Number density $n_0$(m$^{-3}$)     & 1.68e21       & 5.6e20        & 1.68e20       \\
Temperature $T$(K)                 & 273           & 273           & 273           \\
Velocity $u_0$(m/s)                & 238.377       & 238.377       & 238.377       \\
Knudsen number $(\Kn_{11}$, $\Kn_{12})$ & (0.793, 0.804) & (2.379, 2.411) & (7.931, 8.036) \\
Knudsen number $(\Kn_{21}$, $\Kn_{22})$ & (0.606, 0.555) & (1.819, 1.664) & (6.065, 5.548) \\ \hline
\end{tabular}
}
\caption{Couette flow in Sec. \ref{sec:couette}: Running parameters for Couette flow.}\label{table:couette}
\end{table}

\begin{figure}[!htb]
    \centering
      \subfloat[$y$-component velocity, $u_2$ (m/s)]
    {\includegraphics[width=0.45\textwidth, height=0.36\textwidth,
      clip]{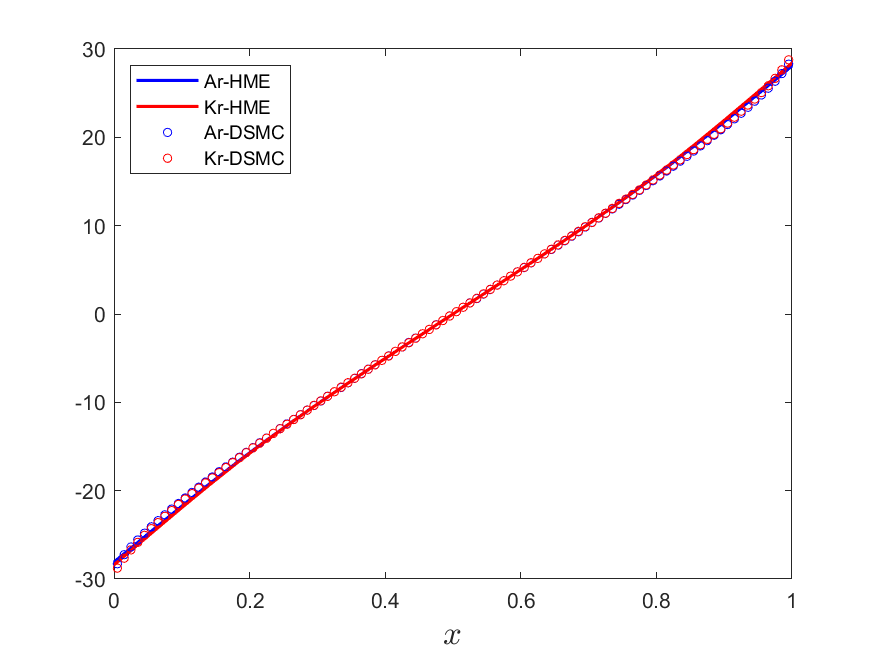}}\hfill
      \subfloat[Temperature, $T$(K)]
    {\includegraphics[width=0.45\textwidth, height=0.36\textwidth,
      clip]{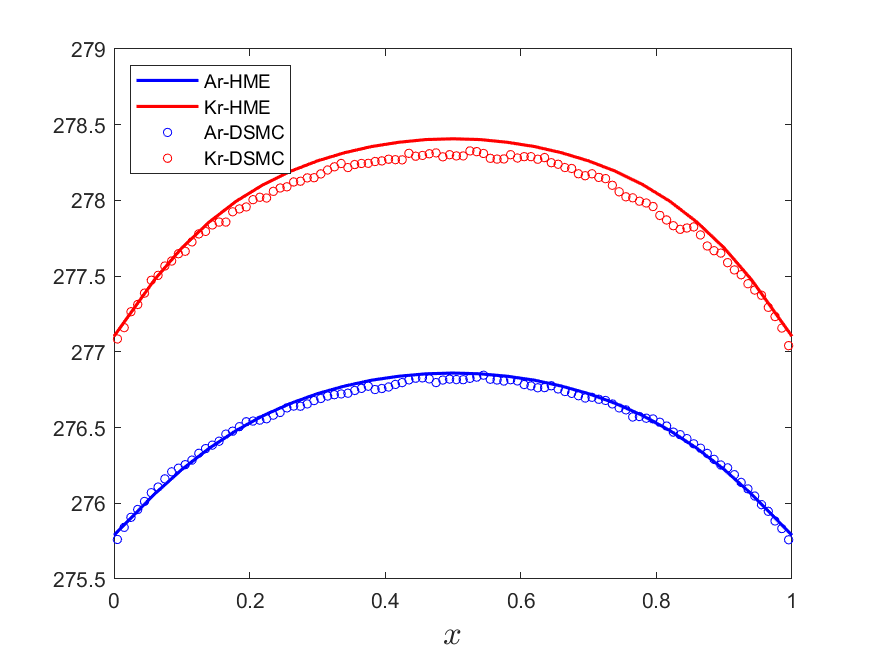}}\\
      \subfloat[Stress tensor, $\sigma_{12}$ (kg$\cdot$ m$^{-1}$ $\cdot$ s$^{-2}$)]
    {\includegraphics[width=0.45\textwidth, height=0.36\textwidth,
      clip]{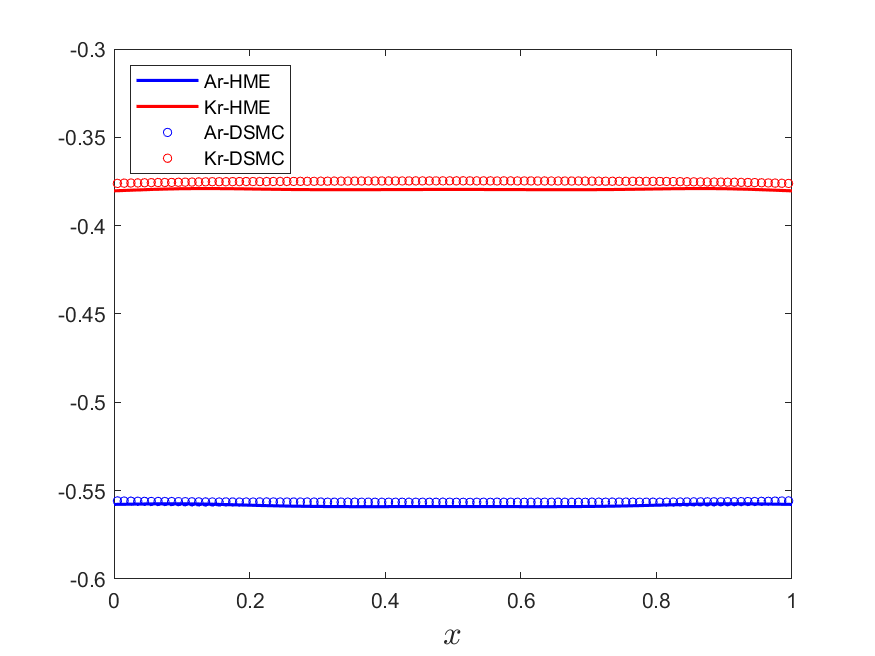}}\hfill
      \subfloat[Heat flux, $q_1$ (kg $\cdot$  s$^{-3}$) ]
    {\includegraphics[width=0.45\textwidth, height=0.36\textwidth,
      clip]{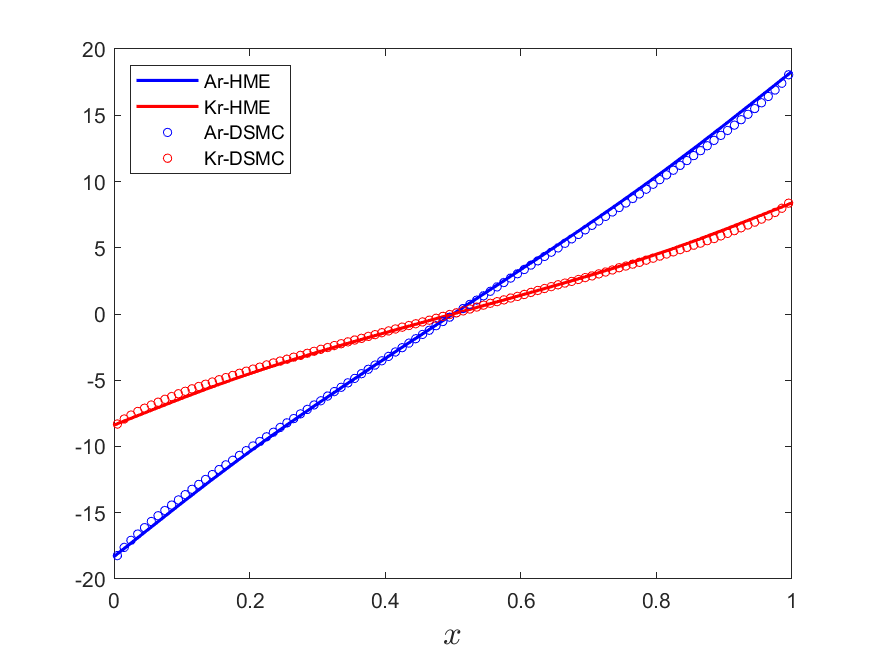}}\hfill
    \caption{Couette flow in Sec. \ref{sec:couette}: Numerical solutions of the Couette flow for case 1, where the Knudsen number is $(\Kn_{11}, \Kn_{12})= (0.793, 0.804)$ and $(\Kn_{21}, \Kn_{22}) = (0.606, 0.555)$. }
    \label{fig:Couette1}
\end{figure}

\begin{figure}[!htb]
    \centering
      \subfloat[$y$-component velocity, $u_2$ (m/s)]
    {\includegraphics[width=0.45\textwidth, height=0.36\textwidth,
      clip]{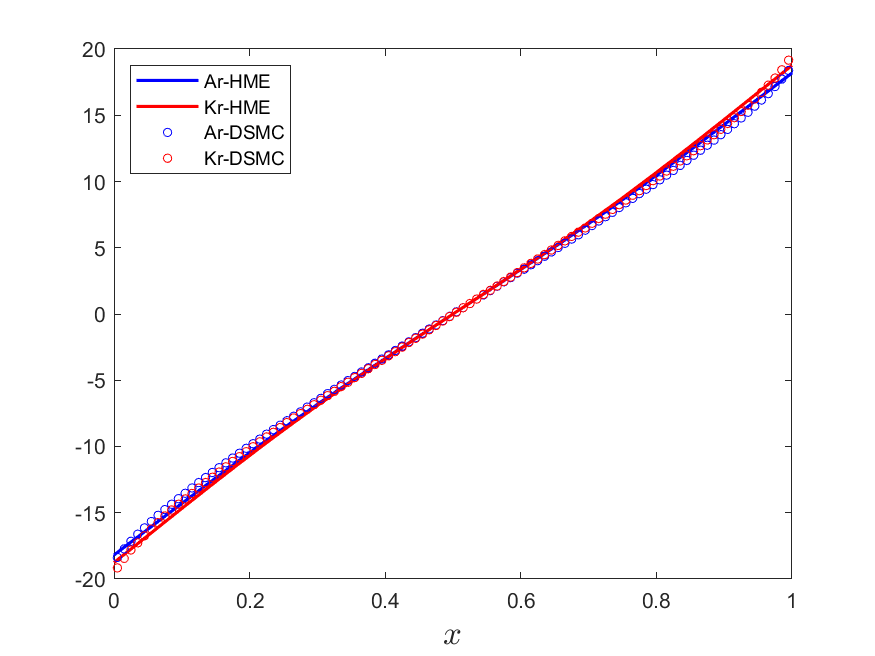}}\hfill
      \subfloat[Temperature, $T$(K)]
    {\includegraphics[width=0.45\textwidth, height=0.36\textwidth,
      clip]{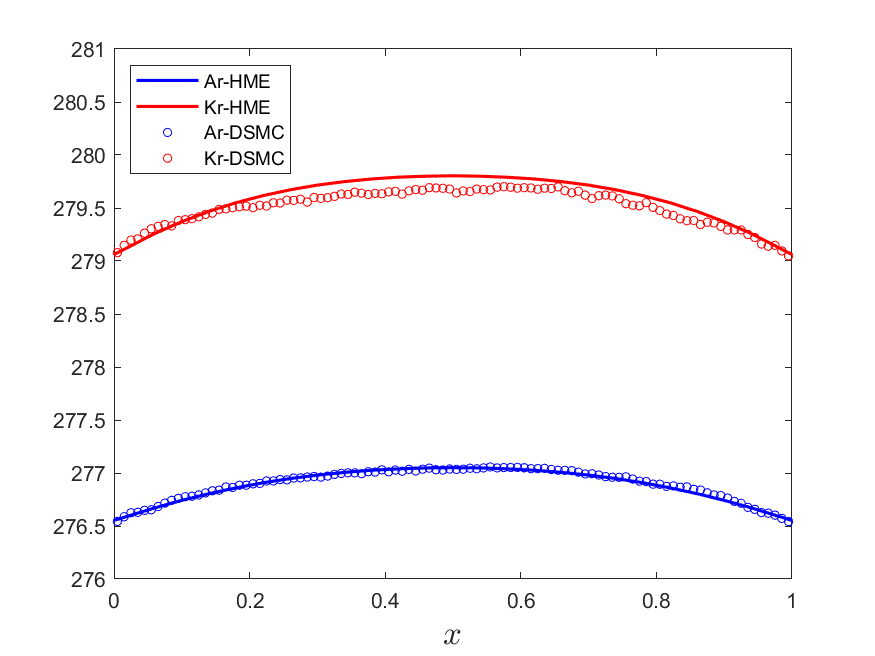}}\\
      \subfloat[Stress tensor, $\sigma_{12}$ (kg$\cdot$ m$^{-1}$ $\cdot$ s$^{-2}$)]
    {\includegraphics[width=0.45\textwidth, height=0.36\textwidth,
      clip]{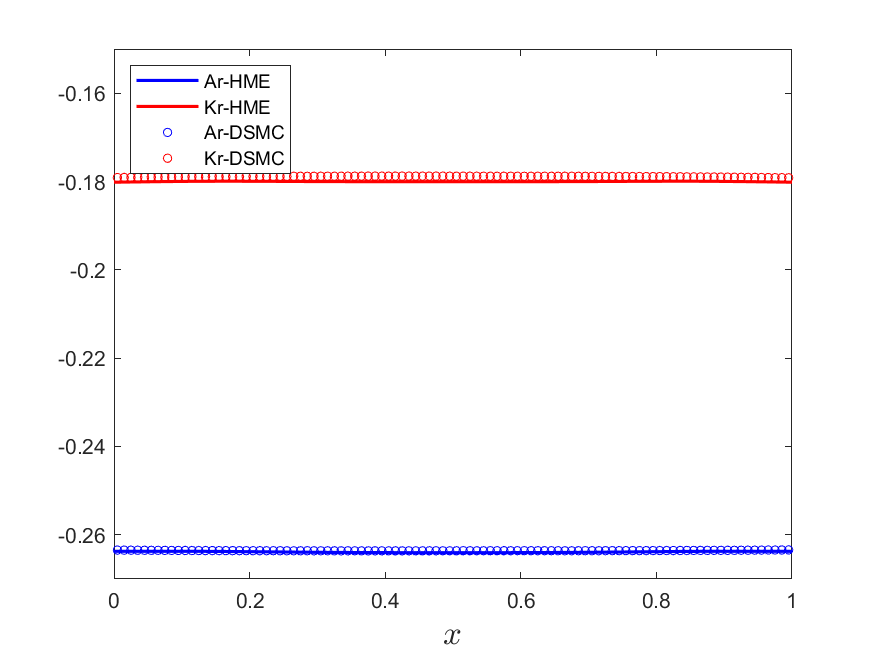}}\hfill
      \subfloat[Heat flux, $q_1$ (kg $\cdot$  s$^{-3}$) ]
    {\includegraphics[width=0.45\textwidth, height=0.36\textwidth,
      clip]{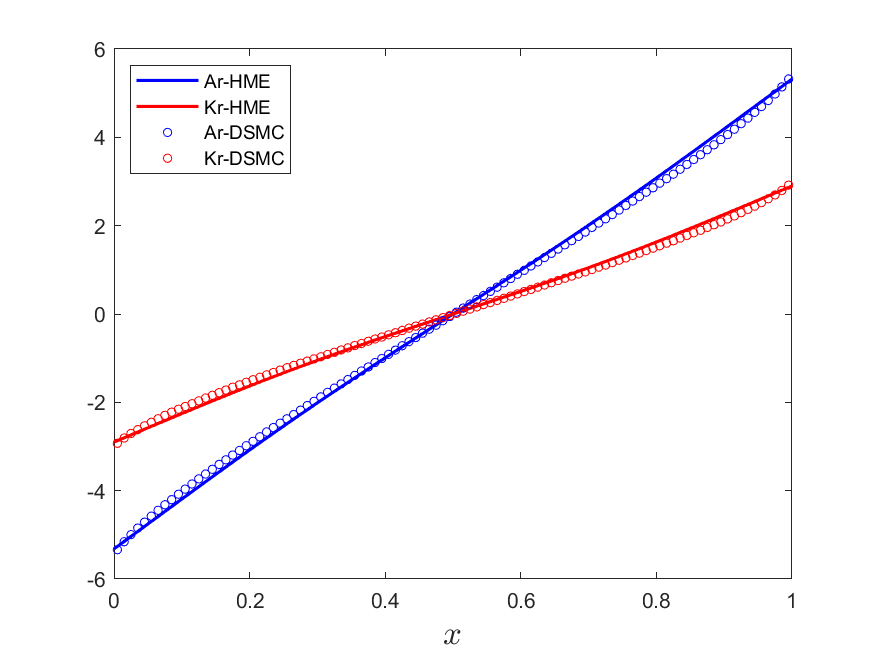}}\hfill
    \caption{ Couette flow in Sec. \ref{sec:couette}: Numerical solutions of the Couette flow for case 2, where the Knudsen number is $(\Kn_{11}, \Kn_{12})= (2.379, 3.411)$ and $(\Kn_{21}, \Kn_{22}) = (1.819, 1.664)$. }
    \label{fig:Couette2}
\end{figure}

\begin{figure}[!htb]
    \centering
      \subfloat[$y$-component velocity, $u_2$ (m/s)]
    {\includegraphics[width=0.45\textwidth, height=0.36\textwidth,
      clip]{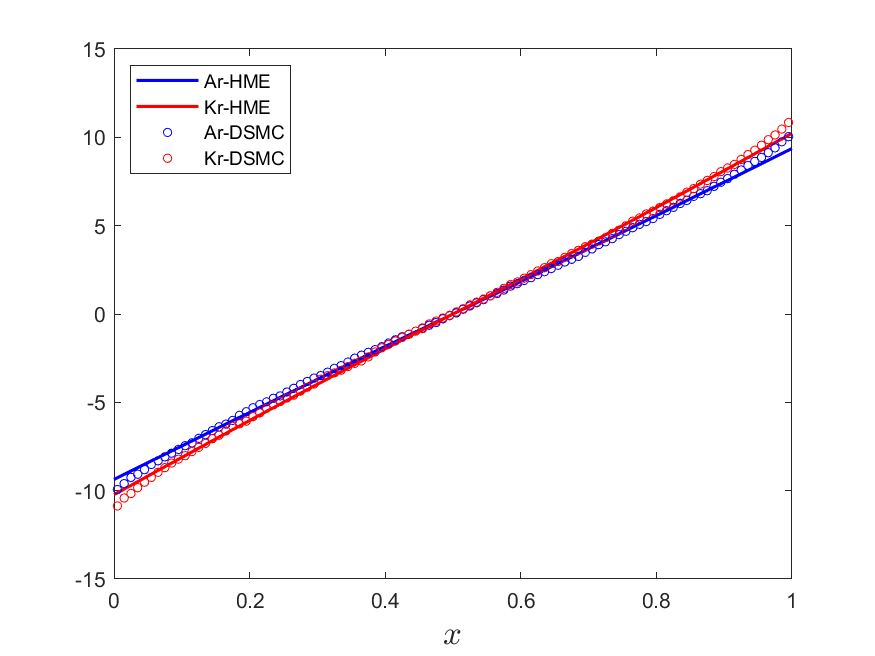}}\hfill
      \subfloat[Temperature, $T$(K)]
    {\includegraphics[width=0.45\textwidth, height=0.36\textwidth,
      clip]{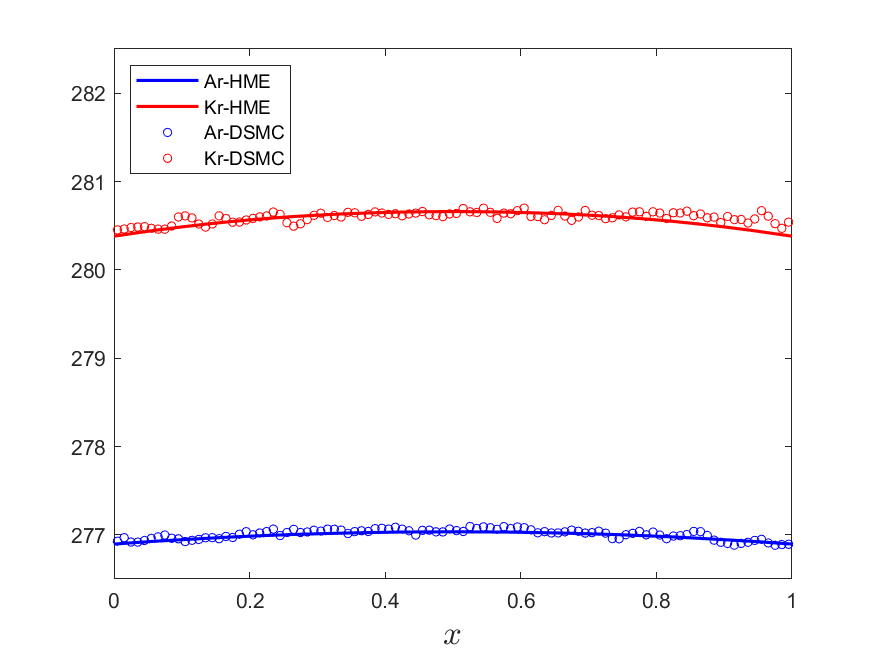}}\\
      \subfloat[Stress tensor, $\sigma_{12}$ (kg$\cdot$ m$^{-1}$ $\cdot$ s$^{-2}$)]
    {\includegraphics[width=0.45\textwidth, height=0.36\textwidth,
      clip]{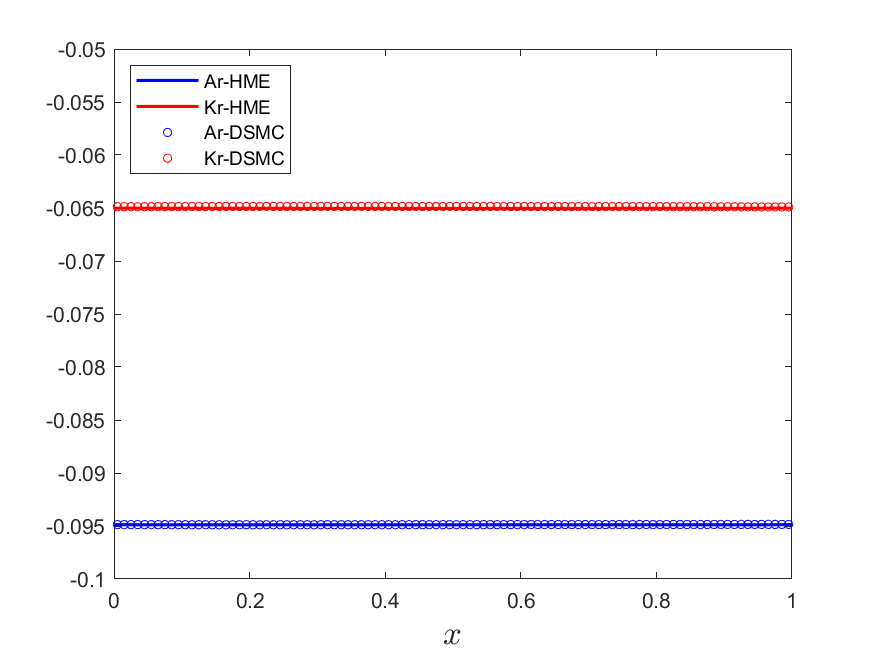}}\hfill
      \subfloat[Heat flux, $q_1$ (kg $\cdot$  s$^{-3}$) ]
    {\includegraphics[width=0.45\textwidth, height=0.36\textwidth,
      clip]{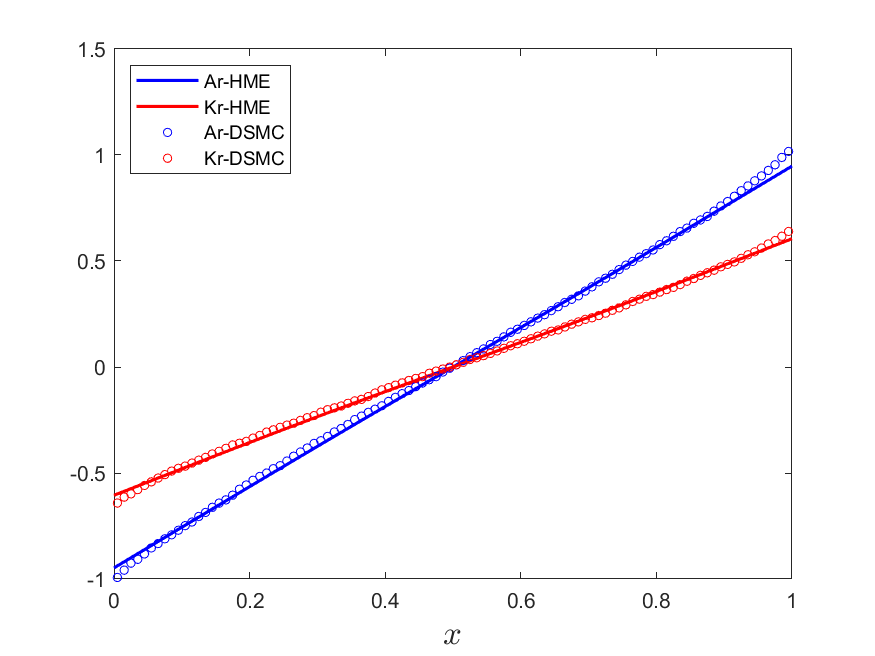}}\hfill
    \caption{
    Couette flow in Sec. \ref{sec:couette}: Numerical solutions of the Couette flow for case 3, where the Knudsen number is $(\Kn_{11}, \Kn_{12})= (7.931, 8.036)$ and $(\Kn_{21}, \Kn_{22}) = (6.065, 5.548)$.}
    \label{fig:Couette3}
\end{figure}

For the Couette flow, the average macroscopic velocity of the whole domain is zero. Therefore, we choose the expansion center for the convection step as $\ou_{\rm Ar} = \ou_{\rm Kr} = \bz$. The average temperature for the initial condition after non-dimensionalization is $\mT_{\rm Ar} = \mT_{\rm Kr} = 1$. Therefore, in the convection step, the expansion center of temperature is set as $\oT_{\rm Ar} = \mT_{\rm Ar} / \hat{m}_{\rm Ar} = 1$ and  $\oT_{\rm Kr} = \mT_{\rm Kr} / \hat{m}_{\rm Kr} = 0.477$, where $\hat{m}_{\rm Ar}$ and $\hat{m}_{\rm Kr}$ are the mass after non-dimensionalization. In the collision step, the expansion center will be decided by local macroscopic variables. The details can be referred to Algorithm \ref{algo:inhomo}. 

In the simulation, the macroscopic velocity in the $y$-direction $u_2$, temperature $T$, stress tensor $\sigma_{12}$, and the heat flux $q_1$ are studied. In the first case, since the Knudsen number is relatively small, we set the length of the quadratic collision term as $M_0 = 10$, and the total moment number as $M = 40$. The numerical results are shown in Fig. \ref{fig:Couette1}, from which we can see that for these both species Ar and Kr, the numerical solutions of these four variables match well with the reference solutions obtained by the DSMC method. Though there is a little difference in the temperature between the numerical solution and the reference, the relative error is less than $0.1\%$. For the second case, we still set $M_0 = 10$ and $M = 40$. The behavior of the numerical solution is similar to that of the first case, as is shown in Fig. \ref{fig:Couette2}, where the numerical solution and the reference solution are almost on top of each other. For the third case, since the Knudsen number is relatively large, we increase $M$ to $60$ and keep $M_0 = 10$ due to the limitation of the computational cost. From Fig. \ref{fig:Couette3}, we can see that for the two species, though there is a little disparity between the numerical solution and the reference solution for the velocity $u_2$ and the heat flux $q_1$ on the boundary, the relative error is quite small. We also find that in this case, there are small oscillations for the reference solution in temperature $T$, while the numerical results are still quite smooth. 

\begin{table}[!htp]
\centering
\def\arraystretch{1.5}
{\footnotesize
\begin{tabular}{llll}
\hline 
                                   & Case 1         & Case 2         & Case 3\\
\hline
Run-time data:  &  &  & \\
Total CPUtime $T_{\rm CPU}$(s): & 15264 & 17217 & 27273 \\ 
Elapsed time(Wall time) $T_{\rm Wall}$(s): & 4923.69 & 6600.08 & 10164.6 \\
\hline
\end{tabular}
}
\caption{Couette flow in Sec. \ref{sec:couette}: Run-time data for Couette flow.}\label{table:couette_runtime}
\end{table}

The simulation of the Couette flow is done on the CPU model Intel Xeon E5-2697A V4 @ 2.6GHz, and 4 threads are used in this test. We implement the simulation for a long enough time to obtain the steady-state solution. In the real simulation, the final time for the simulation is $t = 10$ for all three cases. The total CPU time and wall time are recorded in Tab. \ref{table:couette_runtime}, from which we can see that the computational time is increasing with the increase of $M$. The computational time also shows the high efficiency of this Hermite spectral method, even for the quite rarefied cases.

\subsection{1D case: Fourier heat transfer}
\label{sec:fourier}
The second example of the 1D case is the Fourier heat transfer problem, which is also widely studied \cite{Shashank2019,ZhichengHu2019,Burnett}. In this example, we also consider the Ar-Kr mixture between two infinitely 
large plates. Similar to the Couette flow, the distance between the walls is $10^{-3}$m, and both boundaries are assumed to be purely diffusive. Different from the Couette flow, the Fourier heat transfer problem considers a distinction of temperature instead of velocity on two walls. In this example, a different collision model, the VHS model, is adopted for the collision term, the detailed parameters of which are listed in Tab. \ref{table:coll_para}. For the Fourier heat transfer problem, the steady-state solution is studied and the DSMC method is utilized to obtain the reference solution. 

In the simulation, the same WENO reconstruction with $25$  uniform grid cells and ${\rm CFL} = 0.45$ is utilized as the Couette flow problem. Two different number densities and two different boundary conditions are tested. The detailed parameters for the initial and boundary conditions are listed in Tab. \ref{table:fourier}, where the parameters for the nondimensionalization are also listed. The same expansion centers $[\ou_{\rm Ar}, \oT_{\rm Ar}] = [\bz, 1]$ and $[\ou_{\rm Kr}, \oT_{\rm Kr}] = [\bz, 0.477]$ are adopted here for the convection step while the expansion centers of the collision step are decided with local macroscopic variables as in Algorithm \ref{algo:inhomo}. 

\begin{table}[!ht]
\centering
\def\arraystretch{1.5}
{\footnotesize
\begin{tabular}{lllll}
\hline
                                   & Case 1       & Case 2        & Case 3       & Case 4        \\ \hline
Initial conditions                 &              &               &              &               \\
Temperature $T$(K)                 & 273          & 273           & 273          & 273           \\
Velocity $\bv$ (m/s)               & (0, 0, 0)      & (0, 0, 0)       & (0 ,0, 0)      & (0, 0, 0)       \\
Number density $n_{\rm Ar}$(m$^{-3}$)        & 1.68e21      & 5.6e20        & 1.68e21      & 5.6e20        \\
Number density $n_{\rm Kr}$(m$^{-3}$)        & 8.009e20     & 2.667e20      & 8.009e20     & 2.667e20      \\ \hline
Characteristic variables           &              &               &              &               \\
Length $x_0$ (m)                   & 1e-3  & 1e-3  & 1e-3  & 1e-3 \\
Mass $m_0$ (kg)                    & 6.63e-27     & 6.63e-27      & 6.63e-27     & 6.63e-27      \\
Number density $n_0$(m$^{-3}$)     & 1.68e21      & 5.6e20        & 1.68e21      & 5.6e20        \\
Temperature $T$(K)                 & 273          & 273           & 273          & 273           \\
Velocity $u_0$(m/s)                & 238.377      & 238.377       & 238.377      & 238.377       \\
Knudsen number $(Kn_{11}$, $Kn_{12})$ & (0.77, 0.782) & (2.311, 2.346) & (0.77, 0.782) & (2.311, 2.346) \\
Knudsen number $(Kn_{21}$, $Kn_{22})$ & (0.54, 0.591) & (1.62, 1.774)  & (0.54, 0.591) & (1.62, 1.774)  \\ \hline
Boundary conditions           &              &               &              &               \\
Left wall temperature $T_l$(K)               & 223          & 223         & 109.2          & 109.2      \\
Right wall temperature $T_r$(K)               & 323          & 323         & 436.8          & 436.8         \\
 \hline
\end{tabular}
}
\caption{Fourier heat transfer in Sec. \ref{sec:fourier}: Running parameters for Fourier heat transfer.}
\label{table:fourier}
\end{table}

\begin{figure}[!htb]
    \centering
      \subfloat[Density, $\rho$ (m$^{-3}$)]
    {\includegraphics[width=0.45\textwidth, height=0.36\textwidth,
      clip]{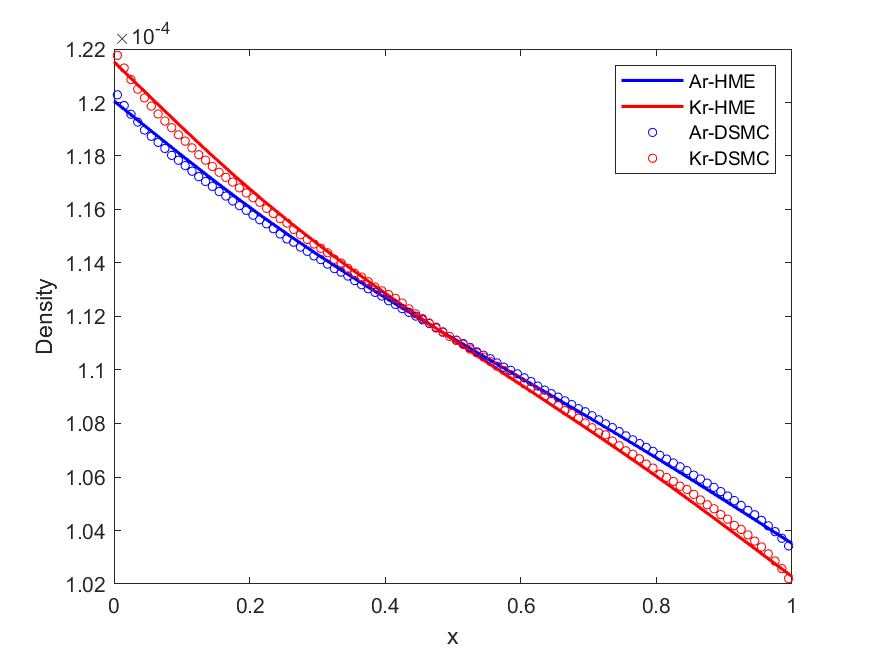}}\hfill
      \subfloat[Temperature, $T$(K)]
    {\includegraphics[width=0.45\textwidth, height=0.36\textwidth,
      clip]{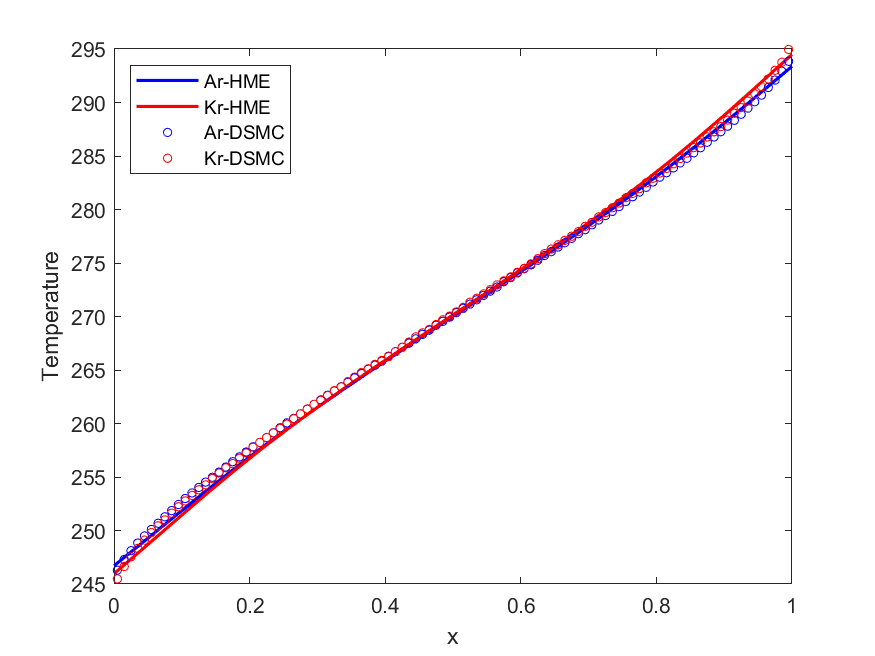}}\\
      \subfloat[Stress tensor, $\sigma_{12}$ (kg$\cdot$ m$^{-1}$ $\cdot$ s$^{-2}$)]
    {\includegraphics[width=0.45\textwidth, height=0.36\textwidth,
      clip]{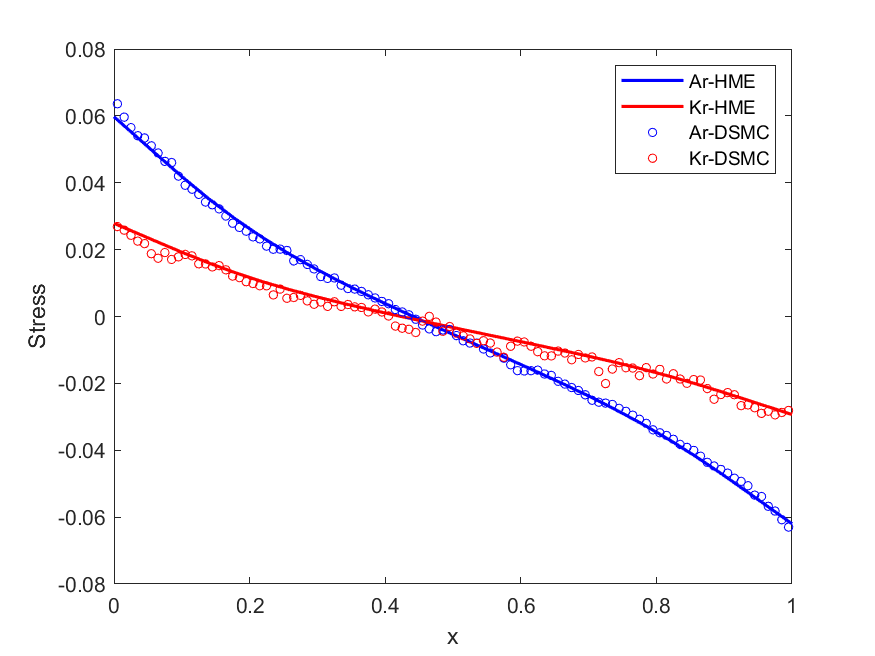}}\hfill
      \subfloat[Heat flux, $q_1$ (kg $\cdot$  s$^{-3}$) ]
    {\includegraphics[width=0.45\textwidth, height=0.36\textwidth,
      clip]{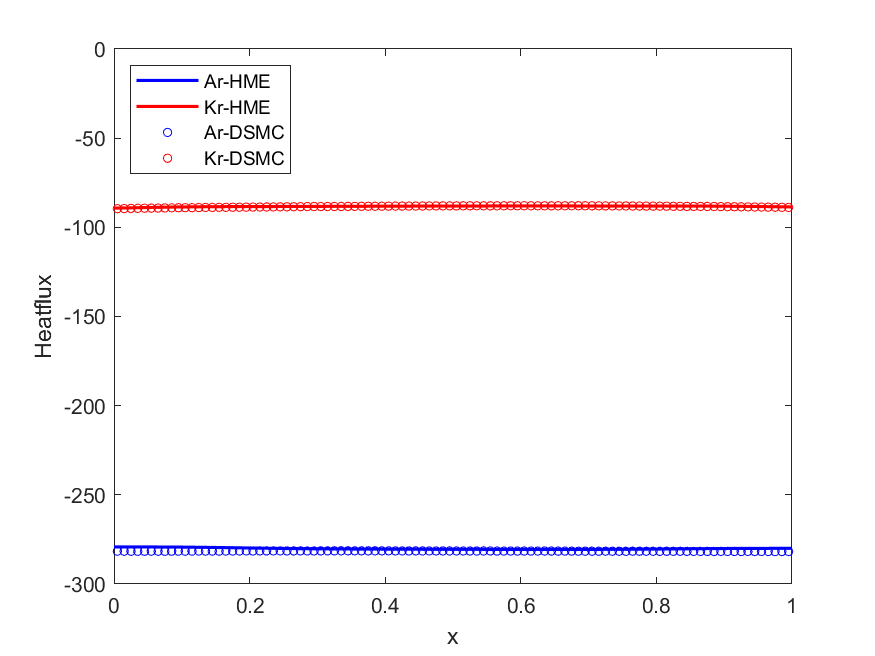}}\hfill
    \caption{Fourier heat transfer in Sec. \ref{sec:fourier}: Numerical results of Fourier heat transfer for case 1, where the Knudsen number is $(\Kn_{11}, \Kn_{12}) = (0.77, 0.782)$ and $(\Kn_{21}, \Kn_{22}) = (0.54, 0.591)$, and the temperature of the boundary condition is $(T_l, T_r) = (223, 323)$.}
    \label{fig:Fourier1}
\end{figure}

\begin{figure}[!htb]
    \centering
      \subfloat[Density, $\rho$ (m$^{-3}$)]
    {\includegraphics[width=0.45\textwidth, height=0.36\textwidth,
      clip]{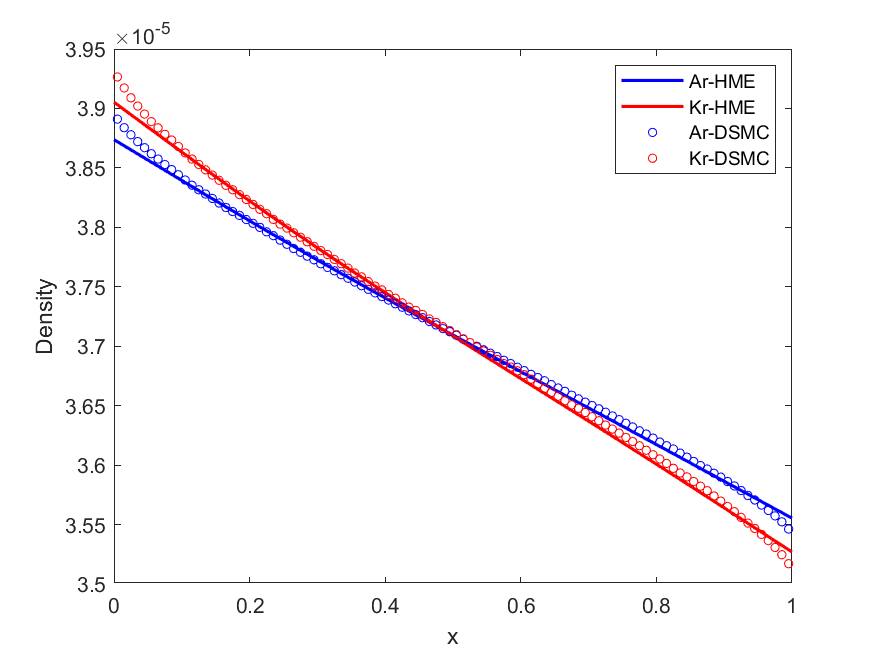}}\hfill
      \subfloat[Temperature, $T$(K)]
    {\includegraphics[width=0.45\textwidth, height=0.36\textwidth,
      clip]{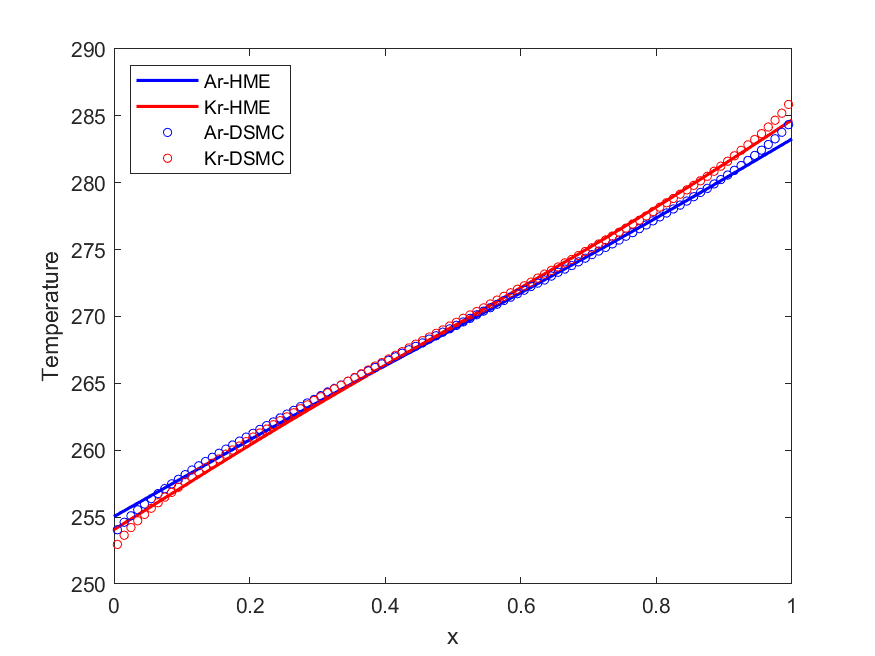}}\\
      \subfloat[Stress tensor, $\sigma_{12}$ (kg$\cdot$ m$^{-1}$ $\cdot$ s$^{-2}$)]
    {\includegraphics[width=0.45\textwidth, height=0.36\textwidth,
      clip]{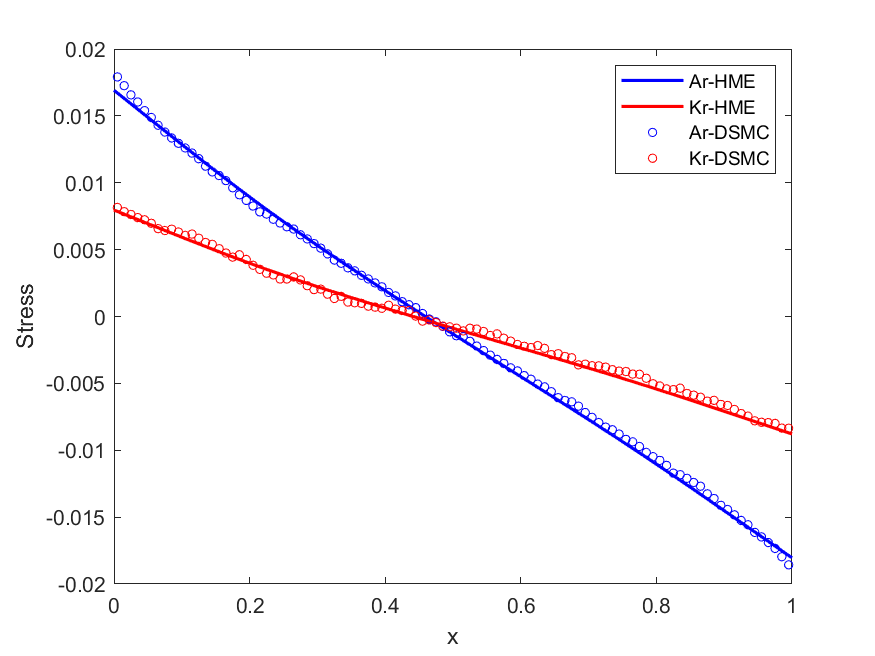}}\hfill
      \subfloat[Heat flux, $q_1$ (kg $\cdot$  s$^{-3}$) ]
    {\includegraphics[width=0.45\textwidth, height=0.36\textwidth,
      clip]{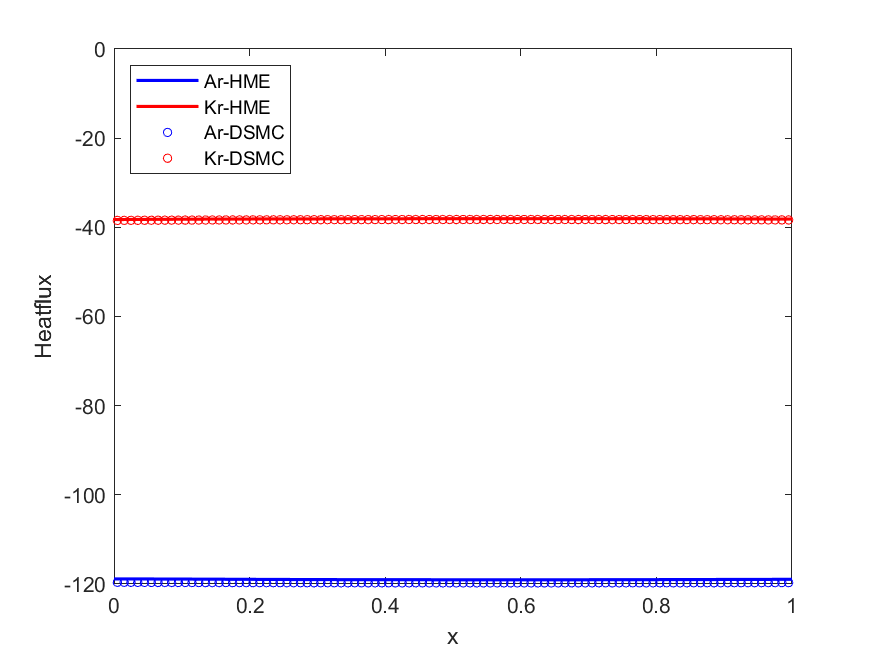}}\hfill
    \caption{Fourier heat transfer in Sec. \ref{sec:fourier}: Numerical results of Fourier heat transfer for case 2, where the Knudsen number is $(\Kn_{11}, \Kn_{12}) = (2.311, 2.346)$ and $(\Kn_{21}, \Kn_{22}) = (1.62, 1.774)$, and the temperature of the boundary condition is $(T_l, T_r) = (223, 323)$. }
    \label{fig:Fourier2}
\end{figure}    
  
\begin{figure}[!htb]
    \centering
      \subfloat[Density, $\rho$ (m$^{-3}$)]
    {\includegraphics[width=0.45\textwidth, height=0.36\textwidth,
      clip]{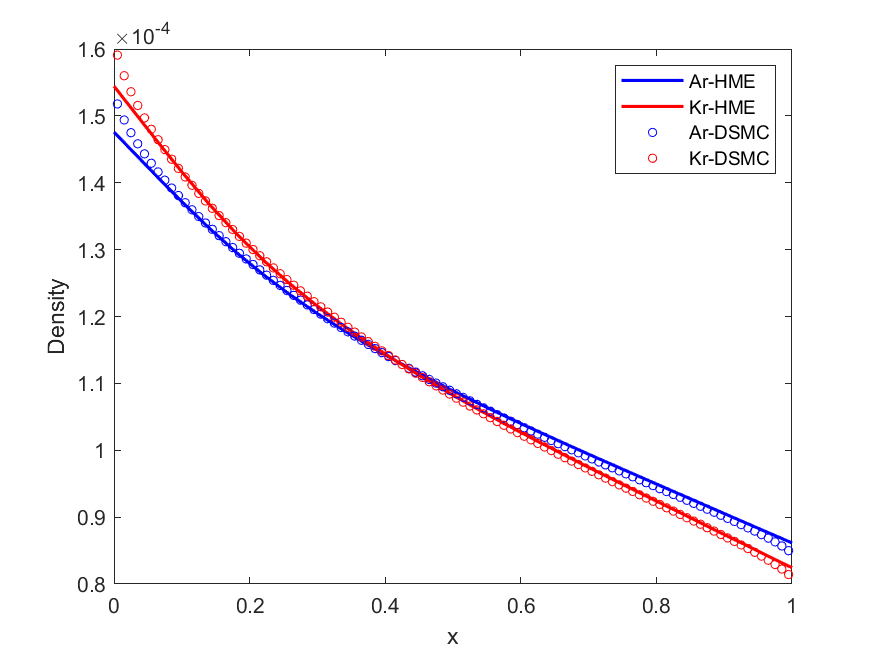}}\hfill
      \subfloat[Temperature, $T$(K)]
    {\includegraphics[width=0.45\textwidth, height=0.36\textwidth,
      clip]{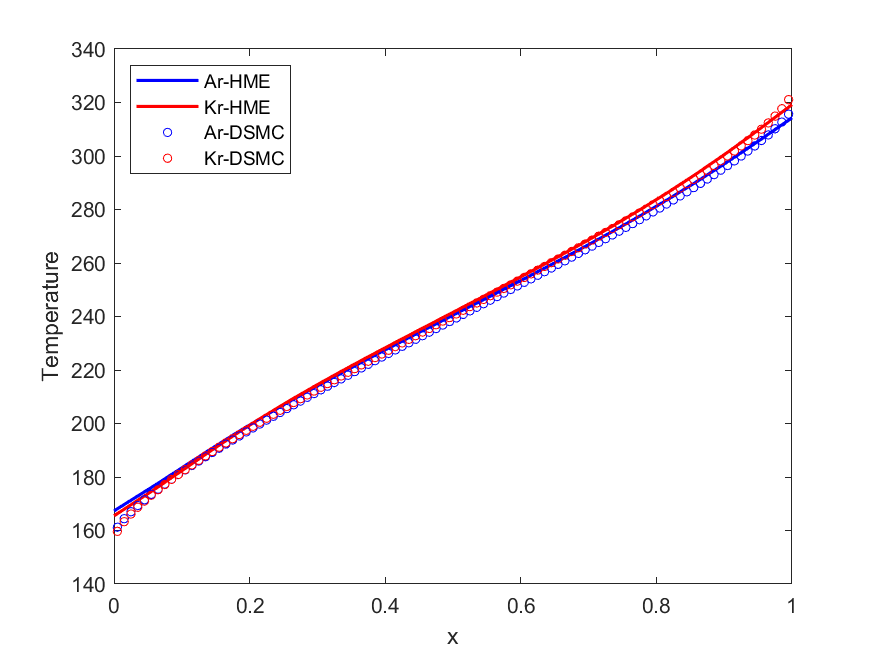}}\\
      \subfloat[Stress tensor, $\sigma_{12}$ (kg$\cdot$ m$^{-1}$ $\cdot$ s$^{-2}$)]
    {\includegraphics[width=0.45\textwidth, height=0.36\textwidth,
      clip]{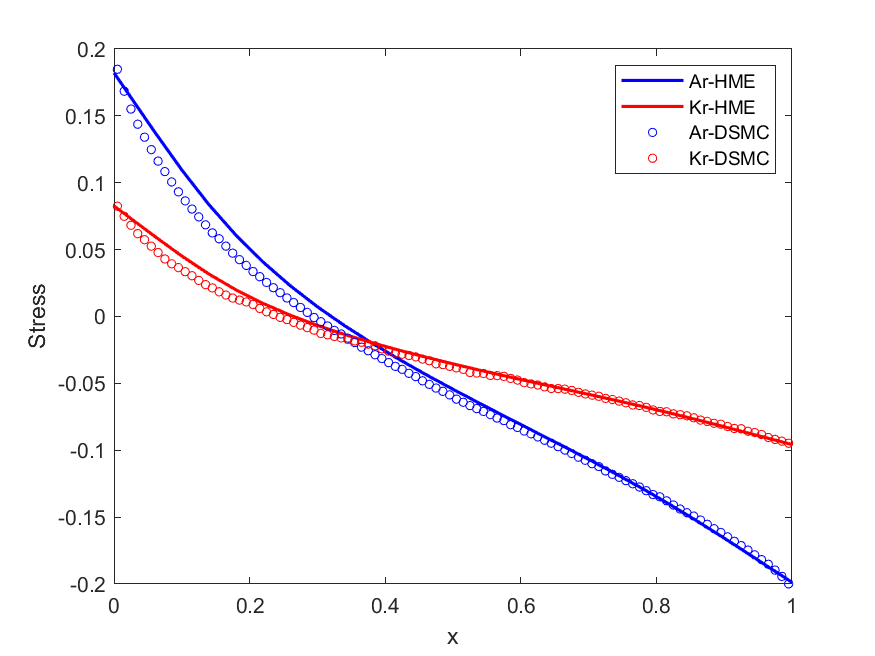}}\hfill
      \subfloat[Heat flux, $q_1$ (kg $\cdot$  s$^{-3}$) ]
    {\includegraphics[width=0.45\textwidth, height=0.36\textwidth,
      clip]{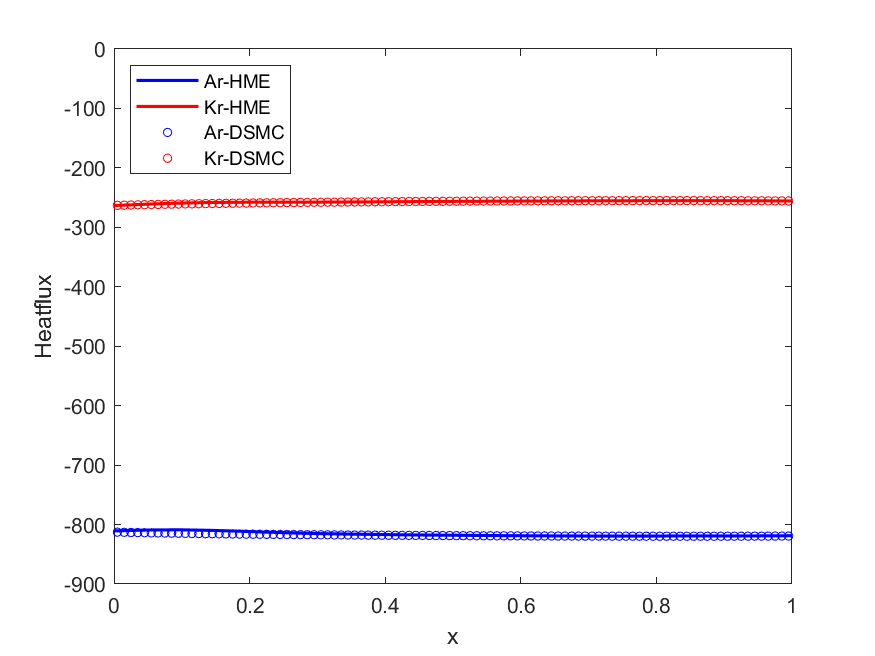}}\hfill
    \caption{Fourier heat transfer in Sec. \ref{sec:fourier}: Numerical results of Fourier heat transfer for case 2, where the Knudsen number is $(\Kn_{11}, \Kn_{12}) = (0.77, 0.782)$ and $(\Kn_{21}, \Kn_{22}) = (0.54, 0.591)$, and the temperature of the boundary condition is $(T_l, T_r) = (109.2, 436.8)$. }
    \label{fig:Fourier3}
\end{figure}

\begin{figure}[!htb]
    \centering
      \subfloat[Density, $\rho$ (m$^{-3}$)]
    {\includegraphics[width=0.45\textwidth, height=0.36\textwidth,
      clip]{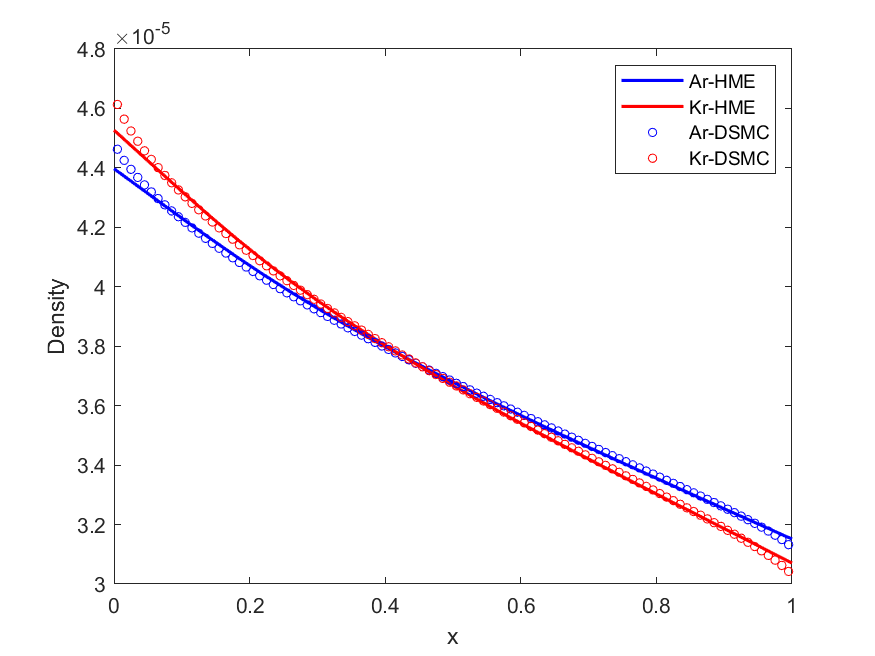}}\hfill
      \subfloat[Temperature, $T$(K)]
    {\includegraphics[width=0.45\textwidth, height=0.36\textwidth,
      clip]{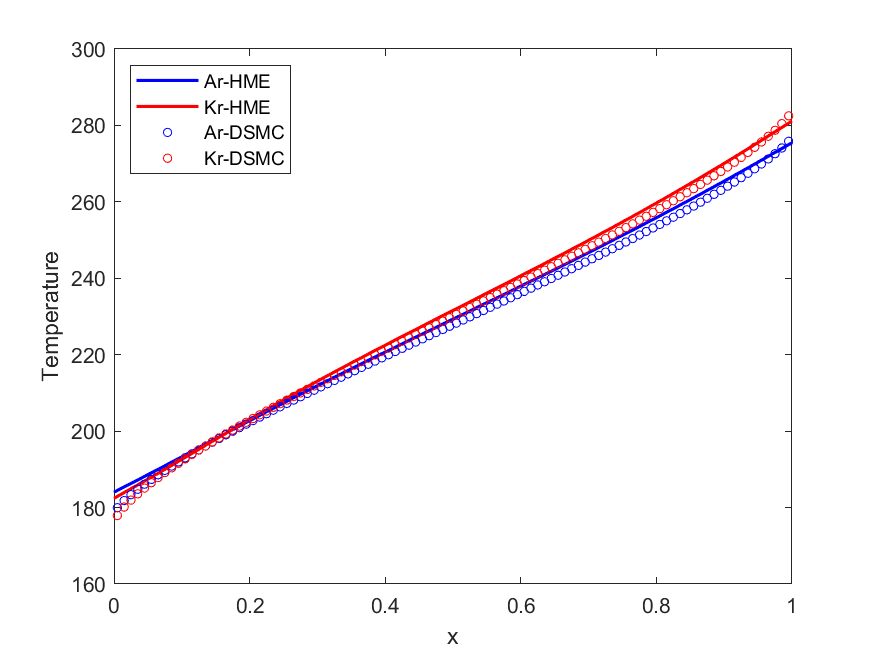}}\\
      \subfloat[Stress tensor, $\sigma_{12}$ (kg$\cdot$ m$^{-1}$ $\cdot$ s$^{-2}$)]
    {\includegraphics[width=0.45\textwidth, height=0.36\textwidth,
      clip]{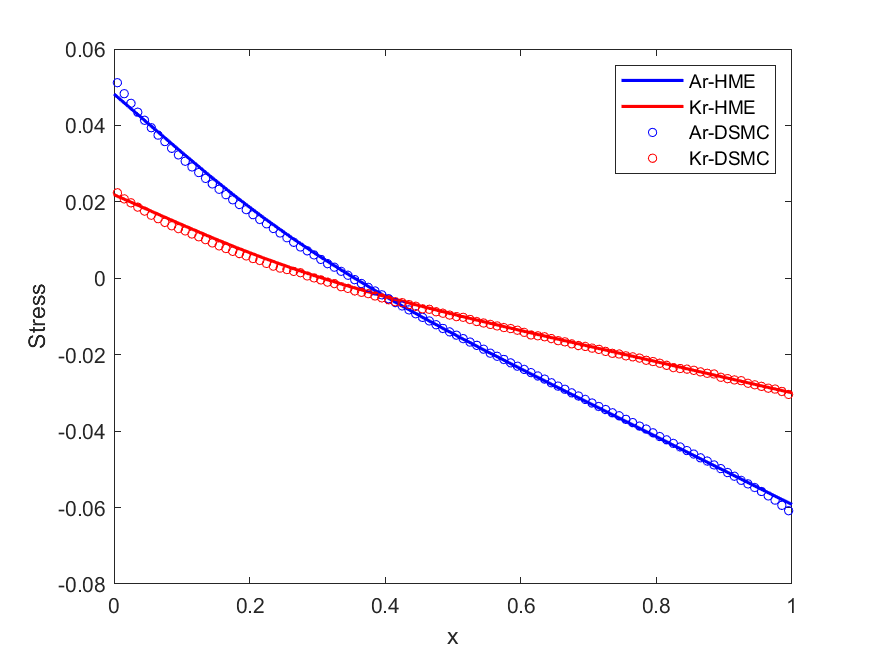}}\hfill
      \subfloat[Heat flux, $q_1$ (kg $\cdot$  s$^{-3}$) ]
    {\includegraphics[width=0.45\textwidth, height=0.36\textwidth,
      clip]{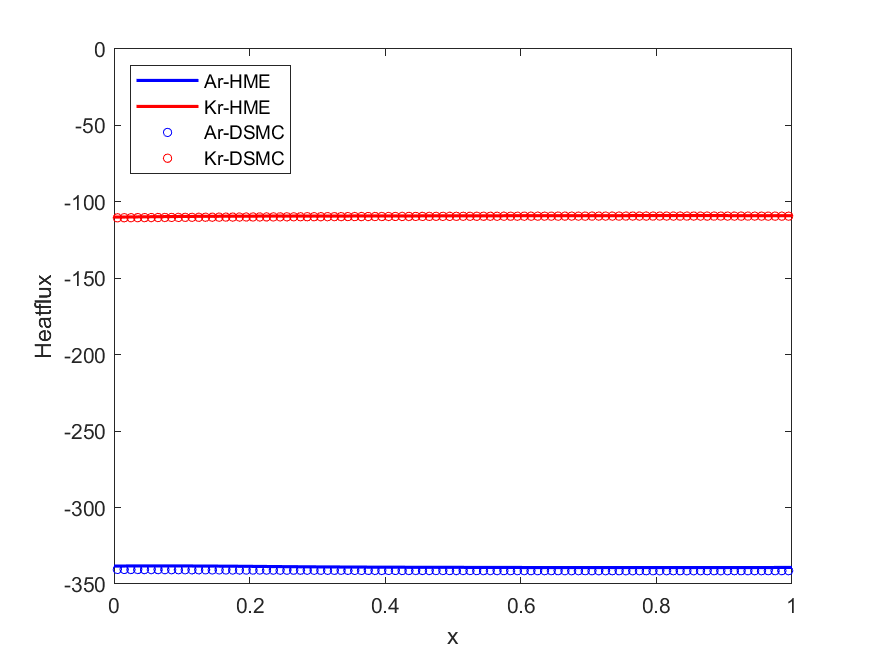}}\hfill
    \caption{
    Fourier heat transfer in Sec. \ref{sec:fourier}: Numerical results of Fourier heat transfer for case 4, where the Knudsen number is $(\Kn_{11}, \Kn_{12}) = (2.311, 2.346)$ and $(\Kn_{21}, \Kn_{22}) = (1.62, 1.774)$, and the temperature of the boundary condition is $(T_l, T_r) = (109.2, 436.8)$. }
    \label{fig:Fourier4}
\end{figure}

For the Fourier heat transfer problem, the density $\rho$, temperature $T$, stress tensor $\sigma_{12}$, and heat flux $q_1$ are studied. In the first and second cases in Tab. \ref{table:fourier}, since that difference between the two boundary conditions is small, which are $(T_l, T _r) = (223, 323)$, we set the order of the total moment number and the length of the quadratic collision term as $(M, M_0) = (40, 10)$. The numerical results are shown in Fig. \ref{fig:Fourier1} and \ref{fig:Fourier2}. In Fig. \ref{fig:Fourier1}, the numerical solutions match well with the reference solutions, and there are some oscillations in the reference solutions in the stress tensor, while those of the numerical solutions are still smooth. In Fig. \ref{fig:Fourier2}, for the temperature $T$, stress tensor $\sigma_{12}$, and heat flux $q_1$ all match well with the reference solution, although there is a little discrepancy for the density on the left boundary. However, the relative error is still quite small, which is about $1\%$. 

In the third and fourth cases in Tab. \ref{table:fourier}, the boundary condition is chosen as $(T_l, T_r) = (109.3, 436.8)$, where the ratio of the temperature is $1 : 4$. In case 3, since the Knudsen number is still small, we set $(M, M_0) = (40, 10)$, the results of which are shown in Fig. \ref{fig:Fourier3}. Most of the numerical solutions and the reference solutions are on top of each other, excepting the little difference in the stress tensor. In case 4, we increase $M$ to $60$ due to the large Knudsen number and keep $M_0 = 10$ due to the limitation of the computational cost. Fig. \ref{fig:Fourier4} shows that even for this large Knudsen number and the ratio of temperature on the boundary condition, we can still catch the behavior of the two species. 

The Fourier heat transfer problem is simulated using the same machine as the Couette flow problem. Here, the final computation time is still $ t = 10$, and the total CPU time and wall time are shown in Tab. \ref{table:fourier_time}. We can find that with the same $M$, the total run times for the first three cases are almost the same, while that of the fourth case is much longer, which may be due to the large Knudsen number and the high ratio of the temperature on the boundary conditions. 

\begin{table}[!ht]
\centering
\def\arraystretch{1.5}
{\footnotesize
\begin{tabular}{lllll}
\hline
                                   & Case 1       & Case 2        & Case 3       & Case 4        \\ \hline
Run-time data:  &  &  &  & \\
Total CPUtime $T_{\rm CPU}$(s): & 16384 & 17599 & 16978 & 43887 \\ 
Elapsed time(Wall time) $T_{\rm Wall}$(s): & 5238.95 & 5839.7 & 5366.48 & 18839.9\\
\hline
\end{tabular}
}
\caption{Fourier heat transfer in Sec. \ref{sec:fourier}: Run-time data for Fourier heat transfer.}
\label{table:fourier_time}
\end{table}


\subsection{2D case: Lid-driven cavity flow}
\label{sec:cavity}
In this section, the two-dimensional lid-driven cavity flow is studied, which is also tested in \cite{ZhichengHu2019, Cai2018, Liu2020}.
In this test, the Ar-Kr mixture is also taken as the working gas, and the HS collision kernel is utilized here in the collision model, the detailed parameters of which are listed in Tab. \ref{table:coll_para}. For the lid-driven cavity flow, the gas is confined in a square cavity with side length $L=10^{-3}m$. All walls are purely diffusive and have temperature $T=273K$. The top lid moves to the right with a constant speed $(50, 0, 0) m/s$, while the other three boundaries are stationary.

In this simulation, a uniform grid with $100 \times 100$ cells is utilized for the spatial discretization with the linear reconstruction adopted. The CFL condition is set as ${\rm CFL} = 0.3$. Two tests with different number densities are carried out, which are corresponding to different Knudsen numbers. The detailed parameters for the initial conditions are listed in Tab. \ref{table:cavity}, where the parameters for the non-dimensionalization are also listed. With \eqref{eq:Kn}, we can obtain the corresponding Knudsen number for the two cases as in Tab. \ref{table:cavity}. The same expansion centers $[\ou_{\rm Ar}, \oT_{\rm Ar}] = [\bz, 1]$ and $[\ou_{\rm Kr}, \oT_{\rm Kr}] = [\bz, 0.477]$  as the Couette flow and Fourier heat transfer problem are adopted here for the convection step. For the collision step, the expansion centers are also decided with local macroscopic variables as in Algorithm \ref{algo:inhomo}. 

\begin{table}[!ht]
\centering
\def\arraystretch{1.5}
{\footnotesize
\begin{tabular}{lll}
\hline
                                        & Case 1         & Case 2         \\ \hline
Initial conditions                      &                &                \\
Temperature $T$(K)                      & 273            & 273            \\
Velocity $\bv$ (m/s)                    & (0, 0, 0)        & (0, 0, 0)        \\
Number density $n_{\rm Ar}$(m$^{-3}$)             & 1.708e22       & 1.708e21       \\
Number density $n_{\rm Kr}$(m$^{-3}$)             & 8.141e21       & 8.141e20       \\ \hline
Characteristic variables                &                &                \\
Mass $m_0$ (kg)                         & 6.63e-27       & 6.63e-27       \\
Number density $n_0$(m$^{-3}$)          & 1.68e21        & 5.6e20         \\
Temperature $T$(K)                      & 273            & 273            \\
Velocity $u_0$(m/s)                     & 238.377        & 238.377        \\
Knudsen number ($Kn_{11}$, $Kn_{12}$)    & (0.1, 0.101)    & (1, 1.014)      \\
Knudsen number ($Kn_{21}$, $Kn_{22}$)      & (0.07, 0.076)   & (0.7, 0.762)    \\ \hline
\end{tabular}
}
\caption{Lid-driven cavity flow in Sec. \ref{sec:cavity}: Running parameters for the cavity flow.}
\label{table:cavity}
\end{table}

In the simulation, the temperature $T$ and the stress tensor $\sigma_{12}$ of the two species are studied. For case 1 in Tab. \ref{table:cavity}, since the Knudsen number is small, we set the total expansion order and the length of the quadratic collision term as $(M, M_0) = (25, 10)$. The numerical results are plotted in Fig. \ref{fig:Cavity1}, where we find that the numerical solution  matches well with the reference solution obtained by the DSMC method. Moreover, small oscillations exist in the reference solution of the temperature, while the numerical solution keeps smooth. For case 2 in Tab. \ref{table:cavity}, we set $(M, M_0)$ as $(M, M_0) = (30, 10)$, and the numerical results are shown in Fig. \ref{fig:Cavity2}. We find that for the case with increased Knudsen number, the numerical results and the reference solution are also almost the same, except that the reference solution has some oscillations and the numerical solutions are all smooth. 

\begin{figure}[!ht]
    \centering
      \subfloat[Temperature of Ar, $T$(K)]
    {\includegraphics[width=0.45\textwidth, height=0.36\textwidth,
      clip]{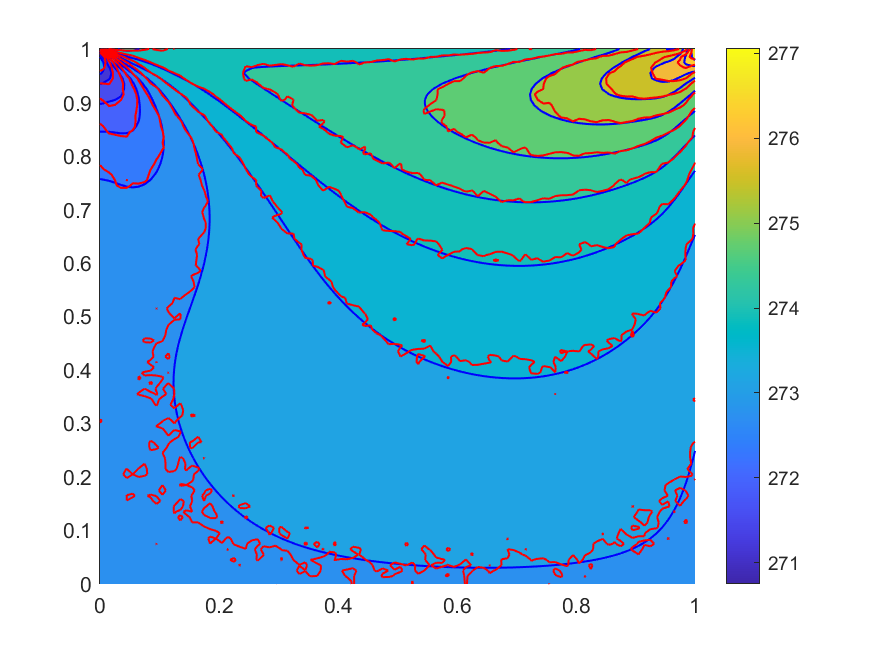}}\hfill
      \subfloat[Stress tensor of Ar, $\sigma_{12}$ (kg$\cdot$ m$^{-1}$ $\cdot$ s$^{-2}$)]
    {\includegraphics[width=0.45\textwidth, height=0.36\textwidth,
      clip]{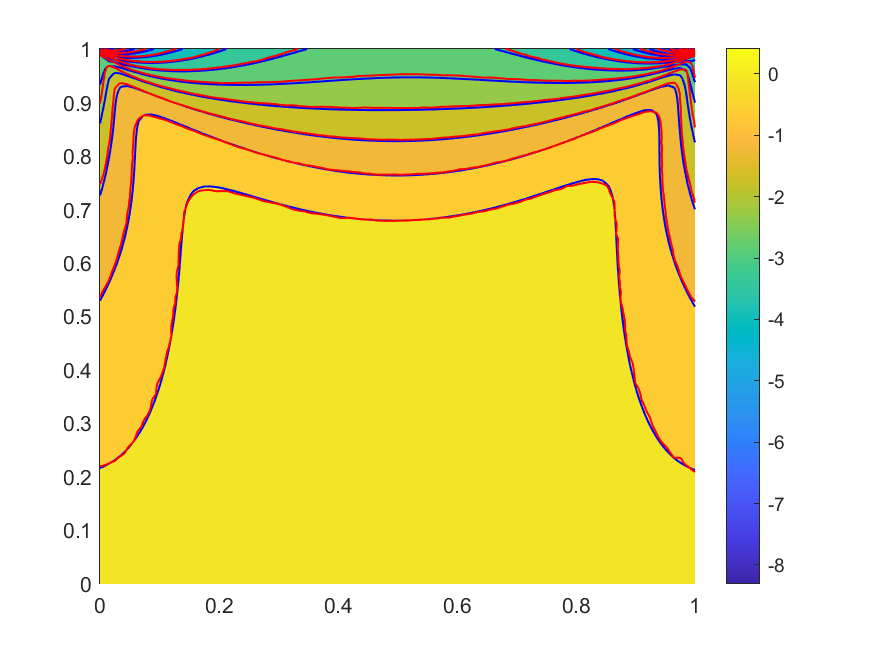}}\\
      \subfloat[Temperature of Kr, $T$(K)]
    {\includegraphics[width=0.45\textwidth, height=0.36\textwidth,
      clip]{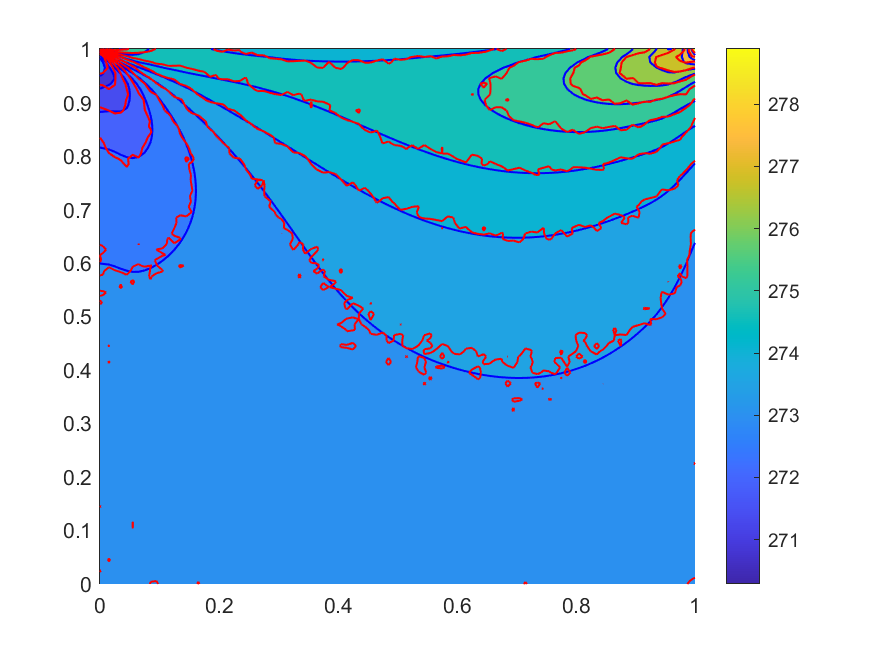}}\hfill
      \subfloat[Stress tensor of Kr, $\sigma_{12}$ (kg$\cdot$ m$^{-1}$ $\cdot$ s$^{-2}$)]
    {\includegraphics[width=0.45\textwidth, height=0.36\textwidth,
      clip]{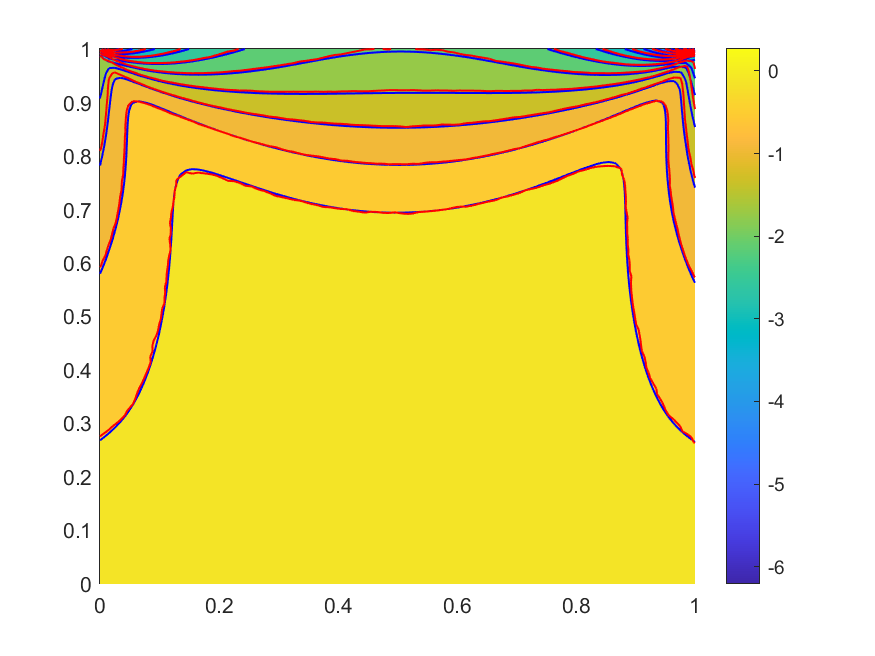}}\hfill
      \caption{Lid-driven cavity flow in Sec. \ref{sec:cavity}: Numerical results of the lid-driven cavity flow for case 1,  where the Knudsen number is $(\Kn_{11}, \Kn_{12}) = (0.1, 0.101)$ and $(\Kn_{21}, \Kn_{22}) = (0.07, 0.076)$.  Blue contours: Numerical solution. Red contours: Reference solution by DSMC.}
    \label{fig:Cavity1}
\end{figure}

\begin{figure}[!ht]
    \centering
      \subfloat[Temperature of Ar, $T$(K)]
    {\includegraphics[width=0.45\textwidth, height=0.36\textwidth,
      clip]{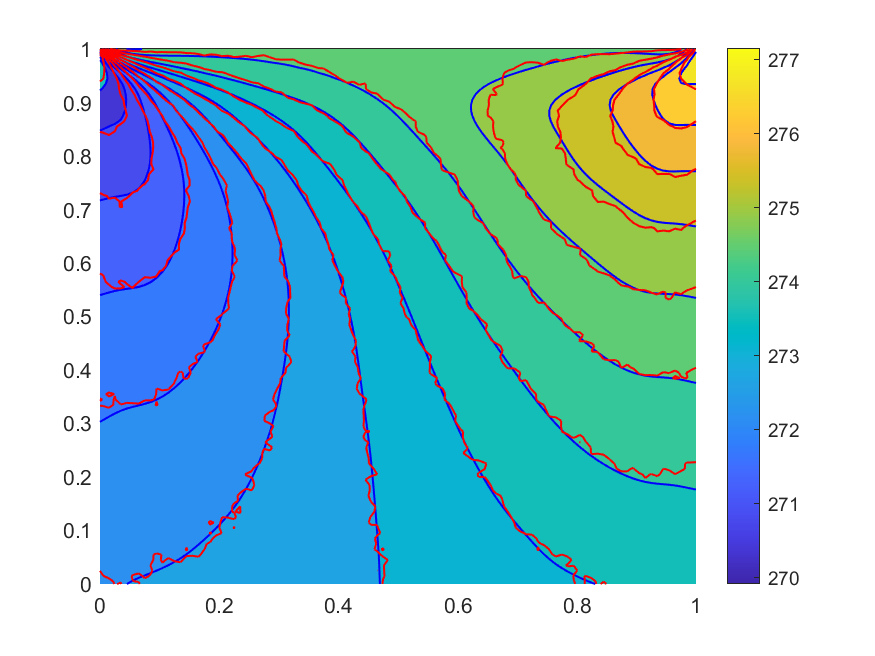}}\hfill
      \subfloat[Stress tensor of Ar, $\sigma_{12}$ (kg$\cdot$ m$^{-1}$ $\cdot$ s$^{-2}$)]
    {\includegraphics[width=0.45\textwidth, height=0.36\textwidth,
      clip]{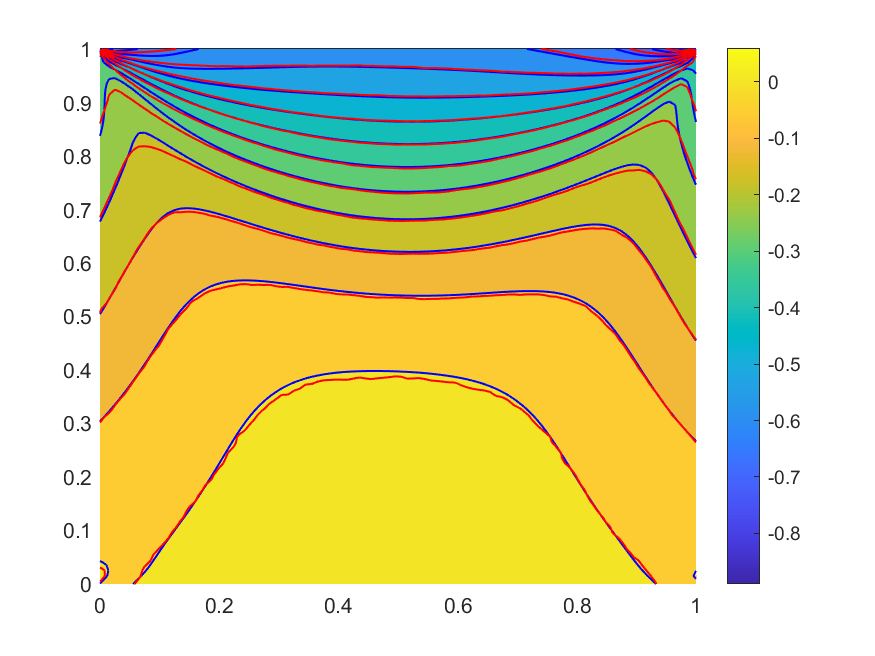}}\\
      \subfloat[Temperature of Kr, $T$(K)]
      {\includegraphics[width=0.45\textwidth, height=0.36\textwidth,
      clip]{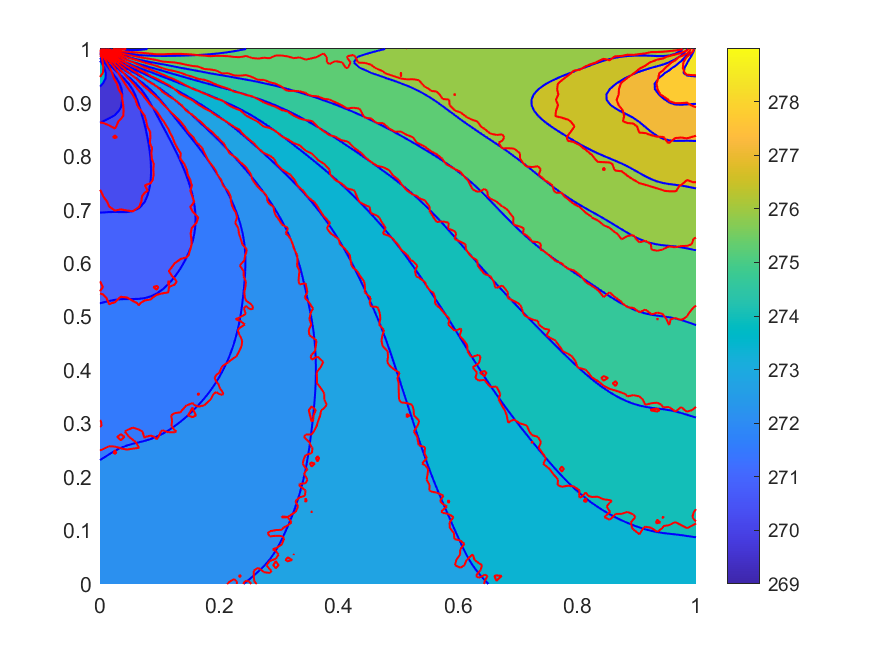}}\hfill
      \subfloat[Stress tensor of Kr, $\sigma_{12}$ (kg$\cdot$ m$^{-1}$ $\cdot$ s$^{-2}$)]
    {\includegraphics[width=0.45\textwidth, height=0.36\textwidth,
      clip]{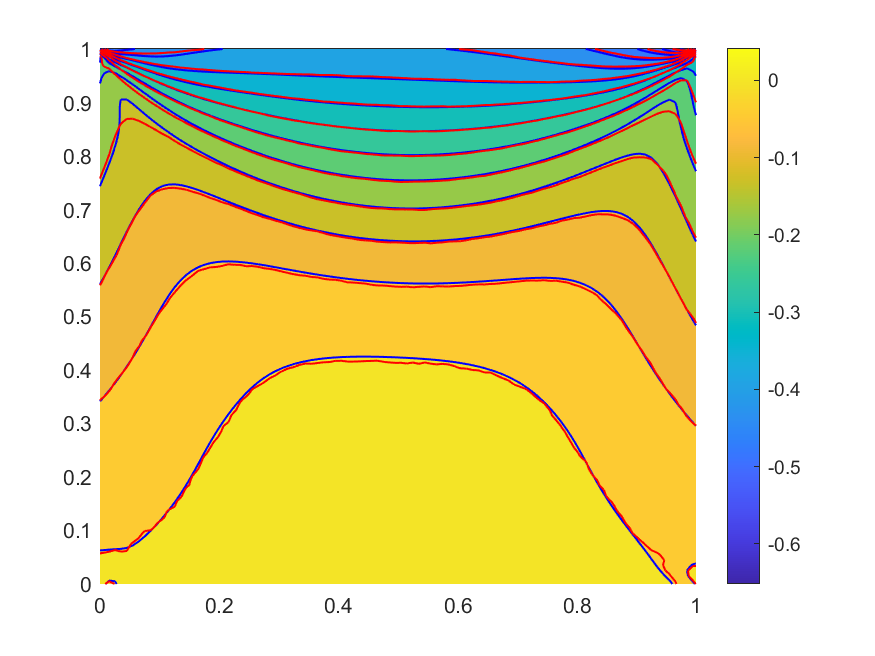}}\hfill
      \caption{Lid-driven cavity flow in Sec. \ref{sec:cavity}: Numerical results of the lid-driven cavity flow for case 2,  where the Knudsen number is $(\Kn_{11}, \Kn_{12}) = (1.0, 1.01)$ and $(\Kn_{21}, \Kn_{22}) = (0.7, 0.76)$.  Blue contours: Numerical solution. Red contours: Reference solution by DSMC. }
    \label{fig:Cavity2}
\end{figure}

\subsection{100-species KW solution}
\label{sec:100KW}
To verify the superiority of the numerical methods proposed in Algorithm \ref{algo:inhomo} in the cost of time and space, we again implement the KW solution test in \ref{subsec:2KW}. The difference is that we will consider the 100-species case, which is much more challenging in the required memory and running time. 
The numerical simulation with such a quantity of species is unreachable by most methods with acceptable consumption of time and space. The general $s$-species Boltzmann equation with constant collision kernel has the same form as in \eqref{eq:KW2}, and it also has the exact solution \cite{krook1977exact}, the form of which is similar to \eqref{eq:soluKW2}. Precisely, the exact solution has the form below 
\begin{equation}
    \label{eq:solu_KW}
    f^{(i)}(t,\bv)=n\UP[i]\left(\frac{m\UP[i]}{2\pi K}\right)^{\frac32}
    \exp\left(-\frac{m\UP[i]|\bv|^2}{2K}\right)
    \left[1-3Q+\frac{m\UP[i]}{K}Q|\bv|^2\right],
\end{equation}
where 
\begin{equation}
    \begin{gathered}
        A=\frac16, \qquad n\UP[i]=1, \qquad  m\UP[i] = i, \qquad  s= 100, \qquad 
        \mu_{ij}=\frac{4m\UP[i]m\UP[j]}{(m\UP[i])^2+(m\UP[j])^2},  \\
        \la_{ij}=\frac{1}{s\mu_{ij}(3-2\mu_{ij})}, \qquad    Q=\frac{1}{\exp(A(t+t_0))-2}, \qquad
        K=\frac{1}{1+2Q}.
    \end{gathered}
\end{equation}
Here $t_0$ is a constant which is large enough to ensure $f\UP[i]$ is positive for each $i$, and we set $t_0 = 20$. The same expansion center $[\ou\UP[i], \oT\UP[i]] = [\bz,1/{m\UP[i]}]$ is utilized here for each species as in Sec. \ref{subsec:2KW}. The corresponding coefficients can be computed as
\begin{equation}
    \label{eq:coefKW}
    \begin{split}
        f_{\al}\UP[i](t)
        =\left\{
            \begin{array}{ll}
                 \frac{(1-\frac{|\al|}2)(-2)^{\frac{|\al|}2}}{\al!}\prod_{j=1}^3\Big[(\al\UP[j]-1)!!\Big]\exp(-\frac{|\al|}2A(t+t_0)), &  \al_1,\al_2,\al_3\; \text{all even},\\
                 0, & \text{otherwise}.
            \end{array}
        \right.
    \end{split}
\end{equation}

In the simulation, the fourth-order Runge-Kutta scheme with the time step length $\Delta t = 0.01$ is utilized. The total expansion order $M$ and the length of the quadratic collision $M_0$ is set as $(M, M_0) = (20, 10)$. To show the efficiency of this new method, instead of showing the numerical solutions, 
we evaluate the relative $L^2$ error and weighted $L^2$ error 
\begin{equation}
    \label{eq:error}
    \begin{split}
    E\UP[i]&=\frac{\left[\int_{\mR^3}|f\UP[i]_{\text{num}}(\bv)-f\UP[i]_{\text{exact}}(\bv)|^2\rd\bv\right]^{\frac12}}
    {\left[\int_{\mR^3}|f\UP[i]_{\text{exact}}(\bv)|^2\rd\bv\right]^{\frac12}}, \\
    E\UP[i]_w&=\frac{\left[\int_{\mR^3}|f\UP[i]_{\text{num}}(\bv)-f\UP[i]_{\text{exact}}(\bv)|^2(\mM\UP[i](\bv))^{-1}\rd\bv\right]^{\frac12}}
    {\left[\int_{\mR^3}|f\UP[i]_{\text{exact}}(\bv)|^2(\mM\UP[i](\bv))^{-1}\rd\bv\right]^{\frac12}}
    \end{split}
\end{equation}
for each species. The integral will be evaluated with the 10-point Gauss-Hermite quadrature in each dimension. Two times $t = 1$ and $5$ are recorded, the detailed of which is shown in Tab. \ref{table:error}. We can see that the two errors are all quite small, which are almost at the order $10^{-8}$. The results clearly show the accuracy of this method. Moreover, with time evolution, the distribution function is approximating the Maxwellian, and that is why the error reduces with time increasing. 

\begin{table}[!htb]
  \centering
    \def\arraystretch{1.5}
    {\footnotesize
  \begin{tabular}{lll}
\hline
    $E\UP[i]$ & $t =1$ & $t = 5$ \\
\hline
    Average of 100 species: & $1.03 \times 10^{-8}$ & $1.12\times 10^{-9}$ \\
    Maximum: & $8.22\times 10^{-8}$ & $2.57\times 10^{-9}$\\
\hline
    $E\UP[i]_w$ & $t = 1$ & $t = 5$ \\
\hline 
    Average of 100 species: & $3.90 \times 10^{-8}$ & $3.58\times 10^{-9}$ \\
    Maximum: & $3.35\times 10^{-7}$ & $6.95\times10^{-9}$ \\
\hline
  \end{tabular}
  }
  \caption{100-species KW solution in Sec. \ref{sec:100KW}: Error estimations of the distribution function at time $t= 1$ and $5$.}\label{table:error}
\end{table}

The computational cost including the total computational time and the total memory used is shown in Tab. \ref{table:time_KW}. The simulation is run on the CPU model Intel Xeon E5-2697A V4 @ 2.6GHz with 8 threads. Tab. \ref{table:time_KW} shows that the computing time for each collision term is only $5.26 \times 10^{-4}$s, which is quite short for such a problem with $100$ species. Moreover, the total memory to store the expansion coefficients $\Aalk$,
which are represented in the double-precision floating-point format, is $18.85$GB, which is also quite easy to accomplish for computers nowadays. This is due to the special algorithm to calculate $\Aalk$, which is also the unique point in this new method. For Maxwell molecules (i.e., constant collision kernel $B_{ij}$), it has been proved in \cite{Approximation2019} that $\ga_{\ka}^{j}(1,1)$ can be nonzero only when $|\ka|=|j|$. Combining \eqref{eq:theorem_A} and \eqref{eq:tran_gamma},  we know that $\Aalk$ is nonzero only when $|\al|=|\la|+|\ka|$. Therefore, we only need to store the coefficients of $|\al|=|\la|+|\ka|$. The memory for each $A\alk(r)$ can be reduced to $\mO(M_0^8)$. Besides, since the set of expansion coefficients $\Aalk$ is only decided by the mass ratio $r=m\UP[j]/m\UP[i]$, 
we only need to store one group of coefficients $\Aalk$ for duplicate $r$, which will further reduce the storage cost. Furthermore, the new collision model combined by the quadratic collision term and the BGK collision term will reduce the computational cost and greatly improve the efficiency of this new method. All of these special properties make the simulation of this 100-species KW solution problem possible, while similar simulations are never seen before.

\begin{table}[!ht]
  \centering
    \def\arraystretch{1.5}
    {\footnotesize
  \begin{tabular}{lll}
    \hline
     &  $ t= 1$ &  $ t = 5$ \\  \hline 
    Total CPU time(s): & 2102 & 10642\\
    Collision terms computed: & $4\times10^6$ & $2\times10^7$ \\
    CPU time per collision term(s): & $5.26\times10^{-4}$ & $5.32\times10^{-4}$ \\
    Elapsed time(Wall time)(s): & 302.02 & 1759.32 \\
    Total memory of $A\alk$ (Gigabytes): & 18.85 & 18.85\\ 
\hline
  \end{tabular}
  }
  \caption{100-species KW solution in Sec. \ref{sec:100KW}:
  Running time \& memory at time $t= 1$ and $5$. }\label{table:time_KW}
\end{table}

\clearpage
\section{Conclusion}
\label{sec:conclusion}
In this paper, we have developed a numerical scheme of multi-species Boltzmann equation based on the Hermite spectral method. A new collision model is built by combing the quadratic collision term and the BGK collision model to balance the computational cost and the accuracy. The expansion coefficients of the quadratic collision model are almost computed with the properties of the Hermite basis functions. Several numerical examples are implemented to validate this new numerical method. Even for problems with 100-species, this new method could capture the behavior of the particles well with a low computational cost. 

Such numerical results make this new numerical method promising when applied to high-dimensional and more complicated problems. Moreover, the GPU may be adopted to further speed up and the research on the quantum Boltzmann equations are ongoing.


\section*{Acknowledgements}
We thank  Prof. Zhenning Cai from NUS for his valuable suggestions. The work of Yanli Wang is partially supported by the National Natural Science Foundation of China (Grant No. 12171026, U1930402 and 12031013).
\section{Appendix}
\label{sec:app}
\subsection{Proof of Lemma \ref{theorem:step1}}
\label{app:proof}
Now we provide the proof of Lemma \ref{theorem:step1}. For convenience, we will omit the superscript $[\bz,\bT]$ of Hermite polynomials in the proof.
\begin{proof}[proof of Lemma \ref{lemma:tran-Hermite}] 

First, it is easy to verify that $|\bv|^2+|\bw|^2=(1+r)|\bh|^2+
\frac{r}{1+r}|\bg|^2$, $\rd\bv \rd\bw$=$\sr \rd\bg \rd\bh$ and
\begin{equation}
    \label{eq:pd}
    \begin{split}
    \pd{}{\bv}=\frac{1}{1+r}\pd{}{\bh}+\pd{}{\bg}, \qquad 
    \pd{}{\bw}=\frac{\sr}{1+r}\pd{}{\bh}-\frac{1}{\sr}\pd{}{\bg}.
    \end{split}
\end{equation}
By the Leibniz rule, we have the following relation:
\begin{equation}
    \label{rela:deri}
    \pd{^{l_d+k_d}}{v_d^{l_d}w_d^{k_d}}=
    \sum_{i_d=0}^{l_d}\sum_{j_d=0}^{k_d}C_{l_d}^{i_d}C_{k_d}^{j_d}
    \frac{(-1)^{k_d-j_d}r^{j_d-\frac{k_d}{2}}}{(1+r)^{i_d+j_d}}
    \pd{^{l_d}}{h_d^{i_d+j_d}}\pd{^{k_d}}{g_d^{i'_d+j'_d}}, \quad d=1,2,3.
\end{equation}
where $i'_d=l_d-i_d,j'_d=k_d-j_d$. Denote
{\small 
\begin{equation}
    \label{eq:def:I}
    \begin{split}
    I_{\la',\ka'}^{\la,\ka}(r) \triangleq 
    \int_{\mR^3}\int_{\mR^3}&\Hl(\bv)\Hk(\bw)
    H_{\la'}(\sqrt{1+r}\bh)H_{\ka'}\left(\sqrt{\frac{r}{1+r}}\bg\right) \mM(\sqrt{1+r}\bh)\mM\left(\sqrt{\frac{r}{1+r}}\bg\right)\rd\bg \rd\bh.
    \end{split}
\end{equation}
}
Based on the orthogonality of Hermite polynomials, we only need to 
prove that
\begin{equation}
    \label{eq:valueI}
    I_{\la',\ka'}^{\la,\ka}(r)=\left\{
    \begin{array}{ll}
    \frac{1}{r^{\frac32}}\la'!\ka'!\ca1(r)\ca2(r)\ca3(r), &\text{if } 
    \la+\ka=\la'+\ka',\\
    0, & \text{otherwise}.
    \end{array}
    \right.
\end{equation}
With $|\bv|^2+|\bw|^2=(1+r)|\bh|^2+\frac{r}{1+r}|\bg|^2$, we can 
simplify \eqref{eq:def:I} as
{\small
\begin{equation}
    \label{eq:deri-I}
    \begin{split}
    I_{\la',\ka'}^{\la,\ka}(r)=&
    \int_{\mR^3}\int_{\mR^3}\Hl(\bv)\Hk(\bw)
    H_{\la'}(\sqrt{1+r}\bh)H_{\ka'}\left(\sqrt{\frac{r}{1+r}}\bg\right)
    \mM(\bv)\mM(\bw)\rd\bg \rd\bh \\
    =&{\oT^{\frac{|\la|+|\ka|}{2}}}
    \int_{\mR^3}\int_{\mR^3}
    (-1)^{|\la|}\pd{^{|\la|}}{v_1^{l_1}v_2^{l_2}v_3^{l_3}}\mM(\bv)
    (-1)^{|\ka|}\pd{^{|\ka|}}{w_1^{k_1}w_2^{k_2}w_3^{k_3}}\mM(\bw) \\
    &H_{\la'}(\sqrt{1+r}\bh)H_{\ka'}\left(\sqrt{\frac{r}{1+r}}\bg\right)
    \rd\bg \rd\bh \\
    =&{\oT^{\frac{|\la|+|\ka|}{2}}}
    \int_{\mR^3}\int_{\mR^3}\mM(\sqrt{1+r}\bh)\mM\left(\sqrt{\frac{r}{1+r}}\bg\right)\\
    & \prod_{d=1}^3\left[\sum_{i_d=0}^{l_d}\sum_{j_d=0}^{k_d}C_{l_d}^{i_d}
    C_{k_d}^{j_d}
    \frac{(-1)^{k_d-j_d}r^{j_d-\frac{k_d}{2}}}{(1+r)^{i_d+j_d}}
    \pd{^{l_d}}{h_d^{i_d+j_d}}\pd{^{k_d}}{g_d^{i'_d+j'_d}}\right]  H_{\la'}(\sqrt{1+r}\bh)H_{\ka'}\left(\sqrt{\frac{r}{1+r}}\bg\right)
    \rd\bg \rd\bh,
    \end{split}
\end{equation}
}where the second equality uses the definition \eqref{eq:Hermite} and the 
third equality is obtained using integration by parts.
From the differentiation relation \eqref{eq:diff_Her}
\begin{equation}
    \label{eq:diff-rela}
    \pd{}{v_d}H_{\la}=\left\{
    \begin{array}{ll}
    0, &\text{if } l_d=0, \\
    \frac{l_d}{\sqrt{\oT}}H_{\la-e_d}, &\text{if } l_d>0,
    \end{array}
    \right.
\end{equation}
and the orthogonality \eqref{eq:orth}, it holds that (\ref{eq:deri-I}) is nonzero only when $i_d+j_d=l'_d$, 
$i'_d+j'_d=k'_d$, $d=1,2,3$,
which implies $\la+\ka=\la'+\ka'$. When $\la+\ka=\la'+\ka'$ holds, we can apply \eqref{eq:diff-rela} to
\eqref{eq:deri-I} and get
\begin{equation}
    \label{eq:deri2-I}
    \begin{split}
    \frac{I_{\la',\ka'}^{\la,\ka}(r)}{\la'!\ka'!}&=r^{-\frac32}
    \prod_{d=1}^3\left[\sum_{i_d=0}^{l_d}\sum_{j_d=0,l'_d=i_d+j_d}^{k_d}
    C_{l_d}^{i_d}C_{k_d}^{j_d}
    \frac{(-1)^{k_d-j_d}r^{j_d+\frac{k'_d-k_d}{2}}}{(1+r)^{\frac{l'_d+k'_d}
    {2}}}\right] =\frac{1}{r^{\frac32}}\ca1(r)\ca2(r)\ca3(r).
    \end{split}
\end{equation}
Thus \eqref{eq:valueI} holds, which completes the proof of the lemma.
\end{proof}

\subsection{The algorithm of $\gamma_{\ka}^{\bj}(1,1)$}
\label{app:gamma}
In this section, we will briefly introduce the algorithm to obtain $\gamma_{\ka}^{\bj}(1,1)$, and readers can be referred to  \cite[Theorem 2]{Approximation2019} for the detailed algorithm. As is shown in \cite{Approximation2019}, it holds that 
\begin{equation}
    \label{eq:algo_gamma}
    \gamma_{\ka}^{\bj}(1,1)=\sum_{\bm\in\bbN^3, \bm \pq \frac{\ka}{2}}
    \sum_{\bn\in\bbN^3, \bn \pq \frac{\bj}{2}}(2|\ka|-4|\bm|+1)D_{\bm}^{\ka}D_{\bn}^{\bj}
    S_{\ka-2\bm}^{\bj-2\bn}K_{\bm\bn}^{\ka\bj},
\end{equation}
where
\begin{equation}
    \label{eq:def_D}
    D_{\bm}^{\ka}=\frac{(-1)^{|\bm|}4\pi |\bm|!}{(2(\ka-\bm)+1)!!}
    \frac{\ka!}{\bm!},
\end{equation}
and $S_{\ka}^{\la}$ is the coefficient of $v_1^{\ka_1}v_2^{\ka_2}v_3^{\ka_3}w_1^{\la_1}w_2^{\la_2}w_3^{\la_3}$ in the polynomial
\begin{equation}
    \label{eq:def_R}
    R_{|\ka|}(\bv,\bw):=(|\bv||\bw|)^{|\ka|})P_{|\ka|}
    \left(\frac{\bv}{|\bv|}\cdot \frac{\bw}{|\bw|}\right),
\end{equation}
and
\begin{equation}
    \label{eq:def_K}
    \begin{split}
        K_{\bm\bn}^{\ka\bj}=&\int_0^{\infty}\int_{0}^{\pi}
    L_{|\bm|}^{\Big(|\ka|-2|\bm|+\frac12\Big)}\left(\frac{r^2}{4}\right)
    L_{|\bn|}^{\Big(|\la|-2|\bn|+\frac12\Big)}\left(\frac{r^2}{4}\right) \\
    &\times \left(\frac{r}{\sqrt2}\right)^{|\ka|+|\la|-2(|\bm|+|\bn|)+2}
    B_{ij}(r,\varsigma)\left[P_{|\ka|-2|\bm|}(\cos\chi)-1\right]
    \exp\left(-\frac{r^2}{4}\right)\rd\chi \rd r, 
    \end{split}
\end{equation}
where $L_n^{(\al)}(x)$ are the Laguerre polynomials and $P_n(x)$ are the 
Legendre polynomials. Moreover, $B_{ij}(r,\varsigma)$ is in fact decided by $(r,\chi)$ based on \eqref{eq:B}.
As for $S_{\ka}^{\la}$, we can obtain $R_{k}(\bv,\bw)$ with the 
recursion formula:
\begin{equation}
    \label{eq:recur_R}
    \begin{split}
        &R_0(\bv,\bw)=1,\quad R_1(\bv,\bw)=\bv\cdot \bw, \\
        &R_{k+1}(\bv,\bw)=\frac{2k+1}{k+1}(\bv\cdot\bw)R_k(\bv,\bw)-
        \frac{k}{k+1}(|\bv||\bw|)^2R_{k-1}(\bv,\bw),
    \end{split}
\end{equation}
then we can derive the expression of $S_{\ka}^{\la}$  based on \eqref{eq:recur_R}. 


\subsection{Properties of Hermite polynomials}
\label{app:Her}
For the Hermite polynomials \eqref{eq:Hermite}, several important properties are listed below
\begin{property}
(Orthogonality)
\end{property}
\begin{equation}
    \label{eq:orth}
    \int_{\mR^3}H\aut(\bv)H\but(\bv)\mMut(\bv)\rd\bv=\al_1!\al_2!\al_3!
    \delta_{\al,\be}.
\end{equation}
\begin{property}
(Transitivity)
\end{property}
\begin{equation}
    \label{eq:tran}
    H\aut(\bv)=H_{\al}^{0,\zeta}\left(\sqrt{\frac{{\zeta}}{{\;\bT\;}}}(\bv-\ou)\right).
\end{equation}
\begin{property}
\label{property:recur}
(Recursion)
\end{property}
\begin{equation}
    \label{eq:recur}
    \begin{split}
    &H^{\ou,\oT}_{\al+e_d}(\bv)=\frac{v_d-u_d}{\sqrt{\oT}}H^{\ou,\oT}
    _{\al}(\bv)-\al_dH^{\ou,\oT}_{\al-e_d}(\bv), \\
    \text{and} \quad& v_dH^{\ou,\oT}_{\al}=\sqrt{\oT}H^{\ou,\oT}_
    {\al+e_d}+
    u_dH^{\ou,\oT}_{\al}+\al_d\sqrt{\oT}H^{\ou,\oT}_{\al-e_d}.
    \end{split}
\end{equation}
\begin{property}(Differential of Hermite polynomial)
\label{property:deri}
\begin{equation}
    \label{eq:diff_Her}
    \frac{\pa}{\pa v_d}H\aut(\bv)=\frac{\al_d}{\sqrt{\oT}}H_{\al-e_d}^{\ou,\bT}
    (\bv).
\end{equation}
\end{property}
The property of transitivity can be directly derived from the definition \eqref{eq:Hermite}, and the proof of the other properties can be found in \cite{Abramowitz1964}.

\subsection{Solver configurations for spatially inhomogeneous cases}
\label{app:col_param}
In the numerical simulation, SPARTA \cite{Sparta} is employed to carry out the DSMC simulation as the reference. In spatially inhomogeneous cases, the mixture of Argon and Krypton will be taken as the working gas. Three collision models, the VSS, VHS, and HS collision kernels, are utilized in the simulation. The parameters of these three collision models between Argon and Krypton 
are shown in Tab. \ref{table:coll_para}, where the reference diameter is derived with  \cite[Eq. (4.62)]{Bird}.
\begin{table}[h!]
  \centering
  \centering
    \def\arraystretch{1.5}
    {\footnotesize
  \begin{tabular}{llll}
    \hline
    Collision model for Ar-Kr mixture                  & VSS           
    & VHS           & HS     \\
    \hline
Molecular mass: $m_1$ ($\times 10^{-26}$kg)          & 6.63          & 6.63 
         & 6.63   \\
Molecular mass: $m_2$ ($\times 10^{-26}$kg)          & 13.91         & 13.91
     & 13.91  \\
Ref. viscosity: $\mu_{ref,1}$ ($\times 10^{-5}$Pa s) & 2.117         & 2.117
       & 2.117  \\
Ref. viscosity: $\mu_{ref,2}$ ($\times 10^{-5}$Pa s) & 2.328         & 2.328
      & 2.328  \\
Viscosity index: ($\om_{11}, \om_{12}$)              & (0.81, 0.805) 
& (0.81, 0.805) & (0.5, 0.5) \\
Viscosity index: ($\om_{21}, \om_{22}$)              & (0.805, 0.8) 
 & (0.805, 0.8)  & (0.5, 0.5) \\
Scattering parameter: ($\al_{11}, \al_{12}$)         & (1.4, 1.36)   
& (1, 1)        & (1, 1) \\
Scattering parameter: ($\al_{21}, \al_{22}$)         & (1.36, 1.32)  
& (1, 1)        & (1, 1) \\
Ref. diameter: ($d_{ref,11}, d_{ref,12}$)($\times 10^{-10}m$) 
& (4.11, 4.405) & (4.17, 4.465) & (3.63, 3.895) \\
Ref. diameter: ($d_{ref,21}, d_{ref,22}$)($\times 10^{-10}m$) 
& (4.405, 4.7)  & (4.465, 4.76) & (3.895, 4.16) \\
Ref. temperature: ($T_{ref,11}, T_{ref,12}$)(K)      
& 273           & 273           & 273    \\
Ref. temperature: ($T_{ref,21}, T_{ref,22}$)(K)      
& 273           & 273           & 273    \\
\hline
  \end{tabular}
  }
  \caption{Parameters for the collision models.}\label{table:coll_para}
\end{table}


\addcontentsline{toc}{section}{References}
\bibliographystyle{plain}
\bibliography{article}


\end{document}